\documentclass[11pt,reqno]{article}
\usepackage[utf8]{inputenc}
\usepackage{geometry}
\usepackage{color}
\usepackage{amsmath}
\usepackage{amsthm}
\usepackage{amssymb}
\usepackage{graphicx}
\usepackage[all]{xy}
\PassOptionsToPackage{normalem}{ulem}
\usepackage{ulem}
\usepackage[unicode=true,pdfusetitle,
 bookmarks=true,bookmarksnumbered=false,bookmarksopen=false,
 breaklinks=false,pdfborder={0 0 1},backref=page,colorlinks=true]
 {hyperref}
\hypersetup{
 linkcolor=blue,citecolor=blue}

\makeatletter
\numberwithin{equation}{section}
\numberwithin{figure}{section}
\theoremstyle{plain}
\newtheorem{thm}{\protect\theoremname}[section]
\theoremstyle{plain}
\newtheorem{cor}[thm]{\protect\corollaryname}
\theoremstyle{definition}
\newtheorem{defn}[thm]{\protect\definitionname}
\theoremstyle{plain}
\newtheorem{lem}[thm]{\protect\lemmaname}
\theoremstyle{plain}
\newtheorem{prop}[thm]{\protect\propositionname}
\theoremstyle{remark}
\newtheorem{rem}[thm]{\protect\remarkname}

\usepackage{graphicx}
\usepackage{appendix}

\usepackage{enumitem}
\usepackage[backref=page]{hyperref}

\date{\today}

\usepackage{ytableau}

\makeatother

\providecommand{\corollaryname}{Corollary}
\providecommand{\definitionname}{Definition}
\providecommand{\lemmaname}{Lemma}
\providecommand{\propositionname}{Proposition}
\providecommand{\remarkname}{Remark}
\providecommand{\theoremname}{Theorem}

\begin{document}
\global\long\def\F{\mathrm{\mathbf{F}} }%
\global\long\def\Aut{\mathrm{Aut}}%
\global\long\def\C{\mathbf{C}}%
\global\long\def\H{\mathcal{H}}%
 
\global\long\def\U{\mathcal{U}}%

\global\long\def\ext{\mathrm{ext}}%
 
\global\long\def\hull{\mathrm{hull}}%
 
\global\long\def\triv{\mathrm{triv}}%
 
\global\long\def\Hom{\mathrm{Hom}}%

\global\long\def\trace{\mathrm{tr}}%
\global\long\def\End{\mathrm{End}}%
 
\global\long\def\tsg{\widetilde{\Sigma_{g}}}%

\global\long\def\L{\mathcal{L}}%
\global\long\def\W{\mathcal{W}}%
\global\long\def\E{\mathbb{E}}%
\global\long\def\SL{\mathrm{SL}}%
\global\long\def\R{\mathbf{R}}%
\global\long\def\Pairs{\mathrm{PowerPairs}}%
\global\long\def\Z{\mathbf{Z}}%
\global\long\def\rs{\to}%
\global\long\def\A{\mathcal{A}}%
\global\long\def\a{\mathbf{a}}%
\global\long\def\rsa{\rightsquigarrow}%

\global\long\def\D{\mathcal{D}}%
\global\long\def\b{\mathbf{b}}%
\global\long\def\df{\mathrm{def}}%
\global\long\def\eqdf{\stackrel{\df}{=}}%
\global\long\def\ZZ{\overline{Z}}%
\global\long\def\Tr{\mathrm{Tr}}%
\global\long\def\N{\mathbf{N}}%
\global\long\def\std{\mathrm{std}}%
 
\global\long\def\HS{\mathrm{H.S.}}%
\global\long\def\e{\mathbf{e}}%
\global\long\def\c{\mathbf{c}}%
\global\long\def\d{\mathbf{d}}%
\global\long\def\AA{\mathbf{A}}%
\global\long\def\BB{\mathbf{B}}%
 
\global\long\def\v{\mathbf{v}}%
\global\long\def\spec{\mathrm{spec}}%
\global\long\def\Ind{\mathrm{Ind}}%
\global\long\def\half{\frac{1}{2}}%
\global\long\def\Re{\mathrm{Re}}%
 
\global\long\def\Im{\mathrm{Im}}%
\global\long\def\Rect{\mathrm{Rect}}%
\global\long\def\Crit{\mathrm{Crit}}%
\global\long\def\Stab{\mathrm{Stab}}%
\global\long\def\SL{\mathrm{SL}}%
\global\long\def\Tab{\mathrm{Tab}}%
\global\long\def\Cont{\mathrm{Cont}}%
\global\long\def\I{\mathcal{I}}%
\global\long\def\J{\mathcal{J}}%
\global\long\def\short{\mathrm{short}}%
\global\long\def\Id{\mathrm{Id}}%
\global\long\def\B{\mathcal{B}}%
\global\long\def\ax{\mathrm{ax}}%
\global\long\def\cox{\mathrm{cox}}%
\global\long\def\row{\mathrm{row}}%
\global\long\def\col{\mathrm{col}}%
\global\long\def\X{\mathbb{X}}%
\global\long\def\Fat{\mathsf{Fat}}%

\global\long\def\V{\mathcal{V}}%
\global\long\def\P{\mathbb{P}}%
\global\long\def\Fill{\mathsf{Fill}}%
\global\long\def\fix{\mathsf{fix}}%
 
\global\long\def\reg{\mathrm{reg}}%
\global\long\def\edge{E}%
\global\long\def\id{\mathrm{id}}%
\global\long\def\emb{\mathrm{emb}}%

\global\long\def\Hom{\mathrm{Hom}}%
 
\global\long\def\F{\mathrm{\mathbf{F}} }%
  
\global\long\def\pr{\mathrm{Prob} }%
 
\global\long\def\tr{{\cal T}r }%
\global\long\def\core{\mathrm{Core}}%
\global\long\def\pcore{\mathrm{PCore}}%
\global\long\def\im{\vartheta}%
\global\long\def\br{\mathsf{BR}}%
 
\global\long\def\sbr{\mathsf{SBR}}%
 
\global\long\def\ebs{\mathsf{EBS}}%
 
\global\long\def\ev{\mathrm{ev}}%
 
\global\long\def\CC{\mathcal{C}}%
 
\global\long\def\sides{\mathrm{Sides}}%
\global\long\def\tp{\mathrm{top}}%
\global\long\def\lf{\mathrm{left}}%
\global\long\def\MCG{\mathrm{MCG}}%
\global\long\def\EE{\mathcal{E}}%
 
\global\long\def\mog{\mathfrak{MOG}}%
 
\global\long\def\fg{\le_{\mathrm{f.g.}}}%
 
\global\long\def\v{\mathfrak{v}}%
\global\long\def\e{\mathfrak{e}}%
\global\long\def\f{\mathfrak{f}}%
 
\global\long\def\d{\mathfrak{d}}%
\global\long\def\he{\mathfrak{he}}%

\global\long\def\defect{\mathrm{Defect}}%
\global\long\def\M{\mathcal{M}}%
\global\long\def\sdefect{\max\defect}%

\title{The Asymptotic Statistics of Random Covering Surfaces}
\author{Michael Magee and Doron Puder}
\maketitle
\begin{abstract}
Let $\Gamma_{g}$ be the fundamental group of a closed connected orientable
surface of genus $g\geq2$. We develop a new method for integrating
over the representation space $\X_{g,n}=\Hom(\Gamma_{g},S_{n})$ where
$S_{n}$ is the symmetric group of permutations of $\{1,\ldots,n\}$.
Equivalently, this is the space of all vertex-labeled, $n$-sheeted
covering spaces of the the closed surface of genus $g$.

Given $\phi\in\X_{g,n}$ and $\gamma\in\Gamma_{g}$, we let $\fix_{\gamma}(\phi)$
be the number of fixed points of the permutation $\phi(\gamma)$.
The function $\fix_{\gamma}$ is a special case of a natural family
of functions on $\X_{g,n}$ called Wilson loops. Our new methodology
leads to an asymptotic formula, as $n\to\infty$, for the expectation
of $\fix_{\gamma}$ with respect to the uniform probability measure
on $\X_{g,n}$, which is denoted by $\E_{g,n}[\fix_{\gamma}]$. We
prove that if $\gamma\in\Gamma_{g}$ is not the identity, and $q$
is maximal such that $\gamma$ is a $q$\textsuperscript{th} power
in $\Gamma_{g}$, then
\[
\E_{g,n}\left[\fix_{\gamma}\right]=d(q)+O(n^{-1})
\]
as $n\to\infty$, where $d\left(q\right)$ is the number of divisors
of $q$. Even the weaker corollary that $\E_{g,n}[\fix_{\gamma}]=o(n)$
as $n\to\infty$ is a new result of this paper. We also prove that
$\E_{g,n}[\fix_{\gamma}]$ can be approximated to any order $O(n^{-M})$
by a polynomial in $n^{-1}$.
\end{abstract}
\tableofcontents{}

\section{Introduction\label{sec:Introduction}}

Let $g\geq2$ and let $\Sigma_{g}$ be a closed orientable topological
surface of genus $g$. We fix a base point $o\in\Sigma_{g}$ and let
\begin{equation}
\Gamma_{g}\eqdf\pi_{1}\left(\Sigma_{g},o\right)\cong\left\langle a_{1},b_{1},\ldots,a_{g},b_{g}\,\middle|\,\left[a_{1},b_{1}\right]\cdots\left[a_{g},b_{g}\right]\right\rangle \label{eq:presentation of Gamma}
\end{equation}
be the fundamental group of $\Sigma_{g}$. Denote by 
\[
\X_{g,n}\eqdf\Hom\left(\Gamma_{g},S_{n}\right)
\]
the \emph{representation space} of all homomorphisms from $\Gamma_{g}$
to $S_{n}$, where $S_{n}$ is the symmetric group of permutations
of $\{1,\ldots,n\}$. From another point of view, the space $\X_{g,n}$
can be viewed as the space of degree-$n$ covering maps of $\Sigma_{g}$.
Indeed, for every not-necessarily-connected degree-$n$ covering map
\[
p\colon X\twoheadrightarrow\Sigma_{g},
\]
we may identify the fiber $p^{-1}\left(o\right)$ with $\left\{ 1,\ldots,n\right\} $,
and the monodromy action of $\pi_{1}\left(\Sigma_{g},o\right)$ on
the fiber then gives rise to a homomorphism $\phi\in\Hom(\Gamma_{g},S_{n})$.
This gives a one-to-one correspondence between $\mathbb{X}_{g,n}$
and degree-$n$ covering maps with $p^{-1}\left(o\right)=\left\{ 1,\ldots,n\right\} $.
This correspondence is discussed in more detail in $\S\S$\ref{subsec:Expectations-and-probabilities}.

The space $\X_{g,n}$ was studied by Liebeck and Shalev \cite{LiebeckShalev},
who showed that a uniformly random homomorphism $\phi\colon\Gamma_{g}\to S_{n}$
satisfies $\phi\left(\Gamma_{g}\right)\supseteq A_{n}$ a.a.s.~(asymptotically
almost surely, namely, with probability tending to $1$ as $n\to\infty$)
\cite[Thm 1.12]{LiebeckShalev}\footnote{The paper \cite{LiebeckShalev} considers, more generally, random
homomorphisms from any Fuchsian group to $S_{n}$.}. In particular the image is a.a.s.~transitive, or, equivalently,
the corresponding random degree-n covering space is a.a.s.~connected.
When $\Gamma_{g}$ is replaced by a non-abelian free group, the analogous
result holds by the famous theorem of Dixon \cite{dixon1969probability}
that two random permutations in $S_{n}$ a.a.s.~generate $S_{n}$
or $A_{n}$.

In the current work we address the problem of integration over the
space $\X_{g,n}$. Namely, our goal is to analyze the expected value
$\mathbb{E}_{g,n}\left[f\right]$ of functions $f$ on $\X_{g,n}$
with respect to the uniform measure on $\X_{g,n}$. The functions
on $\X_{g,n}$ that we consider are natural functions that arise from
loops in $\Sigma_{g}$. Given an element $\gamma\in\Gamma_{g}$ and
a character $\chi$ of $S_{n}$, we let 
\[
\chi_{\gamma}(\phi)\eqdf\chi(\phi(\gamma)),\quad\chi_{\gamma}:\X_{g,n}\to\R.
\]
These functions are called \emph{Wilson loops }in the physics literature
\cite[Def. 6.4.1]{Labourie}. Our focus is on the character $\fix$
of $S_{n}$ which assigns to every permutation its number of fixed
points.

The main motivation behind this work is its relevance to the study
of random covers of the closed surface $\Sigma_{g}$. Given some $1\ne\gamma\in\Gamma_{g}$,
consider the geodesic $C_{\gamma}$ in $\Sigma_{g}$ corresponding
to the conjugacy class of $\gamma$. For every homomorphism $\phi\in\X_{g,n}$,
the number of fixed points $\fix_{\gamma}\left(\phi\right)$ is precisely
the number of lifts of $C_{\gamma}$ to a closed geodesic in the degree-$n$
covering corresponding to $\phi$. Indeed, the results of this paper
are crucial ingredients in a subsequent work \cite{magee2020random}
which gives new results on spectral gaps of random covers of a closed
surface. 

Another source of motivation is the rich theory that has been discovered
around similar questions when surface groups are replaced by free
groups (e.g.~\cite{nica1994number,PP15,MPunitary,hanany2020word}
and see $\S\S$\ref{subsec:Related-works-II:free groups} below).
Expanding this theory to other groups is challenging, as the presence
of a relation between the generators presents a fundamental difficulty
that is not present for free groups. Surface groups, among the best
understood and best behaved one-relator groups, are a natural starting
point for this quest. To overcome the difficulty brought up by the
existence of a relation, we develop in this work new machinery, both
in representation theory of $S_{n}$ and in combinatorial group theory. 

\subsubsection*{Expected number of fixed points}

Recall that the expectation $\mathbb{E}_{g,n}\left[\fix_{\gamma}\right]$
is the average number of fixed points in $\phi\left(\gamma\right)$
where $\phi\colon\Gamma_{g}\to S_{n}$ is uniformly random. Our main
results are the following two theorems.
\begin{thm}
\label{thm:rational-approx}Fix $g\geq2$ and $\gamma\in\Gamma_{g}$.
Then there is an infinite sequence of rational numbers 
\[
a_{1}\left(\gamma\right),a_{0}\left(\gamma\right),a_{-1}\left(\gamma\right),a_{-2}\left(\gamma\right),\ldots
\]
such that for any $M\in\N$, as $n\to\infty$,
\begin{equation}
\mathbb{E}_{g,n}\left[\fix_{\gamma}\right]=a_{1}\left(\gamma\right)n+a_{0}\left(\gamma\right)+\frac{a_{-1}\left(\gamma\right)}{n}+\ldots\frac{a_{-\left(M-1\right)}\left(\gamma\right)}{n^{M-1}}+O\left(\frac{1}{n^{M}}\right).\label{eq:Laurent-polynomial}
\end{equation}
\end{thm}

\begin{thm}
\noindent \label{thm:expected-fixed-points-bounded}If $\gamma\in\Gamma_{g}$
is not the identity then, as $n\to\infty$,
\[
\E_{g,n}[\fix_{\gamma}]=O(1).
\]
In fact, if $q\in\N$ is maximal such that $\gamma=\gamma_{0}^{~q}$
for some $\gamma_{0}\in\Gamma$, then, as $n\to\infty$,
\begin{align*}
\E_{g,n}[\fix_{\gamma}] & =d(q)+O\left(\frac{1}{n}\right),
\end{align*}
where $d(q)$ is the number of divisors function. In other words,
$a_{1}(\gamma)=0$ and $a_{0}(\gamma)=d(q)$.
\end{thm}

For example, consider the element $a$ in $\Gamma_{2}=\left\langle a,b,c,d\,\middle|\,\left[a,b\right]\left[c,d\right]\right\rangle $.
This element is not a proper power and so $\mathbb{E}_{2,n}\left[\fix_{a}\right]=1+O\left(n^{-1}\right)$
by Theorem \ref{thm:expected-fixed-points-bounded}. By Theorem \ref{thm:rational-approx},
this average can be approximated to any order $n^{-M}$ by a rational
function in $n$. In this particular case this rational function can
be computed to get for, e.g., $M=5$, 
\[
\mathbb{E}_{2,n}\left[\fix_{a}\right]=1+\frac{1}{n^{2}}+\frac{2}{n^{3}}+\frac{10}{n^{4}}+O\left(\frac{1}{n^{5}}\right).
\]

Given a finite group $G$, the number of homomorphisms $\Gamma_{g}\to G$
is related to the Witten zeta function of $G$, 
\[
\zeta^{G}\left(s\right)\eqdf\sum_{\chi\in\mathrm{Irr}G}\chi\left(1\right)^{-s},
\]
the summation being over the isomorphism classes of irreducible complex
representations of $G$. These functions were introduced by Zagier
\cite{Zagier} after Witten's work in \cite{Witten1991}. The connection
is given by 
\begin{equation}
\left|\Hom\left(\Gamma_{g},G\right)\right|=\left|G\right|^{2g-1}\zeta^{G}\left(2g-2\right).\label{eq:Hurwitz}
\end{equation}
This result goes back to Hurwitz \cite{hurwitz1902ueber}, who gave
a more general formula for arbitrary Fuchsian groups (a proof in English
is given in \cite[Prop. 3.2]{LiebeckShalev}). It is also sometimes
called `Mednykh’s formula’ in the literature after \cite{Mednyhk}.
For the case $G=S_{n}$, the zeta function $\zeta^{S_{n}}$ was studied
in \cite{Lulov96,muller2002character,LiebeckShalev,gamburd2006poisson}.
Inter alia, these works show that for every $s>0$,
\[
\zeta^{S_{n}}\left(s\right)\underset{n\to\infty}{\to}2.
\]
Moreover, their results yield an asymptotic expansion in $n$ which
approximates $\zeta^{S_{n}}\left(s\right)$ as $n\to\infty$, in a
similar manner to the one in Theorem \ref{thm:rational-approx}. As
such, their results can be thought of as the special case of $\gamma=1$
of a version of Theorem \ref{thm:rational-approx}. We elaborate more
in $\S\S$\ref{subsec:Overview-of-this}.

\subsubsection*{Common fixed points of subgroups}

Our proof also yields the following more general result that concerns
not only elements of $\Gamma_{g}$ but also f.g.~(finitely generated)
subgroups. We write $J\fg\Gamma_{g}$ to denote a f.g.~subgroup $J$
of $\Gamma_{g}$. Given $J\fg\Gamma_{g}$ and $\phi\in\mathbb{X}_{g,n}$,
we let $\fix_{J}\left(\phi\right)$ denote the number of elements
in $1,\ldots,n$ that are fixed by all permutations in $\phi\left(J\right)$:
\[
\fix_{J}\left(\phi\right)\eqdf\left|\left\{ i\in\left\{ 1,\ldots,n\right\} \,\middle|\,\sigma\left(i\right)=i~\mathrm{for~all}~\sigma\in\phi\left(J\right)\right\} \right|.
\]
In particular, $\fix_{\left\langle \gamma\right\rangle }=\fix_{\gamma}$
for all $\gamma\in\Gamma_{g}$. For $J\fg\Gamma_{g}$ we let 
\begin{equation}
\chi_{\max}\left(J\right)\eqdf\max\left\{ \chi\left(K\right)\,\middle|\,J\le K\fg\Gamma_{g}\right\} \label{eq:x_max}
\end{equation}
denote the largest Euler characteristic\footnote{Every f.g.~subgroup $K\le\Gamma_{g}$ is either a free group, in
which case $\chi\left(K\right)=1-\mathrm{rank}\left(K\right)$, or
a surface group of genus $h\ge g$, in which case $\chi\left(K\right)=2-2h$.} of a f.g.~subgroup $K\fg\Gamma_{g}$ which contains $J$. Note that
$\chi\left(\Gamma_{g}\right)=2-2g\le\chi_{\max}\left(J\right)\le1$
and that $\chi_{\max}\left(J\right)\ge\chi\left(J\right)$. It is
also true that $\chi_{\max}\left(J\right)=1$ if and only if $J=\left\{ 1\right\} $,
and $\chi_{\max}\left(J\right)\ge0$ if and only if $J$ is cyclic.
In addition, we let
\[
\mog\left(J\right)\eqdf\left\{ K\fg\Gamma_{g}\,\middle|\,J\le K~\mathrm{and}~\chi\left(K\right)=\chi_{\max}\left(J\right)\right\} 
\]
denote the set of ``maximal overgroups'' -- f.g.~subgroups achieving
the maximum from (\ref{eq:x_max}). This set is always finite --
see Corollary \ref{cor:MOG is finite}.
\begin{thm}
\label{thm:subgroups}Let $J\fg\Gamma_{g}$ be a finitely generated
subgroup. Then 
\[
\mathbb{E}_{g,n}\left[\fix_{J}\right]=\left|\mog\left(J\right)\right|\cdot n^{\chi_{\max}\left(J\right)}+O\left(n^{\chi_{\max}\left(J\right)-1}\right).
\]
\end{thm}

Theorem \ref{thm:subgroups} generalizes Theorem \ref{thm:expected-fixed-points-bounded},
as for $\gamma\ne1$, $\chi_{\max}\left(\left\langle \gamma\right\rangle \right)=0$
and 
\[
\mog\left(\left\langle \gamma\right\rangle \right)=\left\{ \,\left\langle \gamma_{0}^{~m}\right\rangle \big|~\,m|q\,\right\} .
\]
The analog of Theorem \ref{thm:rational-approx} holds too for f.g.~subgroups:
there is an infinite sequence of rational numbers
\[
a_{1}\left(J\right),a_{0}\left(J\right),a_{-1}\left(J\right),\ldots
\]
such that for any $M\in\N$, as $n\to\infty$, 
\[
\mathbb{E}_{g,n}\left[\fix_{J}\right]=\sum_{i=-(M-1)}^{1}a_{i}\left(J\right)n^{i}+O\left(n^{-M}\right),
\]
and such that $a_{1}=a_{0}=\ldots=a_{\chi_{\max}\left(J\right)+1}=0$
and $a_{\chi_{\max}\left(J\right)}=\left|\mog\left(J\right)\right|$.\\

\subsection{Related works I: Mirzakhani's integral formulas}

In \cite{Mirzakhani}, Mirzakhani considered a similar problem to
the one in this paper. Instead of integrating over the finite space
$\Hom(\Gamma_{g},S_{n})$, Mirzakhani obtained formulas for the integral
of geometric functions over the the moduli space $\mathcal{M}_{g}$
of complete hyperbolic surfaces of genus $g$, with respect to the
Weil-Petersson volume form $d\mathrm{Vol}_{\mathrm{wp}}$.

The geometric functions that Mirzakhani considers are very much like
our Wilson loops. Given any closed curve $\gamma\in\Sigma_{g}$, for
any complete hyperbolic metric $J$ on $\Sigma_{g}$ there is a unique
curve isotopic to $\gamma$ that is shortest with respect to $J$,
and the length of this curve is called the length of $\gamma$, denoted
by $\ell_{J}([\gamma])$. Here $[\gamma]$ is the isotopy class of
$\gamma$. 

Mirzakhani requires that $\gamma$ be simple, meaning that it does
not intersect itself. This condition is not present in the current
paper and can be viewed as an advantage of our work. To obtain a function
on $\mathcal{M}_{g}$, given a continuous function $f:\R_{+}\to\R_{+}$,
Mirzakhani considers the averaged function
\[
f_{\gamma}(J)\eqdf\sum_{[\gamma']\in\MCG(\Sigma_{g}).[\gamma]}f(\ell_{J}([\gamma']))
\]
where $\MCG(\Sigma_{g})$ is the mapping class group of $\Sigma_{g}$.
Because of the averaging over the mapping class group, $f_{\gamma}$
descends to a function on $\mathcal{M}_{g}$. This type of averaging
is not necessary in the current paper because $\X_{g,n}=\Hom(\Gamma_{g},S_{n})$
is already finite; here $\X_{g,n}$ is playing the role of the Teichmüller
space and not the moduli space. In \cite[Thm. 8.1]{Mirzakhani}, Mirzakhani
gives a formula for 
\[
\int_{\mathcal{M}_{g}}f_{\gamma}\,d\mathrm{Vol}_{\mathrm{wp}}
\]
in terms of integrating $f$ against Weil-Petersson volumes of moduli
spaces. The power of this formula is that in the same paper \cite{Mirzakhani},
Mirzakhani gives explicit recursive formulas for the calculations
of Weil-Petersson volumes. For a more detailed discussion of these
formulas, the reader should consult Wright's survey of Mirzakhani's
work \cite[\S 4]{wright2020tour}.

\subsection{Related works II: Free groups\label{subsec:Related-works-II:free groups}}

Let $\F_{r}$ denote a free group on $r$ generators. For $\gamma\in\F_{r}$,
the problem of integrating the Wilson loop 
\[
\fix_{\gamma}(\phi)\eqdf\fix(\phi(\gamma)),\quad\fix_{\gamma}:\Hom(\F_{r},S_{n})\to\R
\]
over $\Hom(\F_{r},S_{n})$ with respect to the uniform probability
measure is a basic problem that serves as a precursor to that of the
current paper. As mentioned above, many of the considerations used
with free groups no longer apply in the present paper. Indeed, $\Hom(\F_{r},S_{n})$
can be identified with $S_{n}^{r}$ and hence techniques for integrating
over groups are relevant in a much more direct way than in the case
of $\Hom\left(\Gamma_{g},S_{N}\right)$.

Despite being an easier problem, the theory is very rich. It was proved
by Nica in \cite{nica1994number} that the analog of Theorem \ref{thm:expected-fixed-points-bounded}
holds for $\E_{\F_{r},n}[\fix_{\gamma}]$. A significantly sharper
result was given by Puder and Parzanchevski in \cite{PP15} where
they proved that if $\gamma\in\F_{r}$, then as $n\to\infty$
\[
\E_{\F_{r},n}\left[\fix_{\gamma}\right]=1+\frac{c(\gamma)}{n^{\pi(\gamma)-1}}+O\left(\frac{1}{n^{\pi(\gamma)}}\right)
\]
where $\pi(\gamma)\in\{0,\ldots,r\}\cup\{\infty\}$ is an algebraic
invariant of $\gamma$ called the \emph{primitivity rank} and $c(\gamma)\in\N$
is explained in terms of the enumeration of special subgroups of $\F_{r}$
determined by $\gamma$. Obtaining a similarly sharp result in the
context of $\Gamma_{g}$ is an interesting problem that should be
taken up in the future.

Similar Laurent series expansions for the expected value of $\chi_{\gamma}$
on $\Hom(\F_{r},G(n))$ have been proved to exist, and studied, when
$G(n)$ is one of the families of compact Lie groups $\mathrm{U}(n),\mathrm{O}(n),\mathrm{Sp}(n)$
\cite{MPunitary,MPorthsymp}, when $G(n)$ is a generalized symmetric
group \cite{MPsurfacewords}, and when $G(n)=\mathrm{\mathrm{GL}}_{n}(\mathbb{F}_{q})$,
where $\mathbb{F}_{q}$ is a fixed finite field \cite{West}. In all
cases $\chi$ is taken to be a natural character. For example, when
$G(n)=\mathrm{U}(n)$, one such $\chi$ is the trace of the matrix
in the group. Moreover, for $G(n)=\mathrm{U}(n),\mathrm{O}(n),\mathrm{Sp}(n)$
and $\chi$ the trace, all the coefficients of the Laurent series
are understood \cite{MPunitary,MPorthsymp}. 

In works undertaken after the completion of this paper, the first
named author has obtained analogs of Theorem \ref{thm:rational-approx}
and the first part of Theorem \ref{thm:expected-fixed-points-bounded}
for\footnote{In this case, instead of the uniform measure on $\Hom(\Gamma_{g},S_{n})$
that we use here, one should use a natural measure on $\Hom(\Gamma_{g},\mathrm{U(}n))$
that arises from the Atiyah-Bott-Goldman symplectic form \cite{Goldman,AB}
on (a non-singular part of) the character variety $\Hom(\Gamma_{g},\mathrm{U(}n))/\mathrm{U}(n)$.} $\Hom(\Gamma_{g},\mathrm{U(}n))$ and the standard matrix trace \cite{magee2021surface-unitary1,magee2021surface-unitary2}.
The methods used in \emph{(ibid.) }are inspired by those of the current
work. 

\subsection{Related works III: Non-commutative probability}

Theorem \ref{thm:expected-fixed-points-bounded} has a direct consequence
in the setting of Voiculescu's non-commutative probability theory.
Following \cite[Def. 2.2.2]{VDN}, a \emph{$C^{*}$-probability space
}is a pair $(\B,\tau)$ where $\B$ is a unital $C^{*}$-algebra and
$\tau$ is a state\footnote{A state on a unital $C^{*}$ algebra is a positive linear functional
such that $\tau(1)=1.$} on $\B$. We say that a sequence $\{(\B,\tau_{n})\}_{n=1}^{\infty}$
of $C^{*}$-probability spaces converges to $(\B,\tau)$ if for all
elements $b\in\B$
\[
\lim_{n\to\infty}\tau_{n}(b)=\tau(b).
\]
The functions $\tau_{n}:\Gamma_{g}\to\R$ defined by $\tau_{n}(\gamma)\eqdf n^{-1}\E_{g,n}[\fix_{\gamma}]$
extend to states on the full group $C^{*}$-algebra $C^{*}(\Gamma{}_{g})$
of $\Gamma_{g}$. There is also a unique state $\tau_{\reg}$ on $C^{*}(\Gamma_{g})$
that satisfies $\tau_{\reg}(g)=0$ for $g\neq1$; we use the subscript
$\reg$ because the GNS representation of $\tau_{\reg}$ is the \emph{left
regular representation}. One has the following corollary of Theorem
\ref{thm:expected-fixed-points-bounded}:
\begin{cor}
\label{cor:The--probability-spaces-converge}The $C^{*}$-probability
spaces $(C^{*}(\Gamma_{g}),\tau_{n})$ converge to $(C^{*}(\Gamma_{g}),\tau_{\reg})$
as $n\to\infty$.
\end{cor}

It is reasonable to hope that similar results can be obtained when
$\Gamma_{g}$ is replaced by any residually finite one-relator group
(cf. $\S\S$\ref{subsec:residual-finiteness}). We view Corollary
\ref{cor:The--probability-spaces-converge} as an important first
step in this program.

\subsection{\label{subsec:residual-finiteness}Related works IV: Residual finiteness}

A f.g.~discrete group $\Lambda$ is \emph{residually finite} if for
any non-identity $\lambda\in\Lambda$ there is a finite index subgroup
$H\le\Lambda$ such that $\lambda\notin H$. The residual finiteness
of $\Gamma_{g}$ has been known for a long time \cite{baumslag1962generalised,Hempel}.
More recently, various quantifications of residual finiteness and
of the related property of LERF\footnote{Locally extended residual finiteness.}
have been proposed by various authors \cite{Bou-Rabee,LLM}. Theorem
\ref{thm:expected-fixed-points-bounded} can serve as a strengthening
of the residual finiteness of $\Gamma_{g}$, as we now explain.

Note that residual finiteness of a group $\Lambda$ is equivalent
to, for all $e\neq\lambda\in\Lambda$, the existence of $n\in\N$
and $\phi\in\Hom(\Lambda,S_{n})$ such that $\phi(\lambda)\ne1$.
Theorem \ref{thm:expected-fixed-points-bounded} combined with Markov's
inequality implies the following quantitative version of residual
finiteness:
\begin{cor}
Given a non-identity element $e\ne\gamma\in\Gamma_{g}$, for large
enough $n$,
\begin{equation}
\frac{\left|\left\{ \phi\in\Hom(\Gamma_{g},S_{n})\,:\,\text{\ensuremath{\phi(\gamma)\neq1}}\right\} \right|}{|\Hom(\Gamma_{g},S_{n})|}\ge1-\frac{d\left(q\right)}{n}-O\left(\frac{1}{n^{2}}\right),\label{eq:quantitative-RF}
\end{equation}
where $q$ and $d\left(q\right)$ are as in Theorem \ref{thm:expected-fixed-points-bounded},
and the implied constant in the big-$O$ term depends on $\gamma$.
\end{cor}

In fact, the techniques of this paper can be used to show that, for
example, for every $m\in\mathbb{N}$, the expected value of $\fix_{\gamma}^{~m}$
is of the form $c\left(q\right)+O\left(n^{-1}\right)$, where $q$
is as in Theorem \ref{thm:expected-fixed-points-bounded} and $c\left(q\right)$
is a positive integer. This would yield a probability bound similar
to (\ref{eq:quantitative-RF}) but of the form $1-\frac{c\left(q\right)}{n^{m}}+O\left(n^{-m-1}\right)$.

\subsection{Related works V: Benjamini-Schramm convergence.}

In \cite{BenjaminiSchramm} Benjamini and Schramm introduced a notion
of convergence of a sequence of finite graphs to a limiting graph,
known now as \emph{Benjamini-Schramm} convergence. This concept was
extended to convergence of sequences of Riemannian manifolds in \cite{ABBGNRS1,ABBGNRS2}.
Theorem \ref{thm:expected-fixed-points-bounded} has consequences
for the Benjamini-Schramm convergence of random covers of Riemannian
surfaces; there are various of these consequences but we present just
one representative one here\footnote{This consequence of Theorem \ref{thm:expected-fixed-points-bounded}
was first pointed out by Baker and Petri in \cite{BakerPetri}.}.
\begin{cor}
Let $X$ be a closed hyperbolic surface of genus $\geq2$. With respect
to Benjamini-Schramm distance, uniformly random degree-$n$ covering
spaces of $X$ converge in probability as $n\to\infty$ to the hyperbolic
upper half plane $\mathbb{H}$.
\end{cor}

Concretely this means that for any $L>0$ and $\varepsilon>0$, if
$X_{n}$ denotes a random degree-$n$ cover of $X$ (as above), then
a.a.s.~as $n\to\infty$, 

\[
\frac{\mathrm{area}\left(X_{n}^{<L}\right)}{\mathrm{area}\left(X_{n}\right)}<\varepsilon,
\]
where $X_{n}^{<L}$ is the points of $X_{n}$ with local injectivity
radius $<L$. To see how this follows from Theorem \ref{thm:expected-fixed-points-bounded},
viewing $L$ as a constant, any point in $X_{n}^{<L}$ is in a neighborhood,
with bounded area depending on $L$, of some simple closed geodesic
of $X_{n}$ with length $<2L$ \cite[proof of Thm. 4.1.6]{Buser}.
Any such geodesic covers a closed (possibly non-primitive) geodesic
in $X$ of length $<2L$ and these in turn correspond to a finite
list of conjugacy classes in $\Gamma_{g}$. Starting with a conjugacy
class $[\gamma]$, the number of corresponding closed lifted geodesics
in $X_{n}$ is at most $\fix_{\gamma}$. Using Markov's inequality
with Theorem \ref{thm:expected-fixed-points-bounded} gives therefore
a.a.s. that the number of simple closed geodesics of $X_{n}$ with
length $<2L$ is bounded (depending on $L$). This means $\mathrm{area}\left(X_{n}^{<L}\right)$
is bounded a.a.s. and as $\mathrm{area}\left(X_{n}\right)$ is linear
in $n$, this completes the proof.

\subsection{Structure of the proofs and the issues that arise\label{subsec:Overview-of-the-paper}}

The reader of the paper is advised to first read this $\S$\ref{subsec:Overview-of-the-paper},
and then $\S$\ref{sec:Proofs-of-main-theorems}, where all the ideas
of the paper are brought together to give concise proofs of Theorems
\ref{thm:rational-approx}, \ref{thm:expected-fixed-points-bounded},
and \ref{thm:subgroups}, before reading the other sections.

There are two main ideas of the paper, that we will discuss momentarily.
Here we give a `high-level' account of the strategy of proving our
main theorems. At times we oversimplify definitions to be more instructive.
Let us fix $g=2$, and discuss only Theorems \ref{thm:rational-approx}
and \ref{thm:expected-fixed-points-bounded}. The extension of these
results from cyclic groups to more general finitely generated subgroups
is along the same lines. So fix $\gamma\in\Gamma_{2}$.

Firstly, we view $\X_{n}=\X_{2,n}$ as a space of random coverings
of a fixed genus 2 surface $\Sigma_{2}$. By fixing an octagonal fundamental
domain of $\Sigma_{2}$, each covering of $\Sigma_{2}$ is tiled by
octagons. This leads us to the notion of a \emph{tiled surface,} defined
precisely in Definition \ref{def:tiled-surface}. A tiled surface
involves not just a tiling, but a labeling of the edges of the tiling
by generators of the fundamental group of $\Sigma_{2}$. Hence all
the main theorems can be reinterpreted in terms of random tiled surfaces
that are called $X_{\phi}$ for $\phi\in\X_{n}$.

The first observation is that $\mathbb{E}_{n}\left[\fix_{\gamma}\right]=\mathbb{E}_{2,n}\left[\fix_{\gamma}\right]$,
the expected number of fixed points of $\gamma$ under $\phi\in\X_{n}$,
is the expected number of times that we see a fixed annulus $A$,
specified by $\gamma$, immersed in the random tiled surface $X_{\phi}$.
However, this annulus needs not be embedded. On the other hand, it
is possible to produce a finite collection $\mathcal{R}$ of tiled
surfaces, each of which has an immersed copy of $A$, such that 
\begin{equation}
\E_{n}\left[\fix_{\gamma}\right]=\sum_{Y\in\mathcal{R}}\E_{n}^{\emb}(Y).\label{eq:resolution-consequence}
\end{equation}
where $\E_{n}^{\emb}(Y)$ is the expected number of times that $Y$
is \uline{embedded} in the random $X_{\phi}$.

We formalize types of collections $\mathcal{R}$ that have the above
property in Definition \ref{def:resolutions}; we call them \emph{resolutions}
(of $A$). Of course, there is a great deal of flexibility in how
$\mathcal{R}$ is chosen; we will come back to this point shortly.
The benefit to having (\ref{eq:resolution-consequence}) brings us
to the first main idea of the paper:

\emph{We have a new method of calculating $\E_{n}^{\emb}(Y)$, using
the representation theory of symmetric groups $S_{n}$ and more specifically,
the approach to the representation theory of $S_{n}$ developed by
Vershik and Okounkov in \cite{VershikOkounkov}.}

This methodology is developed in $\S$\ref{sec:The-probability-of-an-embedded-tiled-surface}.
The necessary background on representation theory is given in $\S$\ref{sec:Representation-theory-of-symmetric-group},
and in $\S$\ref{sec:Preliminary-representation-theor} we prove some
preliminary representation theoretic results needed for $\S$\ref{sec:The-probability-of-an-embedded-tiled-surface}.
The reader may be interested to see that Theorem \ref{thm:rational-approx}
has, at its source, Proposition \ref{prop:holomorphic-matrix-coefficients}.
See also the overview of $\S$\ref{sec:The-probability-of-an-embedded-tiled-surface}
in $\S\S$\ref{subsec:Overview-of-this}.

This new methodology to calculate $\E_{n}^{\emb}(Y)$ is sufficient
to prove Theorem \ref{thm:rational-approx}. However, in the proof
of Theorem \ref{thm:expected-fixed-points-bounded}, a critical issue
now intervenes. We expect, based on experience with similar projects
(e.g. \cite{PP15,MPunitary}) that 
\begin{equation}
\E_{n}^{\emb}(Y)\approx n^{\chi(Y)}\label{eq:guess}
\end{equation}
as $n\to\infty$. However, this cannot always be the case. For example,
if, roughly speaking, it is possible to glue some octagons to $Y$
to increase the Euler characteristic, forming $Y'$, then the observation
that $\E_{n}^{\emb}(Y)\geq\E_{n}^{\emb}(Y')$ breaks (\ref{eq:guess}).
Then it is not unsurprising that the bounds we obtain for $\E_{n}^{\emb}(Y)$
do not always agree with (\ref{eq:guess}). 

On the other hand, if $Y$ has special properties that we call \emph{boundary
reduced} and \emph{strongly boundary} \emph{reduced, }then we can
get appropriate bounds on\emph{ $\E_{n}^{\emb}(Y)$. }We give the
precise definitions of these properties in Definitions \ref{def:br}
and \ref{def:sbr}. They involve forbidding certain constellations
from appearing in the boundary of $Y$. Even though these constellations
are dictated by representation theory, forbidding them remarkably
relates to natural geometric properties of $Y$. For example, if $Y$
is not boundary reduced, then it is possible to add octagons to $Y$
to decrease the number of edges in its boundary. To give some more
intuition, being boundary reduced can be viewed as a discrete analog
of a hyperbolic surface having geodesic boundary. This means that
these properties are closely related with the problem of finding shortest
representatives (with respect to word length) of elements of $\Gamma_{g}$,
that is addressed by Dehn's algorithm \cite{Dehn}.

If $Y$ is boundary reduced, then we can prove (Theorem \ref{thm:E_n-emb-exact-expression}
and Proposition \ref{prop:xi-bound-for-boundary reduced})
\[
\E_{n}^{\emb}(Y)=O\left(n^{\chi(Y)}\right),
\]
and if $Y$ is strongly boundary reduced, we can prove (Theorem \ref{thm:E_n-emb-exact-expression}
and Proposition \ref{prop:xi-bound-for-strongly-boundary reduced})
\[
\E_{n}^{\emb}(Y)=n^{\chi\left(Y\right)}\left(1+O\left(n^{-1}\right)\right)
\]
(see, again, Section \ref{subsec:Overview-of-this} for a more detailed
overview). Therefore, to prove Theorem \ref{thm:expected-fixed-points-bounded},
it suffices to produce resolutions of the annulus $A$ where we can
control which elements are (strongly) boundary reduced, control their
Euler characteristics, and count the number of elements with maximal
Euler characteristic. The design of these resolutions is the second
main theme of the paper.

For any tiled surface $Z$, we describe an algorithm to produce finite
resolutions of $Z$ with careful control on their properties as above.
This is the main topic of $\S$\ref{sec:resolutions}.\emph{ }Precisely
defining the annulus $A$ that should be used as input, as well as
its generalization for non-cyclic subgroups $J\leq\Gamma$, and counting
the outputs of our algorithm, requires introducing the concept of
a \emph{core surface of a subgroup $J\leq\Gamma$. }For example, above,
$A$ should be taken to be the core surface of $\langle\gamma\rangle$.
The theory of core surfaces that we develop in a companion paper \cite{MPcore}
is analogous to that of Stallings' core graphs\emph{ }for subgroups
of free groups due to Stallings \cite{stallings1983topology}, and
we hope that the results therein may be of independent interest.

\subsection{Notation\label{subsec:Notation}}

Write $\N$ for the natural numbers $1,2,\ldots$. For $n\in\N$ we
use the notation $[n]$ for the set $\{1,\ldots,n\}$. For $m\leq n$,
$m,n\in\N$ we write $[m,n]$ for the set $\{m,m+1,\ldots,n\}$. If
$A$ and $B$ are sets we write $A\backslash B$ for the elements
of $A$ that are not in $B$. We write $(n)_{\ell}$ for the Pochhammer
symbol
\[
(n)_{\ell}\eqdf n(n-1)\ldots(n-\ell+1).
\]

If $V$ is a vector space we write $\End(V)$ for the linear endomorphisms
of $V$. If $V$ is a unitary representation of some group we write
$\check{V}$ for the dual representation. If $P_{1},\ldots,P_{k}$
are a series of expressions we write $\mathbf{1}_{\left\{ P_{1},\ldots,P_{k}\right\} }$
for a value which is 1 if all the statements $P_{i}$ are true and
$0$ else. If $V$ is a vector space we write $\mathrm{Id}_{V}$ for
the identity operator on that space. All integrals over finite sets
are with respect to the uniform probability measure\emph{ }on the
set\emph{. }If $X$ is a CW-complex we write $X^{(i)}$ for its $i$-skeleton.
If we use the symbol $\pm$ more than once in the same expression
or equation, we mean that the same sign is chosen each time. If implied
constants in big-O notation depend on other constants, we indicate
this by adding the constants as a subscript to the $O$, for example,
$O_{\varepsilon}(f(n))$ means the implied constant depends on $\varepsilon$.
We use Vinogradov notation $f(n)\ll g(n)$ to mean that there are
constants $n_{0}\ge0$ and $C_{0}>0$ such that for $n>n_{0}$, $|f(n)|\leq C_{0}g(n)$.
We add subscripts to indicate dependence of the implied constants
on other quantities or objects. If $a,b$ are elements of the same
group, we write $[a,b]\eqdf aba^{-1}b^{-1}$ for their commutator.

\subsection*{Acknowledgments }

We thank Nir Avni, Frédéric Naud, Mark Powell, and Henry Wilton for
helpful discussions related to this work. The research was supported
by the Israel Science Foundation: ISF grant 1071/16 of D.P. This project
has received funding from the European Research Council (ERC) under
the European Union’s Horizon 2020 research and innovation programme
(grant agreement No 850956 and grant agreement No 949143).

\section{Resolutions of core surfaces \label{sec:resolutions}}

\subsection{Tiled surfaces and core surfaces\label{subsec:tiled-surfaces-and-core-surfaces}}

In this $\S\S$\ref{subsec:tiled-surfaces-and-core-surfaces} we summarize
some definitions and results from \cite{MPcore}\footnote{A significant chunk of \cite{MPcore} was part of the first version
of the current paper. As we believe the theory of core surfaces is
of independent interest, and in order to keep the current paper to
a manageable size, we decided to develop an expanded version of this
theory in a separate paper.}. 

\subsubsection{Tiled surfaces}

Consider the construction of the surface $\Sigma_{g}$ from a $4g$-gon
by identifying its edges in pairs according to the pattern $a_{1}b_{1}a_{1}^{-1}b_{1}^{-1}\ldots a_{g}b_{g}a_{g}^{-1}b_{g}^{-1}$.
This gives rise to a CW-structure on $\Sigma_{g}$ consisting of one
vertex (denoted $o$), $2g$ oriented $1-$cells (denoted $a_{1},b_{1},\ldots,a_{g},b_{g}$)
and one $2$-cell which is the $4g$-gon glued along $4g$ $1$-cells\footnote{We use the terms vertices and edges interchangeably with $0$-cells
and $1$-cells, respectively.}. See Figure \ref{fig:Sigma_2} (in our running examples with $g=2$,
we denote the generators of $\Gamma_{2}$ by $a,b,c,d$ instead of
$a_{1},b_{1},a_{2},b_{2}$). We identify $\Gamma_{g}$ with $\pi_{1}\left(\Sigma_{g},o\right)$,
so that in the presentation (\ref{eq:presentation of Gamma}), words
in the generators $a_{1},\ldots,b_{g}$ correspond to the homotopy
class of the corresponding closed paths based at $o$ along the $1$-skeleton
of $\Sigma_{g}$. 

\begin{figure}
\begin{centering}
\includegraphics[scale=0.5]{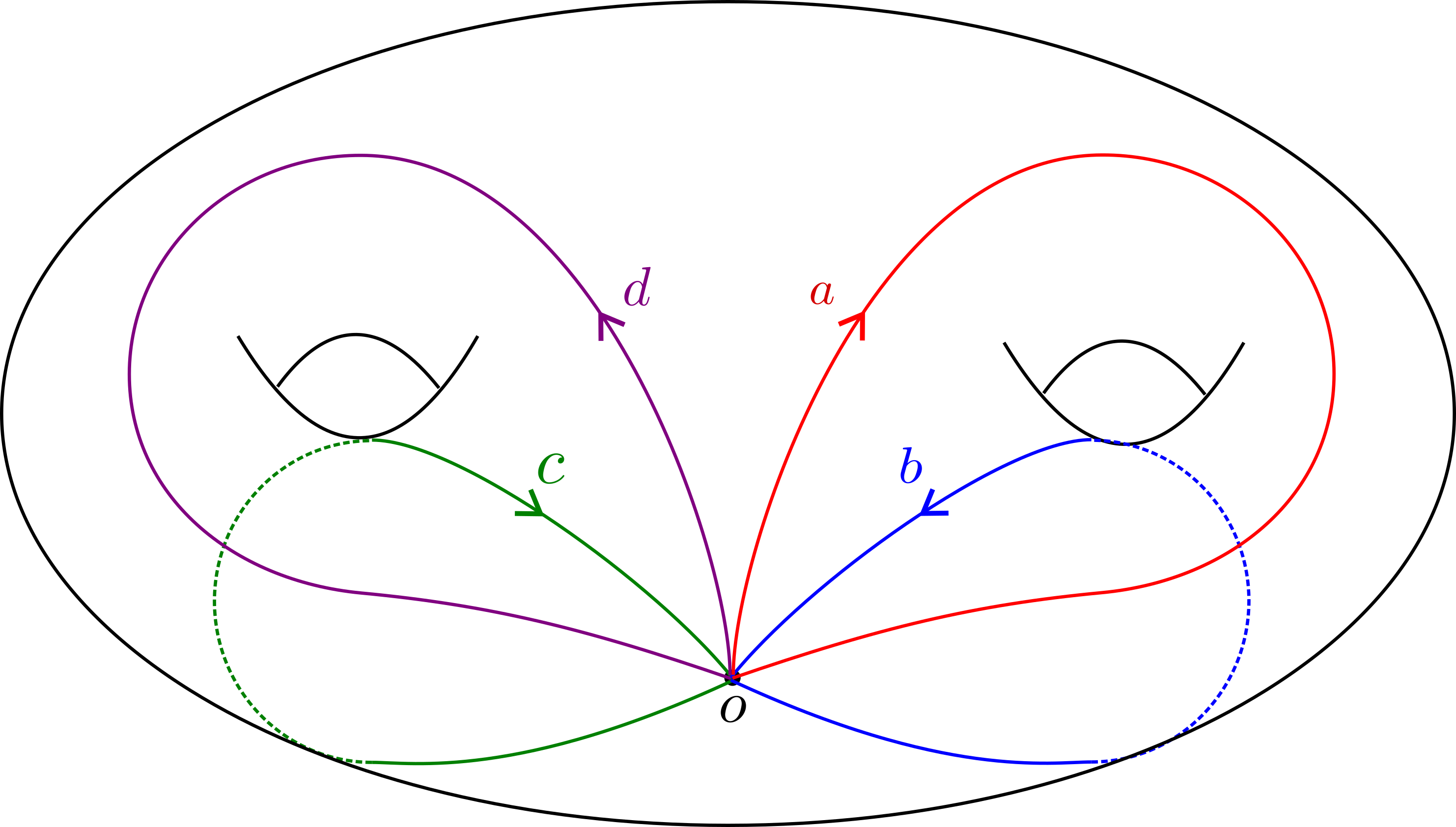}
\par\end{centering}
\caption{\label{fig:Sigma_2}The fixed CW-structure on $\Sigma_{2}$}
\end{figure}

Note that every covering space $p\colon\Upsilon\to\Sigma_{g}$ inherits
a CW-structure from $\Sigma_{g}$: the vertices are the pre-images
of $o$, and the open $1$-cells (2-cells) are the connected components
of the pre-images of the open 1-cells (2-cells, respectively) in $\Sigma_{g}$.
In particular, this is true for the universal covering space $\tsg$\marginpar{$\protect\tsg$}
of $\Sigma_{g}$, which we can now think of as a CW-complex. A sub-complex
of a CW-complex is a subspace consisting of cells such that if some
cell belongs to the subcomplex, then so are the cells of smaller dimension
at its boundary. 
\begin{defn}[Tiled surface]
\label{def:tiled-surface}\cite[Definition 3.1]{MPcore} A \emph{tiled
surface} $Y$ is a sub-complex of a (not-necessarily-connected) covering
space of $\Sigma_{g}$. In particular, a tiled surface is equipped
with the restricted covering map $p\colon Y\to\Sigma_{g}$ which is
an immersion. We write $\v\left(Y\right)$ for the number of vertices
of $Y$, $\e\left(Y\right)$ for the number of edges and $\f\left(Y\right)$\marginpar{$\text{\ensuremath{{\scriptstyle \protect\v\left(Y\right),\protect\e\left(Y\right),\protect\f\left(Y\right)}}}$}
for the number of $4g$-gons. 
\end{defn}

Alternatively, instead of considering a tiled surface $Y$ to be a
complex equipped with a restricted covering map, one may consider
$Y$ to be a complex as above with directed and labeled edges: the
directions and labels ($a_{1},b_{1},\ldots,a_{g},b_{g}$) are pulled
back from $\Sigma_{g}$ via $p$. These labels uniquely determine
$p$ as a combinatorial map between complexes. Figures \ref{fig:Sigma_2}
and \ref{fig:core surfaces - examples} feature examples of tiled
surfaces.

Note that a tiled surface is not always a surface: it may also contain
vertices or edges with no $2$-cells incident to them. However, as
$Y$ is a sub-complex of a covering space of $\Sigma_{g}$, namely,
of a surface, any neighborhood of $Y$ inside the covering is a surface,
and it is sometimes beneficial to think of $Y$ as such. 
\begin{defn}[Thick version of a tiled surface]
\label{def:thick-version}\cite[Definition 3.2]{MPcore} Given a
tiled surface $Y$ which is a subcomplex of the covering space $\Upsilon$
of $\Sigma_{g}$, adjoin to $Y$ a small, closed, tubular neighborhood
in $\Upsilon$ around every edge and a small closed disc in $\Upsilon$
around every vertex. The resulting closed surface, possibly with boundary,
is referred to as the \emph{thick version of $Y$}.

We let $\partial Y$ denote the boundary of the thick version of $Y$
and $\d\left(Y\right)$\marginpar{$\partial Y,\protect\d\left(Y\right)$}
denote the number of edges along $\partial Y$ (so if an edge of $Y$
does not border any $4g$-gon, it is counted twice). 
\end{defn}

In particular, $\d\left(Y\right)=2\e\left(Y\right)-4g\f\left(Y\right)$.
We stress that we do not think of $Y$ as a sub-complex, but rather
as a complex for its own sake, which happens to have the capacity
to be realized as a subcomplex of a covering space of $\Sigma_{g}$.
See \cite[Section 3]{MPcore} for a more detailed discussion. 

It is occasionally useful, for example in Section \ref{sec:The-probability-of-an-embedded-tiled-surface},
to augment the tiled surface $Y$ by adding some new half-edges. Here,
formally, a half-edge is a copy of the interval $[0,\frac{1}{2})$
which is an (open) half of an edge of a covering space of $\Sigma_{g}$.
\begin{defn}[Tiled surface with hanging half-edges]
\label{def:tiled surfaces with hanging half edges}\cite[Section 3.2]{MPcore}
Let $Y$ be a tiled surface which is a subcomplex of the covering
space $p\colon\Upsilon\to\Sigma_{g}$. We denote by $Y_{+}$\marginpar{$Y_{+}$}
the tiled surface $Y$ together with half edges of $\Upsilon$ which
do not belong to $Y$ but are incident to vertices of $Y$. Every
half-edge of $Y_{+}$ added to $Y$ in this manner is called a \emph{hanging
half-edge. }The thick version of $Y_{+}$ is, as above, $Y_{+}$ together
with a small, closed, tubular neighborhood in $\Upsilon$ around every
edge or hanging half-edge, and a small closed disc in $\Upsilon$
around every vertex. We denote by $\partial Y_{+}$ the boundary of
the think version of $Y_{+}$.
\end{defn}

Note that there are exactly $4g$ half-edges incident to every vertex
in $Y_{+}$: some of them originate from edges in $Y$ and some are
hanging half-edges. 

\paragraph*{Morphisms of tiled surfaces}

If $Y_{1}$ and $Y_{2}$ are tiled surfaces, a \emph{morphism} from
$Y_{1}$ to $Y_{2}$ is a map of $CW$-complexes which maps $i$-cells
to $i$-cells for $i=0,1,2$ and respects the directions and labels
of edges. Equivalently, this is a morphism of CW-complexes which commutes
with the restricted covering maps $p_{j}\colon Y_{j}\to\Sigma_{g}$
($j=1,2$). In particular, the restricted covering map from a tiled
surface to $\Sigma_{g}$ is itself a morphism of tiled surfaces. It
is an easy observation that every morphism of tiled surfaces is an
immersion (locally injective). 

\subsubsection{Blocks and Chains}

Some of the notions we use below are taken from \cite{BirmanSeries}.
See \cite[Section 3.2]{MPcore} for a more detailed account.

Given a covering space $\Upsilon$ of $\Sigma_{g}$, every path in
the $1$-skeleton $\Upsilon^{\left(1\right)}$ corresponds to a word
in $\left\{ a_{1}^{\pm1},\ldots,b_{g}^{\pm1}\right\} $. A path that
follows a (part of the) boundary of a $4g$-gon is called a \textbf{\emph{block}}\emph{.
}If it has length at least $2g+1$ it is called a \textbf{\emph{long
block}}, and if it has length exactly $2g$, it is called a \textbf{\emph{half-block}}.
If a (non-cyclic) block of length $b$ sits along the boundary of
a $4g$-gon $O$, the \textbf{\emph{complement}}\textbf{ of the block
}is the block of length $4g-b$ consisting of the complement set of
edges along $O$, so the block and its complement share the same starting
point and the same terminal point.

A \textbf{\emph{chain}}\emph{ }is a path in $\Upsilon^{\left(1\right)}$
that consists of a sequence of blocks $b_{1},\ldots,b_{r}$, such
that if the last vertex of $b_{i}$ and the first vertex of $b_{i+1}$
is $v$, there is exactly one edge incident to $v$ between the last
edge of $b_{i}$ and the first edges of $b_{i+1}$. In other words,
if the $4g$-gons corresponding to the blocks $b_{1},\ldots,b_{r}$
are $O_{1},\ldots,O_{r}$, then $O_{i}$ and $O_{i+1}$ share an edge
$e$ with an endpoint $v$, and $b_{i}$ ends at $v$ and $b_{i+1}$
starts at $v$. See Figure \ref{fig:a complement of a long chain}.
A \textbf{\emph{long chain}} is a chain with corresponding blocks
of lengths 
\[
2g,2g-1,2g-1,\ldots,2g-1,2g.
\]
A \textbf{\emph{half-chain }}is a \emph{cyclic }chain (so the corresponding
path is closed) consisting of blocks each of which is of length $2g-1$.
The \textbf{complement of a long chain }is the chain with blocks of
lengths $2g-1,2g-1,\ldots,2g-1$ which sits along the other side of
the $4g$-gons bordering the long chain and with the same starting
point and endpoint. Note that the complement of a long chain is shorter
by two edges from the long chain (see Figure \ref{fig:a complement of a long chain}).
The \textbf{complement of a half-chain} is defined as follows. If
the half-chain sits along the boundary of the $4g$-gons $O_{1},\ldots,O_{r}$,
its complement is the half-chain sitting along the other sides of
these $4g$-gons: a block (of length $2g-1$) of the half-chain along
$O_{i}$ is replaced by the path of length $2g-1$ along $O_{i}$,
with starting and terminal points one edge away from the starting
and terminal points, respectively, of the block. The complement of
a half-chain has the same length as the original half-chain. The left
part of Figure \ref{fig:core surfaces - examples} illustrates two
complementing half-chains of length $6$ each (with two octagons in
between).

\begin{figure}
\begin{centering}
\includegraphics{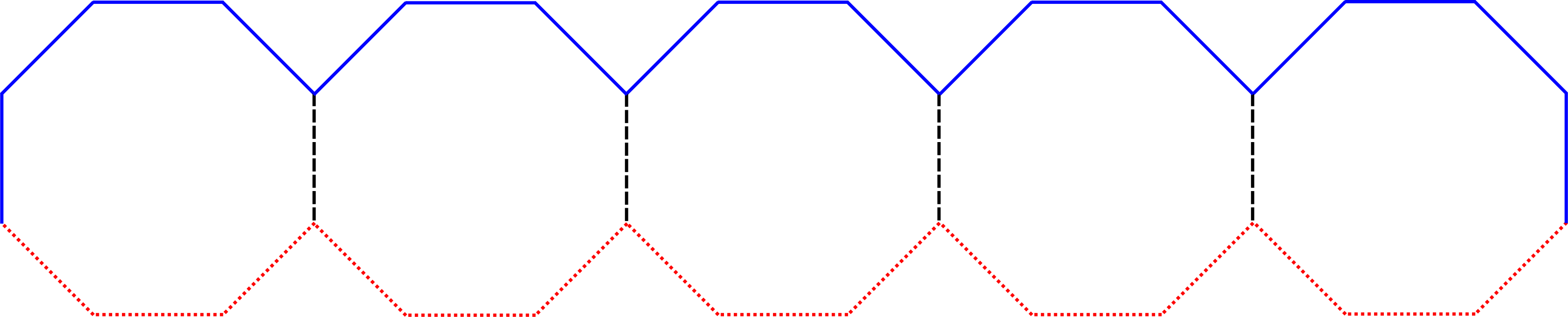}
\par\end{centering}
\caption{\label{fig:a complement of a long chain}A long chain of total length
$17$ (blocks of sizes $4,3,3,3,4$, in blue) and its complement of
length $15$ (in red)}
\end{figure}

A \textbf{\emph{boundary cycle}} of $Y$ is a cycle in $Y^{\left(1\right)}$
corresponding to an oriented boundary component of the thick version
of $Y$ (see Definition \ref{def:thick-version}). We always choose
the orientation so that there are no $4g$-gons to the immediate \textbf{left
}of the boundary component as it is traversed. Therefore boundary
components of $Y$ correspond to unique cycles. Note that $\d\left(Y\right)$
is equal to the sum over boundary cycles of $Y$ of the number of
edges in each such cycle.

\subsubsection{Boundary reduced and strongly boundary reduced tiled surfaces\label{subsec:BR and SBR}}

The following definitions came up from our results in representation
theory in $\S$\ref{sec:The-probability-of-an-embedded-tiled-surface},
but they fit perfectly with classical results in combinatorial group
theory \cite{Dehn} and in particular with \cite{BirmanSeries}. 
\begin{defn}[Boundary reduced]
\label{def:br} A tiled surface $Y$ is \emph{boundary reduced} if
no boundary cycle of $Y$ contains a long block or a long chain.
\end{defn}

In particular, if $Y$ is boundary reduced, then every path that reads
$\left[a_{1},b_{1}\right]\ldots\left[a_{g},b_{g}\right]$ is not only
closed, but there is also a $4g$-gon attached to it. We also need
a stronger version of this property.
\begin{defn}[Strongly boundary reduced]
\label{def:sbr}A tiled surface $Y$ is \emph{strongly boundary reduced}
if no boundary cycle of $Y$ contains a half-block or is a half-chain.
\end{defn}

Because a long block contains (at least two) half-blocks and a long
chain contains (two) half-blocks, a strongly boundary reduced tiled
surface is in particular boundary reduced. The relevance of the notions
of being (strongly) boundary reduced is that our techniques for estimating
$\E_{n}^{\emb}\left(Y\right)$ for a tiled surface $Y$ only give
the right type of estimates when $Y$ is boundary reduced -- see
Proposition \ref{prop:xi-bound-for-boundary reduced}. If $Y$ is
strongly boundary reduced we get even better estimates -- see Proposition
\ref{prop:xi-bound-for-strongly-boundary reduced}. 

Let $Y$ be a compact tiled surface which is a subcomplex of the covering
space $\Upsilon$ of $\Sigma_{g}$. In \cite[Section 4]{MPcore} we
describe the ``boundary reduced closure'' $\br\left(Y\hookrightarrow\Upsilon\right)$
of $Y$ in $\Upsilon$ which is the smallest intermediate tiled surface
which is boundary reduced. By (ibid, Proposition 4.6), $\br\left(Y\hookrightarrow\Upsilon\right)$
is compact too. Likewise, $\sbr\left(Y\hookrightarrow\Upsilon\right)$,
the strongly boundary reduced closure, is the smallest intermediate
tiled surface which is strongly boundary reduced, but this one in
not always compact. Our resolutions in Section \ref{subsec:Resolutions}
are based on a ``compromise'' between these two types of closures.

\subsubsection{Core surfaces\label{subsec:Core-surfaces}}

Finally, let us define the main object which was introduced and analyzed
in \cite{MPcore}, with motivation coming from the current paper.
In analogy to Stallings core graphs and their role in the study of
free groups and their subgroups, we introduced the notion of \emph{core
surfaces} which relates to subgroups of $\Gamma_{g}$:
\begin{defn}[Core surfaces]
\label{def:core surface}\cite[Definition 1.1]{MPcore} Given a subgroup
$1\ne J\le\Gamma_{g}=\pi_{1}\left(\Sigma_{g},o\right)$, consider
the covering space $p\colon\Upsilon\to\Sigma_{g}$ corresponding to
$J$, so $\Upsilon=J\backslash\tsg$. Define the \textbf{core surface
of $J$}, denoted $\core\left(J\right)$\marginpar{$\protect\core\left(J\right)$},
as the tiled surface which is a sub-complex of $\Upsilon$ as follows:
$\left(i\right)$ take the union of all shortest-representative cycles
in the 1-skeleton $\Upsilon^{\left(1\right)}$ of every free-homotopy
class of essential closed curve in $\Upsilon$, and $\left(ii\right)$
add every connected component of the complement which contains finitely
many $4g$-gons.

For completeness define the core surface of the trivial subgroup to
be the $0$-dimensional tiled surface consisting of a single vertex.
\end{defn}

Note that the quotient $\Upsilon=J\backslash\tsg$ is invariant under
conjugation of $J$, so $\core\left(J\right)$ depends only on the
conjugacy class of $J$ in $\Gamma_{g}$. Figure \ref{fig:core surfaces - examples}
gives two examples of core surfaces. As another example, if $J$ is
of finite index in $\Gamma$, $\core\left(J\right)$ is identical
to $J\backslash\tsg$ and is a compact closed surface.

\begin{figure}
\begin{centering}
\includegraphics{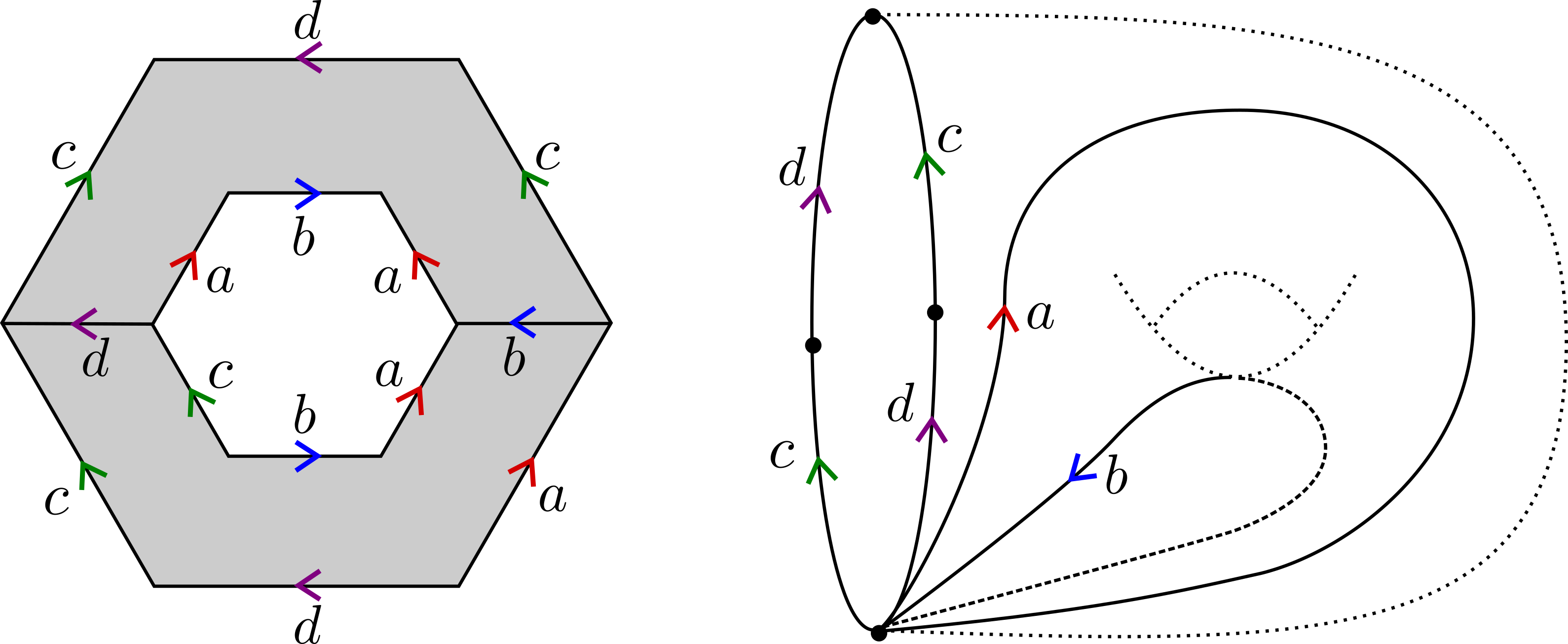}
\par\end{centering}
\caption{\label{fig:core surfaces - examples}Fix $g=2$ and let $\Gamma_{2}=\left\langle a,b,c,d\,\middle|\,\left[a,b\right]\left[c,d\right]\right\rangle $.
On the left is the core surface $\protect\core\left(\left\langle aba^{-2}b^{-1}c\right\rangle \right)$.
It consists of $12$ vertices, $14$ edges and two octagons, and topologically
it is an annulus. On the right is the core surface $\protect\core\left(\left\langle a,b\right\rangle \right)$.
It consists of four vertices, six edges and one octagon, and topologically
it is a genus-1 torus with one boundary component.}
\end{figure}

In \cite{MPcore} we give an intrinsic definition of a core surface
and show there is one-to-one correspondence between core surfaces
(labeled by $a_{1},\ldots,b_{g}$) and conjugacy classes of subgroups
of $\Gamma_{g}$, we provide a ``folding process'' to construct
$\core\left(J\right)$ from a finite generating set of $J$ (provided,
of course, that $J$ is f.g.), and prove basic properties of core
surfaces. In particular, $\core\left(J\right)$ is connected and strongly
boundary reduced (ibid, Proposition 5.3), and whenever $J$ is f.g.,
$\core\left(J\right)$ is compact (ibid, Proposition 5.8). We also
show that whenever $H\le J\le\Gamma_{g}$, the natural morphism between
the corresponding covering spaces $H\backslash\tsg\to J\backslash\tsg$,
restricts to a morphism $\core\left(H\right)\to\core\left(J\right)$.

\subsection{Expectations and probabilities of tiled surfaces\label{subsec:Expectations-and-probabilities}}

\subsubsection*{Correspondence between $\protect\Hom\left(\Gamma_{g},S_{n}\right)$
and $n$-sheeted covering spaces of $\Sigma_{g}$}

Let $M$ be a connected topological space with basepoint $m$. Consider
$n$-sheeted covering spaces of $M$ with the fiber above $m$ labeled
by $\left[n\right]$, so that every point in the fiber has a different
label. If $M$ is ``nice enough'', in particular if $M$ is a surface,
there is a one-to-one correspondence between these labeled $n$-sheeted
covering spaces and the set of homomorphisms $\mathrm{Hom}\left(\pi_{1}\left(M,m\right),S_{n}\right)$
(see, for instance, \cite[pages 68-70]{hatcher2005algebraic}). If
$\hat{M}$ is a labeled $n$-sheeted covering space and $p\colon\hat{M}\to M$
the covering map, the corresponding homomorphism $\theta\colon\pi_{1}\left(M,m\right)\to S_{n}$
is given as follows: for $h\in\pi_{1}\left(M,m\right)$, consider
$\gamma$, a closed path in $M$, based at $m$, which represents
$h$. Then $\theta\left(h\right)\left(i\right)=j$ if and only if
the lift of $\gamma$ at the point $i$ ends at the point $j$.\footnote{There is a subtle issue here with the direction in which permutations
are multiplied in $S_{n}$. The map $\theta\colon\pi_{1}\left(M,m\right)\to S_{n}$
as defined here is a homomorphism only if permutations in $S_{n}$
are composed from right to left. We refer to this issue in the case
of $M=\Sigma_{2}$ in the beginning of Section \ref{sec:The-probability-of-an-embedded-tiled-surface}.}

In our case, this translates to a one-to-one correspondence between
the representation space $\X_{g,n}=\mathrm{Hom}\left(\Gamma_{g},S_{n}\right)$
and labeled $n$-sheeted covering spaces of $\Sigma_{g}$ (pointed
at $o$). For $\phi\colon\Gamma_{g}\to S_{n}$, denote the corresponding
covering space by $p_{\phi}\colon X_{\phi}\to\Sigma_{g}$. As explained
above, $X_{\phi}$ inherits a CW-structure from $\Sigma_{g}$ and
is, therefore, a tiled surface. The fiber above $o$ is precisely
the vertices of $X_{\phi}$, and they are labeled by $\left[n\right]$
in this construction.

\subsubsection*{Expected number of fixed points as expected number of lifts}

Given a compact tiled surface $Y$, we are interested in the expected
number of morphisms from $Y$ to a random $n$-covering of $\Sigma_{g}$,
namely, in
\[
\mathbb{E}_{n}\left(Y\right)\eqdf\mathbb{E}_{\phi\in\mathbb{X}_{g,n}}\left[\#\left\{ \mathrm{morphisms}~Y\to X_{\phi}\right\} \right],
\]
where $\phi$ is sampled uniformly at random from $\mathbb{X}_{g,n}$.
Equivalently, this is the expected number of lifts of the restricted
covering map $p:Y\to\Sigma_{g}$ to the random $n$-covering $X_{\phi}$:
\[
\xymatrix{ & X_{\phi}\ar@{->>}[d]^{p_{\phi}}\\
Y\ar[r]_{p}\ar@{-->}[ur] & \Sigma_{g}
}
\]
Note that if $Y$ is connected and $\phi\in\X_{g,n}$, the number
of morphisms $Y\to X_{\phi}$ is at most $n$, as any vertex of $Y$
can be lifted to one of the $n$ vertices of $X_{\phi}$, and each
such lift can be extended in at most one way to a lift of the whole
of $Y$. For suitable choices of $Y$, $\E_{n}\left(Y\right)$ is
equal to the quantities $\E_{g,n}\left[\fix_{\gamma}\right]$ and
$\E_{g,n}\left[\fix_{J}\right]$ that feature in our main theorems
(Theorems \ref{thm:rational-approx}, \ref{thm:expected-fixed-points-bounded}
and \ref{thm:subgroups}):
\begin{lem}
\label{lem:E(fix) as expected number of lifts}Let $Y$ be a connected
tiled surface and $p\colon Y\to\Sigma_{g}$ the restricted covering
map. For arbitrary vertex $y\in Y$, assume that $p_{*}\left(\pi_{1}\left(Y,y\right)\right)$
is conjugate to $J\le_{\mathrm{f.g.}}\Gamma_{g}$ (as a subgroup of
$\pi_{1}\left(\Sigma_{g},o\right)=\Gamma_{g}$). Then for all $n\in\mathbb{N}$,
\[
\mathbb{E}_{n}\left(Y\right)=\mathbb{E}_{g,n}\left[\fix_{J}\right].
\]
In particular, for $J\le_{\mathrm{f.g.}}\Gamma_{g}$, 
\[
\mathbb{E}_{n}\left(\core\left(J\right)\right)=\mathbb{E}_{g,n}\left[\fix_{J}\right].
\]
\end{lem}

\begin{proof}
In fact, the equality holds at the level of the individual representation
$\phi\in\mathbb{X}_{n}=\Hom\left(\Gamma_{g},S_{n}\right)$: the number
of morphisms $Y\to X_{\phi}$ is equal to the number of common fixed
points $\fix_{J}\left(\phi\right)$. Indeed, because the number of
common fixed points of $\phi\left(J\right)$ is the same as the number
of fixed points of any conjugate, we may assume without loss of generality
that $p_{*}\left(\pi_{1}\left(Y,y\right)\right)=J$. Now, $i\in\left[n\right]$
is a common fixed point of $\phi\left(J\right)$ if and only if $J\le\pi_{1}\left(X_{\phi},v_{i}\right)$,
where $v_{i}$ is the vertex of $X_{\phi}$ labeled $i$, and $\pi_{1}\left(X_{\phi},v_{i}\right)$
is identified with the subgroup 
\[
\left(p_{\phi}\right)_{*}\left(\pi_{1}\left(X_{\phi},v_{i}\right)\right)\le\Gamma_{g}.
\]
By standard facts from the theory of covering spaces \cite[Propositions 1.33 and 1.34]{hatcher2005algebraic},
there is a lift of $p$ to $X_{\phi}$ mapping the vertex $y$ to
$v_{i}$ if and only if (the image in $\Gamma_{g}$ of) $\pi_{1}\left(Y,y\right)$
is contained in (the image in $\Gamma_{g}$ of) $\pi_{1}\left(X_{\phi},v_{i}\right)$,
and this lift, if exists, is unique.

The statement about core surfaces follows from the fact that (the
image in $\Gamma_{g}$ of) $\pi_{1}\left(\core\left(J\right)\right)$
is conjugate to $J$ \cite[Proposition 5.3]{MPcore}. 
\end{proof}
Another type of expectation will also feature in this work. Given
a compact tiled surface $Y$, denote 
\[
\mathbb{E}_{n}^{\emb}\left(Y\right)\eqdf\mathbb{E}_{\phi\in\mathbb{X}_{g,n}}\left[\#\left\{ \mathrm{injective~morphisms}~Y\ensuremath{\to X_{\phi}}\right\} \right],
\]
where the expectation is over a uniformly random $\phi\in\mathbb{X}_{g,n}$. 

\subsection{Resolutions\label{subsec:Resolutions}}
\begin{defn}[Resolutions]
\label{def:resolutions}A resolution $\mathcal{\mathcal{R}}$ of
a tiled surface $Y$ is a collection of morphisms of tiled surfaces
\[
\mathcal{R}=\left\{ f\colon Y\to W_{f}\right\} ,
\]
such that every morphism $h\colon Y\to Z$ of $Y$ into a tiled surface
$Z$ with no boundary decomposes uniquely as $Y\stackrel{f}{\to}W_{f}\stackrel{\overline{h}}{\hookrightarrow}Z$,
where $f\in\mathcal{R}$ and $\overline{h}$ is an \uline{embedding}.
\end{defn}

The purpose of introducing resolutions is the following obvious lemma.
Recall the notation $\mathbb{E}_{n}\left(Y\right)$ and $\mathbb{E}_{n}^{\mathrm{emb}}\left(Y\right)$
from Section \ref{subsec:Expectations-and-probabilities}.
\begin{lem}
\label{lem:resolution-sum-of-probabilities}If $Y$ is a compact tiled
surface and $\mathcal{R}$ is a finite resolution of $Y$, then
\begin{equation}
\mathbb{E}_{n}\left(Y\right)=\sum_{f\in\mathcal{R}}\mathbb{E}_{n}^{\emb}\left(W_{f}\right).\label{eq:resolution}
\end{equation}
\end{lem}

Our main goal in the rest of this subsection is to prove the existence
of a finite resolution whenever we are given a compact tiled surface
$Y$ -- this is the content of Theorem \ref{thm:existence-of-resolution}
below. This resolution will consist strictly of boundary reduced tiled
surfaces $W_{f}$, and some of these will even be strongly boundary
reduced. We shall make use of Theorem \ref{thm:existence-of-resolution}
mainly for $Y$ a core surface of a finitely generated subgroup of
$\Gamma$. In this case the resolution we construct has even nicer
properties -- see Proposition \ref{prop:resolutions of core surfaces}.

Ideally, we would have liked to get a resolution where all the elements
are strongly boundary reduced. Unfortunately, such a resolution does
not always exist. For example, when $g=2$ and $\Gamma_{2}=\left\langle a,b,c,d\,\middle|\,\left[a,b\right]\left[c,d\right]\right\rangle $,
the core surface $Y=\core\left(\left\langle \left[a,b\right]\right\rangle \right)$
does not admit such a resolution as can be inferred from \cite[Figure 4.2]{MPcore}.\\

To prove the existence of a resolution with nice properties, we first
define a process which outputs a ``compromise'' between the $\br$-closure
of a tiled surface and the $\sbr$-closure, introduced in \cite[Section 4]{MPcore}. 
\begin{defn}
\label{def:growing process}Fix $\chi_{0}\in\mathbb{Z}$. Assume that
$h\colon Y\to Z$ is a morphism between tiled surfaces where $Y$
is compact and $Z$ has no boundary. Let $W_{0}$ denote the $h$-image
of $Y$ in $Z$. Set $i=0$. Perform the following algorithm we call
the \emph{growing process}:
\begin{enumerate}
\item If one of the following conditions holds:
\begin{enumerate}
\item $W_{i}$ is strongly boundary reduced, or
\item $W_{i}$ is boundary reduced and $\chi\left(W_{i}\right)<\chi_{0}$,
\end{enumerate}
terminate and return $h\colon Y\to W_{i}$.
\item Obtain $W_{i+1}$ from $W_{i}$ by adding to $W_{i}$ (the closure
of) every $4g$-gon in $Z\setminus W_{i}$ which touches along its
boundary an edge of $\partial W_{i}$ which is part of a half-block
(this includes the case of a long block), a long chain or a half-chain.
Set $i:=i+1$ and return to item $1$.
\end{enumerate}
\end{defn}

It is clear that every step of this process is deterministic. Note
that if the process ends when $W_{i}$ is strongly boundary reduced,
then $W_{i}$ is the unique smallest strongly boundary reduced tiled
surface inside $Z$ containing $W_{0}$, denoted $\sbr(Y\hookrightarrow Z)$
\cite[Section 4]{MPcore}. (In general, $\sbr(Y\hookrightarrow Z)$
is not always compact, but in this case it is.) The growing process
always terminates after finitely many steps:
\begin{lem}
\label{lem:the growing process terminates}The process described in
Definition \ref{def:growing process} always terminates.
\end{lem}

\begin{proof}
Let $\mathfrak{he}\left(W_{i}\right)$ denote the number of hanging
half-edges along the boundary of $\left(W_{i}\right)_{+}$ and consider
the triple 
\begin{equation}
\left(\d\left(W_{i}\right),\chi\left(W_{i}\right),-\mathfrak{he}\left(W_{i}\right)\right).\label{eq:triple}
\end{equation}
For every $i$, $W_{i}$ is a compact sub-surface of $Z$, and so
the three quantities are well-defined integers. We claim that at every
step in the growing process, the triple (\ref{eq:triple}) strictly
reduces with respect to the lexicographic order.

Indeed, assume we do not halt after $i$ steps, and let $O_{1},\ldots,O_{k}$
be the list of $4g$-gons in $Z\setminus W_{i}$ which are added to
$W_{i}$ in order to obtain $W_{i+1}$. By the choice of $4g$-gons,
it is clear that $\d\left(W_{i+1}\right)\le\d\left(W_{i}\right)$.
If the inequality is strict, we are done. So assume $\d\left(W_{i+1}\right)=\d\left(W_{i}\right)$.
This means that $\partial\left(W_{i}\right)$ contains no long blocks
nor long chains, so it is boundary reduced, and that the edges in
the complements in $Z$ of the half-blocks and half-chains at $\partial\left(W_{i}\right)$
all belong to $\partial\left(W_{i+1}\right)$. In other words, let
$p_{1},\ldots,p_{m}$ be these complements in $Z$: so $p_{j}$ is
either a half-block or a half-chain. The equality $\d\left(W_{i+1}\right)=\d\left(W_{i}\right)$
means that all the edges in $p_{1},\ldots,p_{m}$ belong to $\partial W_{i+1}$.

It is easy to see that in this case $\chi\left(W_{i+1}\right)\le\chi\left(W_{i}\right)$:
the number of new $4g$-gons and vertices in $W_{i+1}$ at most balances
the number of new edges. Let $V$ denote the set of internal vertices
in $p_{1},\ldots,p_{m}$ (so not at their endpoints). We have strict
inequality $\chi\left(W_{i+1}\right)<\chi\left(W_{i}\right)$ if and
only if some $v\in V$ belongs to $W_{i}$ or to two different complements
from $p_{1},\ldots,p_{m}$.

Now assume that $\d\left(W_{i+1}\right)=\d\left(W_{i}\right)$ and
$\chi\left(W_{i+1}\right)=\chi\left(W_{i}\right)$. Then $W_{i}$
is boundary reduced and each of the complements $p_{1},\ldots,p_{m}$
is a connected piece of $\partial W_{i+1}$. If $O_{j}$ touches a
half-block of $\partial W_{i}$, its annexation adds a net of $\left(2g-1\right)\left(4g-2\right)-2=8g\left(g-1\right)$
hanging half-edges. Every $4g$-gon along a half-chain of $\partial W_{i}$
also adds on average a net of $8g\left(g-1\right)$ hanging half-edges.
So if we add at least one $4g$-gon at the $\left(i+1\right)$st step,
$-\mathfrak{he}$ strictly decreases. So indeed the triple (\ref{eq:triple})
strictly decreases lexicographically in every step.

Finally, there are at most finitely many steps in which $\d\left(W_{i}\right)$
decreases, because this is a non-negative integer. So it is enough
to show there cannot be infinitely many steps in which $\d\left(W_{i}\right)$
is constant. If $\d\left(W_{i+1}\right)=\d\left(W_{i}\right)$, then
$W_{i}$ is boundary reduced. If $\chi\left(W_{i}\right)$ keeps decreasing,
then eventually we hit the bound $\chi\left(W_{i}\right)<\chi_{0}$
and halt. If $\d\left(W_{i}\right)$ and $\chi\left(W_{i}\right)$
are constant, then $\mathfrak{he}\left(W_{i}\right)$ increases constantly,
but in every tiled surface $W$, $\mathfrak{he}\left(W\right)\le4g\d\left(W\right)$,
so there cannot be infinitely many steps of this type too. This proves
the lemma.
\end{proof}
\begin{lem}
\label{lem:bound on the number of octagons added in the process}There
is a bound $B=B\left(Y\right)$, independent of $h\colon Y\to Z$,
such that in the entire growing process, at most $B=B\left(Y\right)$
$4g$-gons are added to $W_{0}$.
\end{lem}

\begin{proof}
Note that every boundary edge of $W_{0}$ is necessarily an $h$-image
of a boundary edge of $Y$, so that $\d\left(W_{i}\right)\le\d\left(W_{0}\right)\le\d\left(Y\right)$.
In every step, we add at most $\frac{\d\left(W_{i}\right)}{2g-1}\le\frac{\d\left(Y\right)}{2g-1}$
$4g$-gons. So it is enough to bound the number of steps performed
in the growing process until it terminates. There are at most $\frac{\d\left(Y\right)}{2}$
steps in which $\d\left(W_{i}\right)$ strictly decreases ($\d\left(W_{i+1}\right)<\d\left(W_{i}\right)$),
so there are at most $\frac{\d\left(Y\right)}{2}+1$ possible values
of $\d\left(W_{i}\right)$. In steps where $\d\left(W_{i}\right)$
is unchanged, $W_{i}$ is boundary reduced, so by definition $\chi\left(W_{i}\right)\ge\chi_{0}$
(otherwise, the process terminates). Let $\pi_{0}\left(Y\right)$
denote the number of connected components of $Y$. For all $i$, $W_{i}$
is a sub-surface of $Z$ with at most $\pi_{0}\left(Y\right)$ connected
components, and by the classification of surfaces, $\chi\left(W_{i}\right)\le2\pi_{0}\left(Y\right)$.
There are at most $2\pi_{0}\left(Y\right)-\chi_{0}$ steps with $\d\left(W_{i}\right)$
fixed and $\chi\left(W_{i}\right)$ strictly decreasing. Finally,
when $\d\left(W_{i}\right)$ is constant there are at most $2\pi_{0}\left(Y\right)-\chi_{0}+1$
possible values of $\chi\left(W_{i}\right)$, and if $\d\left(W_{i+1}\right)=\d\left(W_{i}\right)$
and $\chi\left(W_{i+1}\right)=\chi\left(W_{i}\right)$ then $\mathfrak{he}\left(W_{i+1}\right)\ge\mathfrak{he}\left(W_{i}\right)+8g\left(g-1\right)$
and $\mathfrak{he}\left(W_{i}\right)\le4g\d\left(W_{i}\right)\le4g\d\left(Y\right)$,
so there are at most $\frac{\d\left(Y\right)}{2\left(g-1\right)}$
steps with the same value of $\d\left(W_{i}\right)$ and $\chi\left(W_{i}\right)$.
Overall there are at most 
\begin{equation}
\frac{\d\left(Y\right)}{2}+\left(\frac{\d\left(Y\right)}{2}+1\right)\left[\left(2\pi_{0}\left(Y\right)-\chi_{0}\right)+\left(2\pi_{0}\left(Y\right)-\chi_{0}+1\right)\cdot\frac{\d\left(Y\right)}{2\left(g-1\right)}\right]\label{eq:number of steps in the growing process}
\end{equation}
steps in the growing process. Define $B\left(Y\right)$ to be $\frac{\d\left(Y\right)}{2g-1}$
times (\ref{eq:number of steps in the growing process}).
\end{proof}
We can now define the sought-after resolution for compact tiled surfaces.
\begin{defn}
\label{def:finite resolution for compact surfaces}Suppose that $Y$
is a compact tiled surface and $\chi_{0}\in\mathbb{Z}$ a fixed integer.
Define the $\chi_{0}$\emph{-resolution }of $Y$ to be the collection
\[
\mathcal{R}=\mathcal{R}\left(Y,\chi_{0}\right)=\left\{ f\colon Y\to W_{f}\right\} 
\]
obtained from all possible morphisms $h\colon Y\to Z$ from $Y$ to
a tiled surface $Z$ with no boundary via the growing process (the
process applied with parameter $\chi_{0}$).
\end{defn}

\begin{thm}
\label{thm:existence-of-resolution}Suppose $Y$ is a compact tiled
surface and $\chi_{0}\in\mathbb{Z}$ a fixed integer. The collection
$\mathcal{R}=\mathcal{R}\left(Y,\chi_{0}\right)$ from Definition
\ref{def:finite resolution for compact surfaces} is a finite resolution
of $Y$ which satisfies further
\begin{description}
\item [{R1}] for every $f\in\mathcal{R}$, the tiled surface $W_{f}$ is
compact and boundary reduced, and
\item [{R2}] for every $f\in\mathcal{R}$ with $\chi\left(W_{f}\right)\ge\chi_{0}$,
the tiled surface $W_{f}$ is strongly boundary reduced.
\end{description}
\end{thm}

\begin{proof}
By Lemma \ref{lem:the growing process terminates} and the halting
conditions of the growing process, it is clear that every such morphism
in $\mathcal{R}$ satisfies \textbf{R1} and \textbf{R2}. Given a morphism
$h\colon Y\to Z$ as in Definition \ref{def:finite resolution for compact surfaces},
$h\left(Y\right)$ (also named $W_{0}$) is a quotient of $Y$ and
therefore the number of cells in $h\left(Y\right)$ is bounded. From
Lemma \ref{lem:bound on the number of octagons added in the process}
we now conclude that that there is a bound on the number of cells
in any $W_{f}$ with $f\in\mathcal{R}$. This shows that $\mathcal{R}$
is finite as there are finitely many tiled surfaces with given bounds
on the number of cells, and finitely many morphisms between two given
compact tiled surfaces.

It remains to show that $\mathcal{R}$ is a resolution. By the way
it was constructed, it is clear that every morphism $h\colon Y\to Z$
with $\partial Z=\emptyset$, decomposes as 
\begin{equation}
Y\stackrel{f}{\to}W_{f}\hookrightarrow Z\label{eq:Wf}
\end{equation}
and that $f\in\mathcal{R}$. To show uniqueness, assume that $h$
decomposes in an additional way
\begin{equation}
Y\stackrel{\varphi}{\to}W_{\varphi}\hookrightarrow Z\label{eq:W-phi}
\end{equation}
where $W_{\varphi}$ is the result of the growing process for some
$h'\colon Y\to Z'$ with $\partial Z'=\emptyset$. We claim that (\ref{eq:Wf})
and (\ref{eq:W-phi}) are precisely the same decompositions of $h$.
Indeed, the growing process defined by $h'\colon Y\to Z'$ takes place
entirely inside $W_{\varphi}$, and does not depend on the structure
of $Z'\backslash W_{\varphi}$: in the $\left(i+1\right)$st step
of the growing process, the decision whether or not to annex more
$4g$-gons and where, depends only on the structure and boundary of
$W_{i}$. Consequently, the growing process defined by the morphism
$h'\colon Y\to Z'$ has the exact same output, in terms of the resulting
element we add to $\mathcal{R}$, as the growing process defined by
the composition $Y\stackrel{\varphi}{\to}W_{\varphi}\hookrightarrow Z$.
But because the growing process is deterministic, the latter is identical
to the growing process defined by $h\colon Y\to Z$.
\end{proof}
As mentioned above, we will use Theorem \ref{thm:existence-of-resolution}
mainly with $Y$ being a core surface. In this case, the theorem can
be strengthened as follows. Recall from Section \ref{sec:Introduction}
that given $J\fg\Gamma_{g}$, we denote by $\mog\left(J\right)$ the
set of f.g.~overgroups of $J$ with maximal Euler characteristic,
and by $\chi_{\max}\left(J\right)$ this maximal Euler characteristic.
\begin{prop}[Addendum to Theorem \ref{thm:existence-of-resolution}]
\label{prop:resolutions of core surfaces}Let $J\le_{\mathrm{f.g.}}\Gamma_{g}$
and let $\chi_{0}\in\mathbb{Z}$. Let $\mathcal{R}_{J,\chi_{0}}=\left\{ f\colon\core\left(J\right)\to W_{f}\right\} $
be the resolution $\mathcal{R}\left(\core\left(J\right),\chi_{0}\right)$
from Definition \ref{def:finite resolution for compact surfaces}.
Then $\mathcal{R}_{J,\chi_{0}}$ satisfies further the following two
properties.
\begin{description}
\item [{R3}] For every $f\in\mathcal{R}_{J,\chi_{0}}$ with $\chi\left(W_{f}\right)\ge\chi_{0}$,
the tiled surface $W_{f}$ is the core surface of some $K\le_{\mathrm{f.g.}}\Gamma_{g}$
with $J\le K$ and $f$ is the natural morphism between the two core
surfaces (the restriction of $J\backslash\tsg\to K\backslash\tsg$).
\item [{R4}] Assume that $\chi_{0}\le\chi_{\max}\left(J\right)$. Then
for every $K\in\mog\left(J\right)$, the natural morphism $\core\left(J\right)\to\core\left(K\right)$
belongs to $\mathcal{R}_{J,\chi_{0}}$.
\end{description}
\end{prop}

For $K\fg\Gamma_{g}$ we have $\chi\left(K\right)=\chi\left(\core\left(K\right)\right)$
\cite[Proposition 5.3]{MPcore}. Proposition \ref{prop:resolutions of core surfaces}
thus shows that as long as $\chi_{0}\le\chi_{\max}\left(J\right)$,
there is a bijection between the elements of $\mog\left(J\right)$
and the elements in the resolution with maximal Euler characteristic.
\begin{cor}
\label{cor:MOG is finite}For every $J\fg\Gamma_{g}$, the set $\mog\left(J\right)$
of f.g.~overgroups of maximal Euler characteristic is finite.
\end{cor}

\begin{proof}[Proof of Proposition \ref{prop:resolutions of core surfaces}]
Suppose that $f\colon\core\left(J\right)\to W_{f}$ satisfies $\chi\left(W_{f}\right)\ge\chi_{0}$.
In particular, $W_{f}$ is strongly boundary reduced by \textbf{R2}.
Let $j\in\core\left(J\right)$ be a vertex and assume without loss
of generality that $J=\pi_{1}\left(\core\left(J\right),j\right)$
(recall that if $\left(Y,y_{0}\right)$ is a based tiled surface,
we identify $\pi_{1}\left(Y,y_{0}\right)$ with the subgroup $p_{*}\left(\pi_{1}\left(Y,y_{0}\right)\right)$
of $\pi_{1}\left(\Sigma_{g},o\right)=\Gamma_{g}$). Define $K\eqdf\pi_{1}\left(W_{f},f\left(j\right)\right)$.
As $W_{f}$ is compact, $K$ is finitely generated. Let $p_{K}\colon\left(K\backslash\tsg,k\right)\to\Sigma_{g}$
be the pointed coverings space with $\pi_{1}\left(K\backslash\tsg,k\right)=K$.
By the unique lifting property from the theory of covering spaces
\cite[Propositions 1.33 and 1.34]{hatcher2005algebraic}, as $W_{f}$
is connected, there is a unique lift $\alpha$ of the restricted covering
map $p_{W_{f}}\colon W_{f}\to\Sigma_{g}$ to $K\backslash\tsg$, which
maps $f\left(j\right)$ to $k$. We will show that $\alpha$ gives
an isomorphism between $W_{f}$ and $\core\left(K\right)$.
\[
\xymatrix{ & \left(\core\left(K\right),k\right)\ar@{^{(}->}[r]^{\iota} & \left(K\backslash\tsg,k\right)\ar@{->>}[d]^{p_{K}} & \left(\core\left(K\right),k\right)\ar@{^{(}->}[r]^{\iota} & \left(K\backslash\tsg,k\right)\ar@{->}[d]^{m}\\
\left(\core\left(J\right),j\right)\ar[r]_{f}\ar[ru] & \left(W_{f},f\left(j\right)\right)\ar[r]_{p_{W_{f}}}\ar@{-->}[ur]^{\exists!\alpha} & \left(\Sigma_{g},o\right) & \left(W_{f},f\left(j\right)\right)\ar@{^{(}->}[r] & \left(Z,f\left(j\right)\right)
}
\]

First we show that $\alpha\left(W_{f}\right)\subseteq\core\left(K\right)$.
Recall that $W_{f}$ is the result of the growing process for some
$h\colon\core\left(J\right)\to Z$ with $\partial Z=\emptyset$. Consider
$W_{0}\eqdf h\left(\core\left(J\right)\right)\subseteq Z$. Recall
that $W_{f}=\sbr\left(W_{0}\hookrightarrow Z\right)$. By the unique
lifting property, $\alpha\circ f$ is the natural morphism $\core\left(J\right)\to K\backslash\tsg$,
which, by \cite[Lemma 5.4]{MPcore}, has image contained in $\core\left(K\right)$.
Hence $\alpha\left(W_{0}\right)\subseteq\core\left(K\right)$. As
$\core\left(K\right)$ is strongly boundary reduced \cite[Propositions 5.3]{MPcore},
we have that $\sbr\left(\alpha\left(W_{0}\right)\hookrightarrow K\backslash\tsg\right)$
is contained in $\core\left(K\right)$. By \cite[Lemma 4.7]{MPcore},
\[
\alpha\left(W_{f}\right)=\alpha\left(\sbr\left(W_{0}\hookrightarrow Z\right)\right)\subseteq\sbr\left(\alpha\left(W_{0}\right)\hookrightarrow K\backslash\tsg\right)\subseteq\core\left(K\right).
\]

Now $Z$ is a covering space of $\Sigma_{g}$ and we may assume it
is connected (because $\core\left(J\right)$ is). Thus $Z$ is identical
to $L\backslash\tsg$ for some $L=\pi_{1}\left(Z,h\left(j\right)\right)$.
By property \textbf{R2},\textbf{ }$W_{f}$ is strongly boundary reduced,
and so by \cite[Corollary 4.11]{MPcore} its embedding in $Z$ is
$\pi_{1}$-injective. In other words, $K\le L$, and, therefore, there
is a morphism $m\colon\left(K\backslash\tsg,k\right)\to\left(Z,f\left(j\right)\right)$.
By the unique lifting property, the composition $m\circ\alpha\colon\left(W_{f},f\left(j\right)\right)\to\left(Z,f\left(j\right)\right)$
must be identical to the embedding $\left(W_{f},f\left(j\right)\right)\hookrightarrow\left(Z,f\left(j\right)\right)$
and therefore $\alpha$ is injective. So $\alpha\left(W_{f}\right)$
is a strongly boundary reduced sub-surface of $\core\left(K\right)$
with fundamental group $K$. By \cite[Lemma 5.7]{MPcore} if follows
that $\alpha\left(W_{f}\right)\supseteq\core\left(K\right)$. We conclude
that $\alpha\colon W_{f}\to\core\left(K\right)$ is an isomorphism,
and \textbf{R3} is proven.

To prove \textbf{R4}, suppose that $K\in\mog\left(J\right)$. Let
$h\colon\core\left(J\right)\to K\backslash\tsg$ be the natural morphism.
By the definition of the resolution $\mathcal{R}_{J}$, the morphism
$h$ factors as $\core\left(J\right)\stackrel{f}{\to}W_{f}\stackrel{}{\hookrightarrow}K\backslash\tsg$
for some $f\in\mathcal{R}_{J}$. Because $h\left(\core\left(J\right)\right)\subseteq\core\left(K\right)$
(by \cite[Lemma 5.4]{MPcore}), and because $\core\left(K\right)$
is strongly boundary reduced, we have $W_{f}\subseteq\core\left(K\right)$.

Let $C$ be a connected component of the difference between the thick
version of $\core\left(K\right)$ and the thick version of $W_{f}$.
As $\core\left(K\right)$ is compact, $\overline{C}$ is compact.
As $\core\left(K\right)$ is connected, $\overline{C}$ must intersect
$\partial W_{f}$ and in particular has at least one boundary component.
As $W_{f}$ is boundary reduced, $\overline{C}$ is not homeomorphic
to a disc, and so $\chi\left(\overline{C}\right)\le0$. Now 
\[
\chi\left(K\right)=\chi\left(\core\left(K\right)\right)=\chi\left(W_{f}\right)+\sum_{C}\chi\left(\overline{C}\right),
\]
the sum being over all connected components as above. We conclude
that $\chi\left(W_{f}\right)\ge\chi\left(K\right)=\chi_{\max}\left(J\right)\ge\chi_{0}$.
By \textbf{R2}, $W_{f}$ is strongly boundary reduced and by \textbf{R3},
$W_{f}=\core\left(M\right)$ for some subgroup $M$. But then $M\in\mog\left(J\right)$,
$\chi\left(\core\left(M\right)\right)=\chi\left(K\right)$, and every
connected component $C$ as above satisfies $\chi\left(\overline{C}\right)=0$.
As $\overline{C}$ has at least one boundary component, it must be
an annulus. But then $\core\left(M\right)$ is a deformation retract
of $\core\left(K\right)$, so $M=K$ up to conjugation and so $W_{f}=\core\left(M\right)=\core\left(K\right)$. 
\end{proof}
\begin{cor}
\label{cor:resolution of a cyclic subgroup}Suppose $1\ne\gamma\in\Gamma_{g}$
is a non-trivial element. Let $q$ be the maximal natural number such
that $\gamma=\gamma_{0}^{~q}$ for some $\gamma_{0}\in\Gamma_{g}$,
and $d(q)$ the number of positive divisors of $q$. Then $\core\left(\left\langle \gamma\right\rangle \right)$
has a finite resolution $\mathcal{R}_{\gamma}=\left\{ f\colon\core\left(\left\langle \gamma\right\rangle \right)\to W_{f}\right\} $
with $W_{f}$ boundary reduced for every $f\in R_{\gamma}$, and with
exactly $d\left(q\right)$ elements $f\in\mathcal{R}_{\gamma}$ with
$\chi\left(W_{f}\right)\ge0$. Moreover, these $d\left(q\right)$
elements are precisely 
\begin{equation}
f_{m}\colon\core\left(\left\langle \gamma\right\rangle \right)\to\core\left(\left\langle \gamma_{0}^{~m}\right\rangle \right)\label{eq:elements of chi=00003D0}
\end{equation}
for $m\mid q$, where $f_{m}$ is the natural morphism between the
core surfaces.
\end{cor}

\begin{proof}
Construct $\mathcal{R}_{\gamma}=\left\{ f\colon\core\left(\left\langle \gamma\right\rangle \right)\to W_{f}\right\} $
as $\mathcal{R}\left(\core\left(\left\langle \gamma\right\rangle \right),0\right)$
from Definition \ref{def:finite resolution for compact surfaces}.
By Theorem \ref{thm:existence-of-resolution} and Proposition \ref{prop:resolutions of core surfaces},
the elements in $\mathcal{R}_{\gamma}$ with $\chi\left(W_{f}\right)=0$
are precisely the core surfaces of the subgroups in $\mog\left(\left\langle \gamma\right\rangle \right)$.
So it remains to show only that $\mog\left(\left\langle \gamma\right\rangle \right)$
are precisely $\left\langle \gamma_{0}^{~m}\right\rangle $ with $m\mid q$.

But f.g.~subgroups $K\le\Gamma_{g}$ with $\chi\left(K\right)=0$
are necessarily cyclic. Assume $K=\left\langle \delta\right\rangle \in\mog\left(\left\langle \gamma\right\rangle \right)$,
so $\gamma\in\left\langle \delta\right\rangle $, and we may assume
that $\gamma$ is a positive power of $\delta$ (otherwise switch
to $\delta^{-1}$). As every finitely generated subgroup of $\Gamma$
of infinite index is free (e.g., \cite{scott1978subgroups}), and
a subgroup of index $t$ of $\Gamma_{g}$ cannot be generated by less
then $t\left(2g-1\right)+1$ elements, we conclude that the subgroup
$\left\langle \delta,\gamma_{0}\right\rangle \le\Gamma$ is free.
Because there is a relation $\gamma_{0}^{~q}=\delta^{k}$ for some
$k\in\mathbb{N}$, it must be a cyclic subgroup. By definition, $\gamma_{0}$
is not a proper power, and so $\delta$ must be a positive power of
$\gamma_{0}$, and hence $\delta=\gamma_{0}^{~m}$ for some $m\mid q$.
\end{proof}

\section{Background: representation theory of the symmetric group\label{sec:Representation-theory-of-symmetric-group}}

In this section we give background on the complex representation theory
of $S_{n}$ that will be used in the sequel. We follow the Vershik-Okounkov
approach to the representation theory of $S_{n}$ developed in \cite{VershikOkounkov}.

\subsection{Young diagrams\label{subsec:Young-diagrams}}

A \emph{partition} is a sequence $\lambda=(\lambda_{1},\lambda_{2},\ldots,\lambda_{\ell})$
with each $\lambda_{i}\in\N$ and $\lambda_{1}\geq\lambda_{2}\geq\cdots\geq\lambda_{\ell}$.
If $\sum\lambda_{i}=n$ we write this as $\lambda\vdash n$. Such
partitions are in one-to-one correspondence with \emph{Young diagrams.}
A Young diagram (YD) consists of a collection of left-aligned rows
of identical square boxes, where the number of boxes in each row is
non-increasing from top to bottom. Given a partition $\lambda=(\lambda_{1},\lambda_{2},\ldots,\lambda_{\ell})$,
the corresponding Young diagram has $\ell$ rows and $\lambda_{i}$
boxes in the $i$\textsuperscript{th} row, where $i$ increases from
top to bottom. We think of partitions as Young diagrams, and vice
versa, freely throughout the sequel. If $\lambda$ and $\mu$ are
two Young diagrams we say $\mu\subset\lambda$ if every box of $\mu$
is a box of $\lambda$. We say $\mu\subset_{k}\lambda$ if $\mu\subset\lambda$
and $\mu$ and $\lambda$ differ by $k$ boxes.

A \emph{skew Young diagram} (SYD) is formally a pair of Young diagrams
$\mu$ and $\lambda$ with $\mu\subset\lambda$, and is denoted by
$\lambda/\mu$. We also think of $\lambda/\mu$ as a diagram consisting
of the boxes of $\lambda$ that are not in $\mu$. We can think of
a Young diagram $\lambda$ also as a skew diagram $\lambda=\lambda/\emptyset$
where $\emptyset$ is the empty diagram with no boxes. Therefore statements
that we make about skew Young diagrams apply in this way also to Young
diagrams.

The \emph{size $|\lambda/\mu|$ }of a SYD $\lambda/\mu$ is the number
of boxes that it contains, or $\sum\lambda_{i}-\sum\mu_{i}$. If $\square$
is a particular box appearing in a SYD we let $i(\square)$ be the
row number (starting at 1, counting from top to bottom) of the box
and $j(\square)$ the column number (starting at 1, counting from
left to right) of the box. The \emph{content }of a box $\square$
in a SYD is 
\[
c(\square)\eqdf j(\square)-i(\square).
\]
If $\square_{1}$ and $\square_{2}$ are two boxes in a SYD we let
\[
\ax(\square_{1},\square_{2})\eqdf c(\square_{1})-c(\square_{2}),
\]
this is called the \emph{axial distance }between $\square_{1}$ and
$\square_{2}$.

If $\lambda$ is a YD, we write $\check{\lambda}$ for the YD obtained
from $\lambda$ by swapping rows and columns, namely, by transposing.
This $\check{\lambda}$ is called the \emph{conjugate }of $\lambda$.

\subsection{Young tableaux}

Let $\lambda/\mu$ be a SYD with $\lambda\vdash n$ and $\mu\vdash m$.
A \emph{standard Young tableau} of shape $\lambda/\mu$ is a filling
of the boxes of $\lambda/\mu$ with the numbers $m+1,\ldots,n$ such
that
\begin{itemize}
\item each number appears in exactly one box of $\lambda/\mu$, and
\item the numbers in the boxes are strictly increasing from left to right
and from top to bottom.
\end{itemize}
In the sequel, we will refer to standard Young tableaux simply as
\emph{tableaux. }For $\lambda/\mu$ a SYD, we write $\Tab(\lambda/\mu)$
for the collection of tableaux of shape $\lambda/\mu$. If $\lambda\vdash n$,
$T\in\Tab(\lambda)$ and $m\in\left[n\right]$, we write $\mu_{m}\left(T\right)$
for the Young diagram obtained by deleting the boxes containing $m+1,\ldots,n$
from $T$, so $\mu_{m}\left(T\right)\vdash m$. We also write $T\lvert_{\le m}\in\Tab\left(\mu_{m}\left(T\right)\right)$
for the tableau formed by the numbers-in-boxes of $T$ that are $\leq m$,
and $T\lvert_{>m}$ for the tableau formed by the numbers-in-boxes
of $T$ that are $>m$. In general the shape of $T\lvert_{>m}$ will
be a SYD. If $T$ is a tableau of shape $\lambda/\mu$ where $\lambda\vdash n$
and $\mu\vdash m$, and $m<i\leq n$, we write $\text{\fbox{\ensuremath{i}}}_{T}$
for the box containing $i$ in $T$.

If $\lambda\vdash n$ and $\mu\subset\lambda$ then we have a concatenation
between $\Tab(\mu)$ and $\Tab(\lambda/\mu)$: if $T\in\Tab(\mu)$
and $R\in\Tab(\lambda/\mu)$, let $T\sqcup R$ be the tableau obtained
by adjoining $R$ to $T$.

\subsection{Representations of symmetric groups\label{subsec:Representations-of-symmetric}}

The irreducible unitary representations of $S_{n}$ are parameterized,
up to unitary equivalence, by Young diagrams of size $n$. This correspondence
between Young diagrams and representations is denoted by 
\[
\lambda\mapsto V^{\lambda}.
\]
Each $V^{\lambda}$ is a finite dimensional complex vector space with
a unitary action of $S_{n}$ and is also a module for the group algebra
$\C[S_{n}]$. Let $d_{\lambda}\eqdf\dim V^{\lambda}$. It is known
that $d_{\lambda}=\left|\Tab\left(\lambda\right)\right|$.

We now follow Vershik-Okounkov \cite{VershikOkounkov}. The natural
ordering of $[n]$ induces a filtration
\[
S_{1}\subset S_{2}\subset\cdots\subset S_{n-1}\subset S_{n}
\]
of $S_{n}$, where $S_{m}$ is the subgroup of $S_{n}$ fixing each
of the numbers in $\left[m+1,n\right]$. If $W$ is any unitary representation
of $S_{n}$, for $m\in[n]$ and $\mu$ a YD of size $m$, we write
$W_{\mu}$ for the linear span in $W$ of all elements in the image
of $\Hom_{S_{m}}(V^{\mu},W)$. In other words, $W_{\mu}$ is the span
of copies of $V^{\mu}$ in the restriction of $W$ to $S_{m}$. This
$W_{\mu}$ is called the\emph{ $\mu$-isotypic subspace} \emph{of}
$W$.

Vershik and Okounkov describe a specific orthonormal basis of $V^{\lambda}$,
called a \emph{Gelfand-Tsetlin basis}, that will be useful to us here.
The basis is indexed by $T\in\Tab(\lambda)$; each such $T$ gives
a basis vector $v_{T}$. The vectors $v_{T}$ can be characterized
up to multiplication by complex scalars of modulus $1$ in the following
way. The intersection of subspaces 
\[
\left(V^{\lambda}\right)_{\mu_{1}(T)}\cap\left(V^{\lambda}\right)_{\mu_{2}(T)}\cap\cdots\cap\left(V^{\lambda}\right)_{\mu_{n-1}(T)}
\]
is one-dimensional and contains the unit vector $v_{T}$ \cite[$\S$1]{VershikOkounkov}.
One important corollary of this is that if $\mu\vdash m\in[n]$, then
$\left(V^{\lambda}\right)_{\mu}\neq\{0\}$ if and only if $\mu\subset\lambda$.
Also note that if $\mu_{1},\mu_{2}\subset\lambda,$ $\mu_{1},\mu_{2}\vdash m\in[n]$,
and $\mu_{1}\neq\mu_{2}$ then $\left(V^{\lambda}\right)_{\mu_{1}}$
is orthogonal to $\left(V^{\lambda}\right)_{\mu_{2}}$.

More generally, if $\lambda/\mu$ is a SYD with $\lambda\vdash n$
and $\mu\vdash m$ then there is a \emph{skew module }$V^{\lambda/\mu}$
that is a unitary representation of $S'_{n-m}$ where we write $S'_{n-m}$
for the copy of $S_{n-m}$ in $S_{n}$ that fixes the elements $[m]$.
Formally,
\[
V^{\lambda/\mu}\eqdf\Hom_{S_{m}}\left(V^{\mu},V^{\lambda}\right)
\]
where the action of $S'_{n-m}$ is by left multiplication: for $\varphi\in\Hom_{S_{m}}\left(V^{\mu},V^{\lambda}\right)$,
$\tau\in S'_{n-m}$ and $v\in V^{\mu}$, $\left(\tau.\varphi\right)\left(v\right)\eqdf\tau.\left(\varphi\left(v\right)\right)$.
This action preserves $V^{\lambda/\mu}$ as $S'_{n-m}$ is in the
centralizer of $S_{m}$ in $\C[S_{n}]$. We write $d_{\lambda/\mu}$
for the dimension of $V^{\lambda/\mu}$. Since $d_{\lambda/\mu}$
is the multiplicity of $V^{\mu}$ in the restriction of $V^{\lambda}$
to $S_{m}$, by Frobenius reciprocity, it is also the multiplicity
of $V^{\lambda}$ in the induced representation $\mathrm{Ind}_{S_{m}}^{S_{n}}V^{\mu}$.
By calculating the dimension of $\mathrm{Ind}_{S_{m}}^{S_{n}}V^{\mu}$
in two ways, we obtain the following result that will be useful later.
\begin{lem}
\label{lem:induced-rep-dimension}Let $n\in\N$, $m\in\left[n\right]$
and $\mu\vdash m$. Then,
\[
\sum_{\lambda\vdash n\colon\mu\subset\lambda}d_{\lambda/\mu}d_{\lambda}=\frac{n!}{m!}d_{\mu}.
\]
\end{lem}

The module $V^{\lambda/\mu}$ has an orthonormal basis $w_{T}$ indexed
by $T\in\Tab(\lambda/\mu)$ \cite[Section 7]{VershikOkounkov}. One
also has the following property that we will use later \cite[eq. (3.65)]{Ceccherini-SilbersteinScarabottiTolli}.
\begin{lem}
\label{lem:tensor-product-isomorphism}Let $n\in\N$, $m\in\left[n\right]$,
$\lambda\vdash n$ and $\mu\vdash m$ and assume that $\mu\subset\lambda$.
Then the map 
\[
v_{T}\otimes w_{R}\mapsto v_{T\sqcup R},\quad T\in\Tab(\mu),\,R\in\Tab(\lambda/\mu)
\]
linearly extends to an isomorphism of unitary $\left(S_{m}\times S'_{n-m}\right)$-representations
$V^{\mu}\otimes V^{\lambda/\mu}\cong\left(V^{\lambda}\right)_{\mu}$. 
\end{lem}

There is also an explicit formula for the action of $S'_{n-m}$ on
$V^{\lambda/\mu}$. A full exposition of this formula can be found
in \cite[\S 6]{VershikOkounkov}. Recall that $S'_{n-m}$ is generated
by the \emph{Coxeter generators}

\[
s_{i}\eqdf(i~\ i+1)
\]
for $m<i<n$, where $\left(i~~i+1\right)$ is our notation for a transposition
switching $i$ and $i+1$. Therefore it is sufficient to describe
how the $s_{i}$ act on $V^{\lambda/\mu}$. Say that $T$ is admissible
for $s_{i}$ if the boxes containing $i$ and $i+1$ in $T$ are neither
in the same row nor the same column.

For $T\in\Tab(\lambda/\mu)$ let 
\begin{align*}
s_{i}T & =\begin{cases}
T & \text{if \ensuremath{T} is not admissible for \ensuremath{s_{i}}}\\
T' & \text{if \ensuremath{T} is admissible for \ensuremath{s_{i}}},
\end{cases}
\end{align*}
where $T'$ is the tableaux obtained from $T$ by swapping $i$ and
$i+1$. The admissibility condition ensures $T'$ is a valid standard
Young tableau. Then one has \emph{Young's orthogonal form}

\begin{equation}
s_{i}w_{T}=\frac{1}{\ax(\text{\fbox{\ensuremath{i+1}}}_{T},\text{\fbox{\ensuremath{i}}}_{T})}w_{T}+\sqrt{1-\frac{1}{\ax(\text{\fbox{\ensuremath{i+1}}}_{T},\text{\fbox{\ensuremath{i}}}_{T})^{2}}}w_{s_{i}T}.\label{eq:cox-action}
\end{equation}
Note that as a special case of this formula, if $T$ is not admissible
for $s_{i}$, then $\ax(\text{\fbox{\ensuremath{i+1}}}_{T},\text{\fbox{\ensuremath{i}}}_{T})=\pm1$
and 
\begin{equation}
s_{i}w_{T}=\frac{1}{\ax(\text{\fbox{\ensuremath{i+1}}}_{T},\text{\fbox{\ensuremath{i}}}_{T})}w_{T}=\begin{cases}
w_{T} & \mathrm{if}~i~\mathrm{and}~i+1~\mathrm{are~in~the~same~row},\\
-w_{T} & \mathrm{if}~i~\mathrm{and}~i+1~\mathrm{are~in~the~same~column.}
\end{cases}\label{eq:cox-gen-degen}
\end{equation}

\begin{rem}
\label{rem:S_0}For completeness of some of our statements, we need
to define the notions above also for $S_{0}$, the symmetric group
of the empty set. This is the trivial group. Whenever $\mu=\lambda$,
we have $\Tab\left(\lambda/\mu\right)=\left\{ \emptyset\right\} $,
and the representation $V^{\lambda/\mu}$ is one-dimensional with
basis $w_{T}$, for $T$ the empty tableau.
\end{rem}

\section{Preliminary representation theoretic results\label{sec:Preliminary-representation-theor}}

In this section we give some preliminary results on representation
theory that will be used in the rest of the paper. Although some results
here seem to be novel (in particular Proposition \ref{prop:bound-for-matrix-coef-in-terms-of-D}),
this section plays only a supporting role in the paper.

\subsection{Commutants}

Recall that if $V$ is a finite-dimensional vector space, and $\A$
is a subalgebra of $\End(V)$, then the \emph{commutant} of $\A$
in $\End(V)$ is the algebra of elements $b\in\End(V)$ such that
\[
ba=ab
\]
for all $a\in\A$. For $m\in[n]$ and $\lambda\vdash n$ let $Z(\lambda,m,n)$
denote the commutant of the image of $\C[S_{m}]$ in $\End(V^{\lambda}).$
We identify 
\begin{align}
\End(V^{\lambda}) & \cong V^{\lambda}\otimes\check{V^{\lambda}}\label{eq:End-isomorphism.}
\end{align}
and give $\End(V^{\lambda})$ the Hermitian inner product induced
from $V^{\lambda}.$
\begin{lem}
\label{lem:commutant}Let $m\in[n]$ and $\lambda\vdash n$. The algebra
$Z(\lambda,m,n)$ has an orthonormal basis given by
\begin{equation}
\left\{ \EE_{\mu,R_{1},R_{2}}^{\lambda}\eqdf\frac{1}{\sqrt{d_{\mu}}}\sum_{T\in\Tab(\mu)}v_{T\sqcup R_{1}}\otimes\check{v}_{T\sqcup R_{2}}\,:\,\mu\vdash m,\,\mu\subset\lambda,\,R_{1},R_{2}\in\Tab(\lambda/\mu)\right\} .\label{eq:E-def}
\end{equation}
\end{lem}

\begin{proof}
Let $\A\subseteq\End(V^{\lambda})$ be the algebra generated by the
$\EE_{\mu,R_{1},R_{2}}^{\lambda}$ (over all $\mu)$ and $\A_{\mu}\subseteq\A$
be the algebra generated by the \emph{$\EE_{\mu,R_{1},R_{2}}^{\lambda}$}
with a fixed value of $\mu$. Suppose that $Q\in\Tab(\lambda)$. The
formula for the action of $\EE_{\mu,R_{1},R_{2}}^{\lambda}$ on $v_{Q}$
is 
\begin{equation}
\EE_{\mu,R_{1},R_{2}}^{\lambda}\left(V_{Q}\right)=\mathbf{1}\left\{ \mu_{m}(Q)=\mu,\,Q\lvert_{>m}=R_{2}\right\} \frac{1}{\sqrt{d_{\mu}}}v_{Q\lvert_{\le m}\sqcup R_{1}}.\label{eq:E-action-formula-1}
\end{equation}
It is clear that the $\EE_{\mu,R_{1},R_{2}}^{\lambda}$ are an orthonormal
set of elements in $\End(V^{\lambda})$. It follows from (\ref{eq:E-action-formula-1})
that 
\[
\EE_{\mu_{1},R_{1},R_{2}}^{\lambda}\EE_{\mu_{2},R_{3},R_{4}}^{\lambda}=\mathbf{1}\left\{ \mu_{1}=\mu_{2},R_{3}=R_{2}\right\} \frac{1}{\sqrt{d_{\mu_{1}}}}\EE_{\mu_{1},R_{1},R_{4}}^{\lambda},
\]
so the $\EE_{\mu,R_{1},R_{2}}^{\lambda}$ are an orthonormal basis
for the algebra $\A$, and those with a fixed $\mu$ are an orthonormal
basis for $A_{\mu}$. Furthermore, we have 
\[
\A=\bigoplus_{\mu\subset\lambda,\mu\vdash m}\A_{\mu}.
\]

For $\mu\vdash m$, let $p_{\mu}\in\C[S_{m}]$ be the central idempotent
projection onto the $\mu$-isotypic component of $\C[S_{m}]$ and
let $P_{\mu}^{~\lambda}$ be the image of $p_{\mu}$ in $\End(V^{\lambda})$.
The element $P_{\mu}^{~\lambda}$ is the orthogonal projection onto
$\left(V^{\lambda}\right)_{\mu}$, and $P_{\mu}^{~\lambda}$ is in
the center of $Z(\lambda,m,n)$\@. Hence for every $z\in Z(\lambda,m,n)$,
we can write
\[
z=\bigoplus_{\mu\vdash m,\,\mu\subset\lambda}z^{(\mu)}
\]
where $z^{(\mu)}\eqdf P_{\mu}^{\lambda}zP_{\mu}^{\lambda}.$ Moreover,
if we let $\B_{\mu}$ be the algebra generated by the image of $\C[S_{m}]$
in $\End(V_{\mu}^{\lambda})$, each $z^{(\mu)}$ must be in the commutant
$\B'_{\mu}$ of $\B_{\mu}$ in $\End(V_{\mu}^{\lambda}$). On the
other hand, if $z=\bigoplus_{\mu\vdash m,\,\mu\subset\lambda}z^{(\mu)}$
and each $z^{(\mu)}\in\B'_{\mu}$ then $z\in Z(\lambda,m,n)$. This
shows that 
\begin{align}
Z(\lambda,m,n) & =\bigoplus_{\mu\vdash m,\,\mu\subset\lambda}\B'_{\mu}.\label{eq:split-to-sum}
\end{align}

Since $V^{\mu}$ is an irreducible module for $\C[S_{m}]$, the algebra
generated by $\C[S_{m}]$ in $\End(V^{\mu})$ is the whole of $\End(V^{\mu})$.
Hence, under the isomorphism of Lemma \ref{lem:tensor-product-isomorphism},
the algebra $\B_{\mu}$ is identified with $\End(V^{\mu})\otimes\C\mathrm{Id}_{V_{\lambda/\mu}}$.
By a classical theorem, due to Tomita \cite{Tomita} in the generality
of von Neumann algebras\footnote{This is however easy to prove in the special case here that we use
it.}, the commutant of a tensor product is the tensor product of the two
commutants. Therefore, still using the isomorphism of Lemma \ref{lem:tensor-product-isomorphism},
we have 
\[
\B'_{\mu}\cong\C\Id_{V^{\mu}}\otimes\End(V_{\lambda/\mu}).
\]
 This space is the algebra $\A_{\mu}$, so $\B'_{\mu}=\A_{\mu}$ and
\[
Z(\lambda,m,n)=\bigoplus_{\mu\vdash m,\,\mu\subset\lambda}\B'_{\mu}=\bigoplus_{\mu\vdash m,\,\mu\subset\lambda}\A_{\mu}=\A
\]
 as required. 
\end{proof}

\subsection{Bounds for the dimensions of irreducible representations\label{subsec:Bounds-for-the-dimensions-of-irreps}}

In this section we give bounds related to the dimensions of irreducible
representations that we use later. We first note a very simple bound
for the dimensions of irreducible representations of $S_{n}$. For
a YD, denote by $b_{\lambda}\eqdf\left|\lambda\right|-\lambda_{1}$
the number of boxes outside the first row.
\begin{lem}
\label{lem:simple-lemma-on-dimensions}Let $\lambda\vdash n$. Suppose
that $\lambda_{1}=n-b_{\lambda}\ge\frac{n}{2}$. Then 
\[
\binom{\lambda_{1}}{b_{\lambda}}\leq d_{\lambda}\leq n^{b_{\lambda}}.
\]
\end{lem}

\begin{proof}
The first inequality is given by Liebeck and Shalev in \cite[Lemma 2.1]{LiebeckShalev}.
To bound $d_{\lambda}$ from above, note that the standard tableaux
of shape $\lambda$ can be obtained by choosing which $\lambda_{1}$
elements of $[n]$ are in the first row (of which there are at most
$\binom{n}{b_{\lambda}}$ choices), and choosing the remaining $b_{\lambda}$
numbers' locations outside the first row, of which there are at most
$b_{\lambda}!$ choices. Hence
\[
d_{\lambda}\leq b_{\lambda}!\binom{n}{b_{\lambda}}\leq n^{b_{\lambda}}.
\]
\end{proof}
\begin{lem}
\label{lem:dim-ratio-bound}Let $\lambda\vdash n$, $\nu\subset_{k}\lambda$.
If $n\ge k+2b_{\lambda}$ then

\begin{equation}
\frac{\left(n-b_{\lambda}\right)^{b_{\lambda}}}{b_{\lambda}^{~b_{\lambda}}\left(n-k\right)^{b_{\nu}}}\leq\frac{d_{\lambda}}{d_{\nu}}\leq\frac{b_{\nu}^{~b_{\nu}}n^{b_{\lambda}}}{\left(n-k-b_{\nu}\right)^{b_{\nu}}}.\label{eq:dim-ratio-bound}
\end{equation}
\end{lem}

\begin{proof}
By assumption, $n\ge k+2b_{\lambda}\ge k+2b_{\nu}$, so $n-b_{\lambda}\ge\frac{n}{2}$
and $n-k-b_{\nu}\ge\frac{n-k}{2}$ and Lemma \ref{lem:simple-lemma-on-dimensions}
applies. The statement now follows from Lemma \ref{lem:simple-lemma-on-dimensions}
together with the inequality $\binom{p}{q}\geq(\frac{p}{q})^{q}$,
holding for $p,q\in\N$ with $p\ge q\ge1$.
\end{proof}

\subsection{An estimate for matrix coefficients in skew modules\label{subsec:An-estimate-for-mat-coefs-in-skew-modules}}

Recall that the Coxeter generators of $S_{n}$ are $s_{i}=\left(i~~i+1\right)$
for $i\in\left[n-1\right]$. If $\tau\in S_{n}$, we write $\ell_{\cox}(\tau)$
for the minimal length of a word of Coxeter generators that equals
$\tau$. Assume that $\lambda\vdash n$, $m\in\left[n\right]$ and
$\nu\vdash m$. For $T\in\Tab\left(\lambda/\nu\right)$, we write
$\tp(T)\subseteq\left[m+1,n\right]$ for the set of elements in the
top row of $T$ (which may be empty: it is of size $\lambda_{1}-\nu_{1}$).

For any two subsets $A,B$ of $[n]$, we define $d(A,B)=\left|A\setminus B\right|$.
When restricted to subsets of $\left[n\right]$ with exactly $p$
elements, for some $p\in\left[0,n\right]$, this function is a metric.
Moreover, the function $d$ is clearly invariant under $S_{n}$, that
is, if $\sigma\in S_{n}$ and $A,B\subseteq\left[n\right]$, then
$d\left(\sigma(A),\sigma(B)\right)=d\left(A,B\right).$
\begin{prop}
\label{prop:bound-for-matrix-coef-in-terms-of-D}Suppose $m\leq n$,
$\lambda\vdash n$, $\nu\vdash m$ and $\nu\subset\lambda$, and write
$k=n-m$. If $\lambda_{1}+\nu_{1}>n+k^{2}$, then for any $T,T'\in\Tab(\lambda/\nu)$
and $\sigma\in S'_{k}$ we have 
\begin{equation}
\left|\langle\sigma w_{T},w_{T'}\rangle\right|\leq\left(\frac{k^{2}}{\lambda_{1}+\nu_{1}-n}\right)^{d\left(\sigma\tp(T),\tp(T')\right)}.\label{eq:matrix-coeff-bound}
\end{equation}
\end{prop}

Note that if the top row of $\lambda/\nu$ is empty, namely, if $\nu_{1}=\lambda_{1}$,
then $\tp\left(T\right)=\emptyset$ for every $T\in\Tab\left(\lambda/\nu\right)$
and the upper bound in (\ref{eq:matrix-coeff-bound}) is trivial:
$\left(k^{2}/\left(\lambda_{1}+\nu_{1}-n\right)\right)^{0}=1$. In
particular, this is the case if $m=n$, in which case $k=0$, the
bound is $0^{0}=1$, and we have an action of the trivial group on
a one-dimensional space spanned by $w_{T}$ for $T$ the empty tableau
(see Remark \ref{rem:S_0}).
\begin{proof}
If $k=0$ the statement is trivial, so we may assume $k\ge1$. We
prove (\ref{eq:matrix-coeff-bound}) as a consequence of the following
slightly stronger statement:

\textbf{(S)} If $\lambda_{1}+\nu_{1}\ge n+\ell_{\cox}(\sigma)$, then
for any $T\in\Tab(\lambda/\nu)$, $A_{0}\subseteq\left[m+1,n\right]$
of size $\lambda_{1}-\nu_{1}$, and any unit vector $u$ in 
\[
W_{A_{0}}\eqdf\mathrm{span}\left(\left\{ w_{T'}\,\middle|\,T'\in\Tab(\lambda/\nu),\,\tp(T')=A_{0}\right\} \right)
\]
we have 
\begin{equation}
|\langle\sigma w_{T},u\rangle|\leq\left(\frac{\ell_{\cox}(\sigma)}{\nu_{1}+\lambda_{1}-n}\right)^{d(\sigma\tp(T),A_{0})}.\label{eq:matrix-coeff-bound-stronger}
\end{equation}
The proposition follows from \textbf{(S)} by using the bound $\ell_{\cox}(\sigma)\leq k^{2}$
and setting $A_{0}=\tp(T')$, $u=w_{T'}$.

Let $D\eqdf d\left(\sigma\tp(T),A_{0}\right)$. We prove \textbf{(S)
}by induction on $\ell\eqdf\ell_{\cox}(\sigma)$. The base case of
the induction is $\ell=0$. Then $\sigma=\mathrm{id}$ and 
\[
|\langle\sigma w_{T},u\rangle|=|\langle w_{T},u\rangle|=0
\]
unless $\tp(T)=A_{0}$, meaning $D=0$. On the other hand, if $D=0$
then 
\[
|\langle\sigma w_{T},u\rangle|\leq1=0^{0}=\left(\frac{\ell}{\nu_{1}+\lambda_{1}-n}\right)^{D}
\]
as required.

For the inductive step, for $\ell\geq1$ we write $\sigma=s_{j}\sigma'$
where $\ell_{\cox}(\sigma')=\ell-1$ and $j\in\left[m+1,n-1\right]$.
Two scenarios can occur.

\textbf{(i)} Suppose $s_{j}A_{0}=A_{0}$. In this case, by the definition
of the action of the Coxeter generators in (\ref{eq:cox-action}),
$s_{j}u$ is a unit vector in $W_{A_{0}}$. Also, by the invariance
of the distance function under $s_{j}$,
\[
d(\sigma'\tp(T),A_{0})=d(\sigma\tp(T),s_{j}A_{0})=d(\sigma\tp(T),A_{0})=D.
\]
The inductive hypothesis then yields
\[
|\langle\sigma w_{T},u\rangle|=|\langle\sigma'w_{T},s_{j}u\rangle|\leq\left(\frac{\ell-1}{\nu_{1}+\lambda_{1}-n}\right)^{D}\leq\left(\frac{\ell}{\nu_{1}+\lambda_{1}-n}\right)^{D},
\]
 as required.

\textbf{(ii) }Suppose otherwise that $s_{j}A_{0}\neq A_{0}$. This
means that exactly one of $j$ and $j+1$ are in $A_{0}$. We write
\begin{equation}
u=\sum_{T'\in\Tab(\lambda/\nu)\colon\tp(T')=A_{0}}\beta_{T'}w_{T'}.\label{eq:expansion-of-u}
\end{equation}
For each $T'$ with $\tp(T')=A_{0}$, we have
\[
|\ax(\text{\fbox{\ensuremath{j+1}}}_{T'},\text{\fbox{\ensuremath{j}}}_{T'})|\geq\nu_{1}+\lambda_{1}-n.
\]

From (\ref{eq:expansion-of-u}) and the formula for the action of
Coxeter generators (\ref{eq:cox-action}) we can therefore write $s_{j}u=w_{1}+w_{2}$
where $w_{1}\in W_{s_{j}A_{0}}$ and $w_{2}\in W_{A_{0}}$ are orthogonal
vectors with $\|w_{1}\|\leq1$ and $\|w_{2}\|\leq(\nu_{1}+\lambda_{1}-n)^{-1}$.
Hence
\[
\left|\langle\sigma w_{T},u\rangle\right|=\left|\langle\sigma'w_{T},s_{j}u\rangle\right|\leq\left|\langle\sigma'w_{T},w_{1}\rangle\right|+\left|\langle\sigma'w_{T},w_{2}\rangle\right|.
\]
Note that $d(\sigma'\tp(T),s_{j}A_{0})=d\left(\sigma\tp(T),A_{0}\right)=D$,
so by the inductive hypothesis
\begin{equation}
|\langle\sigma w_{T},u\rangle|\leq\left(\frac{\ell-1}{\nu_{1}+\lambda_{1}-n}\right)^{D}+\frac{1}{\nu_{1}+\lambda_{1}-n}\left(\frac{\ell-1}{\nu_{1}+\lambda_{1}-n}\right)^{d(\sigma'\tp(T),A_{0})}.\label{eq:matrix-bound-temp}
\end{equation}
By the triangle inequality, 
\begin{align*}
D-1=d(\sigma'\tp(T),s_{j}A_{0})-d(s_{j}A_{0},A_{0})\leq d(\sigma'\tp(T),A_{0}),
\end{align*}
so using $\nu_{1}+\lambda_{1}\geq n+\ell$ we obtain from (\ref{eq:matrix-bound-temp})
\begin{align*}
|\langle\sigma w_{T},u\rangle| & \leq\left(\frac{\ell-1}{\nu_{1}+\lambda_{1}-n}\right)^{D}+\frac{1}{\nu_{1}+\lambda_{1}-n}\left(\frac{\ell-1}{\nu_{1}+\lambda_{1}-n}\right)^{D-1}\\
 & =\frac{(\ell-1)^{D-1}\ell}{(\nu_{1}+\lambda_{1}-n)^{D}}\leq\left(\frac{\ell}{\nu_{1}+\lambda_{1}-n}\right)^{D},
\end{align*}
as required.
\end{proof}

\subsection{Families of Young diagrams and zeta functions\label{subsec:Families-of-YD-and-zeta}}

Recall from $\S$\ref{sec:Introduction} that the zeta function of
$S_{n}$ is defined by
\[
\zeta^{S_{n}}\left(s\right)\eqdf\sum_{\lambda\vdash n}\frac{1}{d_{\lambda}^{~s}},
\]
and that for $g\ge2$
\[
\left|\mathbb{X}_{g,n}\right|=\left|\Hom\left(\Gamma_{g},S_{n}\right)\right|=\left(n!\right)^{2g-1}\zeta^{S_{n}}\left(2g-2\right).
\]

Let $\Lambda(n,b)$ denote the collection of $\lambda\vdash n$ such
that $b_{\lambda}\ge b$ and $b_{\check{\lambda}}\ge b$. In other
words, $\Lambda(n,b)$ is the collection of Young diagrams of size
$n$ with at least $b$ boxes outside the first row and at least $b$
boxes outside the first column. One has the following useful result
of Liebeck and Shalev \cite[Prop. 2.5]{LiebeckShalev} and, independently,
Gamburd \cite[Prop. 4.2]{gamburd2006poisson}:
\begin{prop}
\label{prop:Liebeck-Shalev-non-effective}For fixed $b\geq0$ and
real $s>0$, as $n\to\infty$
\[
\sum_{\lambda\in\Lambda(n,b)}\frac{1}{d_{\lambda}^{~s}}=O_{b}\left(n^{-sb}\right).
\]
\medskip{}
\end{prop}

The proof of Theorem \ref{thm:rational-approx} will crucially depend
on certain families of Young diagrams that interact nicely with the
skew modules $V^{\lambda/\nu}$. Given a YD $\lambda$, we will write
$\lambda(n)$ for the unique YD $\lambda(n)\vdash n$ which is obtained
from $\lambda$ by either deleting boxes from or adding boxes to the
first row of $\lambda$, if it exists. To be precise, $\lambda(n)$
exists if and only if $n\geq|\lambda|-\left(\lambda_{1}-\lambda_{2}\right)$,
interpreting $\lambda_{2}=0$ if $\lambda$ only has one row.

Now given $k\in\N$, and YDs $\nu\subset_{k}\lambda$, assume that
$n_{1}$ and $n_{2}$ are large enough so that $\lambda(n_{i})$ and
$\nu(n_{i}-k)$ both exist and so that the first row (of length $\lambda_{1}-\nu_{1}$,
which could be zero) of the SYD $\lambda\left(n_{i}\right)/\nu\left(n_{i}-k\right)$
does not border the second row, namely, $\nu\left(n_{i}-k\right)_{1}\ge\lambda_{2}$.
Then there is a natural way to identify $\Tab\left(\lambda(n_{1})/\nu(n_{1}-k)\right)$
with $\Tab\left(\lambda(n_{2})/\nu(n_{2}-k)\right)$ by simply adding
$n_{2}-n_{1}$ to all numbers in boxes of a tableau in $\Tab\left(\lambda(n_{1})/\nu(n_{1}-k)\right)$
and shifting the first row right or left as needed. If $\nu_{1}\ge\lambda_{2}$
and $T\in\Tab\left(\lambda/\nu\right)$, we write $T(n)$ for the
resulting tableau in $\Tab\left(\lambda(n)/\nu(n-k)\right)$. 

\begin{figure}
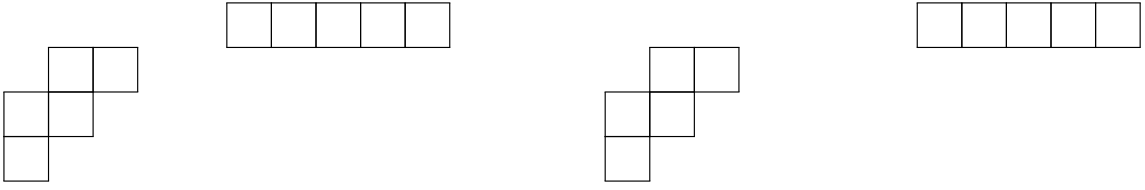

\begin{centering}
\ydiagram{5+5,1+2,2,1} ~~~~~~~~~~~~~~~~\ydiagram{7+5,1+2,2,1}
\par\end{centering}
\caption{This figure depicts two elements of a family of SYDs $\lambda(n)/\nu(n-10)$
for $n=16$ and $n=18$. Here we can take $\lambda=(6,3,2,1)$ and
$\nu=(1,1)$.}
\end{figure}

\medskip{}
Given $n\in\N$ and $k\in\left[n\right]$, recall that we write $S'_{k}$
for the subgroup of $S_{n}$ that acts as the identity on $[n-k]$.
Throughout the paper we fix isomorphisms
\begin{equation}
S_{k}\xrightarrow{\approx}S'_{k},\quad\sigma\mapsto\rho_{k}^{-1}\circ\sigma\circ\rho_{k}\label{eq:Sk-Skdash-isomorphism}
\end{equation}
where we view $S_{k}\leq S_{n}$ in the usual way and
\[
\rho_{k}(i)=\begin{cases}
i+k & \text{if \ensuremath{i\in[n-k]}}\\
i-n+k & \text{if \ensuremath{i\in[n-k+1,n]}.}
\end{cases}
\]
Using these isomorphisms allows us to identify the different subgroups
$S'_{k}$ as $n$ varies: this will recur at several points of the
sequel. It also allows us to note in the following proposition that
matrix coefficients of skew modules $V^{\lambda(n)/\nu(n-k)}$ are
holomorphic functions of $n^{-1}$, for sufficiently large $n$. Recall
that $w_{T_{i}(n)}$ are elements of the Gelfand-Tsetlin basis for
$V^{\lambda(n)/\nu(n-k)}$.
\begin{prop}
\label{prop:holomorphic-matrix-coefficients}Let $k\in\N$, $\sigma\in S_{k}$,
and $\nu\subset_{k}\lambda$ be two Young diagrams that differ by
$k$ boxes. Suppose that $\nu_{1}\ge\lambda_{2}$. Given $T_{1},T_{2}\in\Tab(\lambda/\nu)$
there is a function $F=F_{\sigma,\lambda,\nu,T_{1},T_{2}}$ that is
holomorphic in the ball of radius $\left|\lambda\right|^{-1}$ around
zero, has Taylor expansion around $0$ with rational coefficients,
and such that for all $n\geq|\lambda|$, viewing $\sigma$ as an element
of $S'_{k}\leq S_{n}$ via the isomorphism (\ref{eq:Sk-Skdash-isomorphism}),
\[
\left\langle \sigma w_{T_{1}(n)},w_{T_{2}(n)}\right\rangle =F\left(n^{-1}\right).
\]
\end{prop}

\begin{proof}
Since $V^{\lambda(n)/\nu(n-k)}$ is finite dimensional with dimension
independent of $n$ (as long as $n\ge\left|\lambda\right|$), it suffices
to prove the result in the case that $\sigma\in S_{k}$ is a Coxeter
generator $s_{i}$ with $i\in[k-1]$. Interpreted as an element of
$S'_{k}\leq S_{n}$ via (\ref{eq:Sk-Skdash-isomorphism}), $\sigma$
corresponds to the Coxeter generator $s_{i+n-k}\in S_{n}$.

Let $a$ be the axial distance between $j\eqdf i+\left|\lambda\right|-k$
and $j+1=i+1+\left|\lambda\right|-k$ in $T_{1}$. Note that
\[
\ax\left(\text{\fbox{\ensuremath{i+n-k}}}_{T_{1}\left(n\right)},\text{\fbox{\ensuremath{i+1+n-k}}}_{T_{1}\left(n\right)}\right)=\begin{cases}
a+n-\left|\lambda\right| & \mathrm{if}~j~\mathrm{in~the~first~row~of}~T_{1}~\mathrm{and}~j+1~\mathrm{not},\\
a-\left(n-\left|\lambda\right|\right) & \mathrm{if}~j+1~\mathrm{in~the~first~row~of}~T_{1}~\mathrm{and}~j~\mathrm{not},\\
a & \mathrm{otherwise}.
\end{cases}
\]
In the first case $a>0$ and in the second case $a<0$. By the description
in (\ref{eq:cox-action}) of how the Coxeter generators act, $\langle s_{i+n-k}w_{T_{1}(n)},w_{T_{2}(n)}\rangle$
is therefore one of the following functions of $n$:
\[
0,\,\frac{1}{a-n+|\lambda|},\,\frac{1}{a+n-|\lambda|},\,\sqrt{1-\frac{1}{(a-n+|\lambda|)^{2}}},\,\sqrt{1-\frac{1}{(a+n-|\lambda|)^{2}}}.
\]
If one replaces $n$ by $z^{-1}$, each of these yields a holomorphic
function of $z$ when $|z|$ is sufficiently small. 
\end{proof}
The dimensions of representations in a family $\lambda(n)$ are polynomials
in $n$:
\begin{lem}
\label{lem:dim-polynomial}Given a Young diagram $\lambda$, consider
the family of Young diagrams $\lambda(n)$. There is a polynomial
$G=G_{\lambda}\in\mathbb{Q}\left[t\right]$ of degree $b_{\lambda}$
with rational coefficients such that for every $n$ such that $\lambda\left(n\right)$
exists, 
\[
d_{\lambda(n)}=G(n).
\]
Furthermore, the complex zeros of $G$ are integers $n$ with $n\in\left[0,\left|\lambda\right|\right]$,
and the leading coefficient is $\frac{1}{m}$ for some integer $m$.
\end{lem}

For example, if $\lambda\left(n\right)=\left(n-4,3,1\right)$, then
$d_{\lambda\left(n\right)}=\frac{n\left(n-1\right)\left(n-3\right)\left(n-6\right)}{8}$
for every $n\ge7$.
\begin{proof}
This easily follows from the hook-length formula for the dimension
$d_{\lambda}$ \cite{FRT}.
\end{proof}
Lemma \ref{lem:dim-polynomial} together with Proposition \ref{prop:Liebeck-Shalev-non-effective}
have the following nice consequence for the zeta function $\zeta^{S_{n}}$
that will be crucial in proving Theorem \ref{thm:rational-approx}.
\begin{prop}
\label{prop:zeta-polynomial}For any $s\in\N$ and $M\in\N$, there
is a polynomial $P_{s,M}\in\Z\left[t\right]$ with integer coefficients
of degree $<M$ such that 
\[
\zeta^{S_{n}}\left(s\right)=2\cdot P_{s,M}\left(n^{-1}\right)+O\left(n^{-M}\right)
\]
as $n\to\infty$. The constant coefficient of $P_{s,M}$ is equal
to $1$.
\end{prop}

For example, for $s=2$ and $M=5$ we have 
\[
\zeta^{S_{n}}\left(2\right)=2\left(1+\frac{1}{n^{2}}+\frac{2}{n^{3}}+\frac{11}{n^{4}}\right)+O\left(\frac{1}{n^{5}}\right).
\]

\begin{proof}
Fix $s\in\N$ and $M\in\N$ as in the statement of the proposition.
Let $b=\lceil\frac{M}{s}\rceil$. Proposition \ref{prop:Liebeck-Shalev-non-effective}
implies that 
\[
\zeta^{S_{n}}\left(s\right)=\sum_{\substack{\lambda\vdash n\\
\lambda\notin\Lambda(n,b)
}
}\frac{1}{d_{\lambda}^{~s}}+O(n^{-M})
\]
as $n\to\infty$. The $\lambda$ in the sum above have either $<b$
boxes outside their first row, or $<b$ boxes outside their first
column. For $n>2b$, these options are mutually exclusive. Moreover,
the map $\lambda\mapsto\check{\lambda}$ maps YDs of the first kind
to YDs of the second kind bijectively, and vice versa. Hence if we
let $\Lambda_{b_{\lambda}<b}(n)$ be the collection of $\lambda\vdash n$
with $<b$ boxes outside their first row, then as $n\to\infty$,
\begin{align*}
\zeta^{S_{n}}\left(s\right) & =\sum_{\substack{\lambda\in\Lambda_{b_{\lambda}<b}(n)}
}\left(\frac{1}{d_{\lambda}^{~s}}+\frac{1}{d_{\check{\lambda}}^{~s}}\right)+O\left(n^{-M}\right)\\
 & =2\sum_{\substack{\lambda\in\Lambda_{b_{\lambda}<b}(n)}
}\frac{1}{d_{\lambda}^{~s}}+O\left(n^{-M}\right).
\end{align*}
Now, if $n>2b$, there is a finite collection of YDs $\{\mu_{1},\ldots,\mu_{\ell}\}$,
depending on $b$, with $|\mu_{i}|<2b$ for all $i$, such that for
each $n>2b$
\[
\Lambda_{b_{\lambda}<b}(n)=\left\{ \mu_{1}(n),\mu_{2}(n),\ldots,\mu_{\ell}(n)\right\} .
\]
For each of these $\mu_{i}$, let $G_{\mu_{i}}$ be the polynomial
provided by Lemma \ref{lem:dim-polynomial}. No $G_{\mu_{i}}$ has
any zero $z$ with $|z|>2b$. Hence 
\[
\zeta^{S_{n}}\left(s\right)=2\sum_{\substack{i=1}
}^{\ell}\frac{1}{\left(G_{\mu_{i}}\left(n\right)\right)^{s}}+O\left(n^{-M}\right)
\]
as $n\to\infty$. Because of the special structure of $G_{\mu_{i}}$,
as elaborated in Lemma \ref{lem:dim-polynomial}, $\left(G_{\mu_{i}}\left(n\right)\right)^{-1}$
is equal to a power series in $n^{-1}$ with integer coefficients.
Since $s\in\N$, $\left(G_{\mu_{i}}\left(n\right)\right)^{-s}$ is
too equal to a power series in $n^{-1}$ with integer coefficients.
This proves the first statement. Because the degree of $G_{\mu_{i}}\left(n\right)$
is positive unless $\mu_{i}\left(n\right)=\left(n\right)$ in which
case $G_{\mu_{i}}\left(n\right)=1$, the constant coefficient of $P_{s,M}$
must be $1$.
\end{proof}
In fact, it is the following direct corollary of Proposition \ref{prop:zeta-polynomial}
that we will need.
\begin{cor}
\label{cor:zeta-inv-poly}For any $s\in\N$ and $M\in\N$, there is
a polynomial $Q_{s,M}\in\Z\left[t\right]$ of degree $<M$ and constant
coefficient $1$ such that as $n\to\infty$,
\[
\frac{1}{\zeta^{S_{n}}\left(s\right)}=\frac{1}{2}Q_{s,M}\left(n^{-1}\right)+O\left(n^{-M}\right).
\]
\end{cor}

\section{The probability of an embedded tiled surface\label{sec:The-probability-of-an-embedded-tiled-surface}}

\subsection{Overview of this section\label{subsec:Overview-of-this}}

This short overview is meant to make the results of this section more
transparent and to stress an analogy with known results about the
zeta function of $S_{n}$. For simplicity, we assume $g=2$ throughout
this Section $\S$\ref{sec:The-probability-of-an-embedded-tiled-surface}
and denote $\Gamma=\Gamma_{2}=\left\langle a,b,c,d\,\middle|\,\left[a,b\right]\left[c,d\right]\right\rangle $.

As explained in Section $\S$\ref{sec:Introduction}, 
\begin{equation}
\left|\X_{n}\right|=\left|\X_{2,n}\right|=\left(n!\right)^{3}\cdot\sum_{\lambda\vdash n}\frac{1}{d_{\lambda}^{~2}}.\label{eq:zeta, again}
\end{equation}
If $\left\{ \lambda\left(n\right)\right\} _{n\ge n_{0}}$ is a family
of Young diagrams obtained by extending the first row, as in Section
\ref{subsec:Families-of-YD-and-zeta}, then $d_{\lambda\left(n\right)}$
is a polynomial in $n$ of degree $b_{\lambda}$ (Lemma \ref{lem:dim-polynomial}),
and so the contribution of $\lambda\left(n\right)$ and of $\check{\lambda\left(n\right)}$
to (\ref{eq:zeta, again}) is a rational function in $n$ for every
$n\ge n_{0}$. Proposition \ref{prop:Liebeck-Shalev-non-effective}
(due to \cite{LiebeckShalev,gamburd2006poisson}) states that up to
order $O\left(n^{-2b}\right)$, the zeta function in (\ref{eq:zeta, again})
is determined by those families of Young diagrams with $b_{\lambda}<b$
and their transpose. 

In this Section \ref{sec:The-probability-of-an-embedded-tiled-surface}
we prove something analogous for $\E_{n}^{\emb}(Y)$ where $Y$ is
a compact tiled surface. We write $\v=\v(Y)$, $\e=\e(Y)$, $\f=\f(Y)$
for the number of vertices, edges, and octagons of $Y$, respectively.
We will use the letter $f$ for an element of $\{a,b,c,d\}$ and write
$\e_{f}=\e_{f}(Y)$ for the number of $f$-labeled edges of $Y$.
The first major result is that, depending on a choice of four constant
permutations $\sigma_{f}^{\pm},\tau_{f}^{\pm}\in S'_{\v}$ per letter
(defined in Section \ref{subsec:Construction-of-auxiliary-numberings}),
we have 
\begin{equation}
\mathbb{E}_{n}^{\mathrm{emb}}\left(Y\right)=\frac{\left(n!\right)^{3}}{\left|\X_{n}\right|}\cdot\frac{\left(n\right)_{\v}\left(n\right)_{\f}}{\prod_{f}\left(n\right)_{\e_{f}}}\cdot\sum_{\nu\vdash n-\v}H_{Y}\left(\nu\right),\label{eq:overview of section 4, 1}
\end{equation}
where $H_{Y}\left(\nu\right)$ is some explicit function. This follows
from Theorem \ref{thm:E_n-emb-exact-expression}. Notice that as $n\to\infty$,
the first fraction in (\ref{eq:overview of section 4, 1}) is $\frac{\left(n!\right)^{3}}{\left|\X_{n}\right|}=\frac{1}{2}+O\left(n^{-2}\right)$
by Proposition \ref{prop:Liebeck-Shalev-non-effective}, and the second
one is $\frac{\left(n\right)_{\v}\left(n\right)_{\f}}{\prod_{f}\left(n\right)_{\e_{f}}}=n^{\chi\left(Y\right)}\left(1+O\left(n^{-1}\right)\right)$.
So (\ref{eq:overview of section 4, 1}) gives that 
\begin{equation}
\mathbb{E}_{n}^{\mathrm{emb}}\left(Y\right)=\left(\frac{1}{2}+O\left(n^{-1}\right)\right)n^{\chi\left(Y\right)}\cdot\sum_{\nu\vdash n-\v}H_{Y}\left(\nu\right).\label{eq:overview of section 4, 2}
\end{equation}
Next, our analysis shows that by considering, as above, \emph{families}
of Young diagrams $\left\{ \nu\left(n\right)\right\} _{n\ge n_{0}}$,
then for large enough $n$, $H_{Y}\left(\nu\left(n\right)\right)$
is equal to a converging series $\sum_{j=-\infty}^{K}\beta_{j}n^{j}$,
with $K=K\left(Y,\nu\right)$ some integer. Section \ref{subsec:Approximating Xi by Laurent}
then shows that for any given $M$, there is finite set of families
$\nu\left(n\right)$, with $b_{\nu}$ and $b_{\check{\nu}}$ bounded,
such that all remaining summands in $\sum_{\nu\vdash n-\v}H_{Y}\left(\nu\right)$
outside these families contribute jointly $O\left(n^{-M}\right)$
-- this is analogous to Proposition \ref{prop:Liebeck-Shalev-non-effective}.
Because every tiled surface admits finite resolutions as in Section
\ref{subsec:Resolutions}, this quickly leads to the proof of Theorem
\ref{thm:rational-approx} in Section \ref{sec:Proofs-of-main-theorems}.

\medskip{}

In fact, the analysis so far could have been carried out with graphs
(core graphs à la Stallings) rather than with tiled surfaces. The
importance of tiled surfaces and, moreover, of (strongly) boundary
reduced tiled surfaces, is in our ability to determine the order of
magnitude of $H_{Y}\left(\nu\right)$. Our analysis here culminates
in Proposition \ref{prop:Upsilon-non-trivial-bound} and Section \ref{subsec:Estimating-Xi},
from which it follows that when $Y$ is boundary reduced, 
\[
H_{Y}\left(\nu\left(n\right)\right)=\frac{1}{d_{\nu}^{~2}}\cdot O\left(1\right)
\]
as $n\to\infty,$ and when $Y$ is strongly boundary reduced, 
\begin{equation}
H_{Y}\left(\nu\left(n\right)\right)=\frac{1}{d_{\nu}^{~2}}\left(1+O\left(\frac{1}{n}\right)\right).\label{eq:overview of section 4, SBR}
\end{equation}
This shows that the analysis of the zeta function in (\ref{eq:zeta, again})
can be viewed as a special case of our results. Indeed, when $Y=Y_{\emptyset}$
is the empty tiled surface (which is, in particular, strongly boundary
reduced), (\ref{eq:overview of section 4, 1}) together with (\ref{eq:overview of section 4, SBR})
yield that 
\[
\left|\X_{n}\right|=\left|\X_{n}\right|\cdot\mathbb{E}_{n}^{\mathrm{emb}}\left(Y_{\emptyset}\right)=\left(n!\right)^{3}\cdot\sum_{\nu\vdash n}H_{Y_{\emptyset}}\left(\nu\right)=\left(n!\right)^{3}\cdot\sum_{\nu\vdash n}\frac{1}{d_{\nu}^{~2}}\left(1+O\left(\frac{1}{n}\right)\right).
\]
What we achieve here is the extension of this result to general strongly
boundary reduced tiled surfaces, with an extra factor of $n^{\chi\left(Y\right)}$
appearing. If $Y$ is merely boundary reduced, we obtain the same
result up to multiplicative constants.

\subsubsection*{A remark about composing permutations }

A technical but important remark is due. The bijection
\begin{align*}
\phi & \mapsto X_{\phi}\\
\Hom(\Gamma,S_{n}) & \to\{\text{degree-\ensuremath{n} covers of \ensuremath{\Sigma_{2}\}} }
\end{align*}
described previously involves the version of $S_{n}$ where permutations
are composed with the left-most permutation acting first. On the other
hand, since in the rest of the paper, we work with permutations in
detail, and specifically with the representation theory of $S_{n}$,
it is more standard to view permutations as functions from $[n]$
to $[n]$ and hence to multiply according to composition of functions
(functions acting from the left). \emph{So in the rest of the paper,
permutations will be composed with the right-most permutation acting
first. }These two versions of $S_{n}$ are isomorphic, of course,
and one isomorphism is given by $\mathrm{inv}\colon S_{n}\to S_{n}$
defined by $\sigma\mapsto\sigma^{-1}$. 

Accordingly, by post-multiplication with $\mathrm{inv}$, we turn
a homomorphism 
\[
\phi\in\mathbb{X}_{n}=\Hom\left(\Gamma_{2},S_{n}~{\scriptstyle \mathrm{(left-to-right~version)}}\right)
\]
into a homomorphism 
\[
\overline{\phi}\eqdf\mathrm{inv\circ\phi}\in\Hom\left(\Gamma_{2},S_{n}~{\scriptstyle \mathrm{(right-to-left~version)}}\right).
\]
The homomorphism $\overline{\phi}$ satisfies $\overline{\phi}\left(\gamma\right)=\phi\left(\gamma\right)^{-1}$
for every $\gamma\in\Gamma_{2}$. In particular, with composition
of permutations from right to left, the four permutations $\phi\left(a\right),\phi\left(b\right),\phi\left(c\right),\phi\left(d\right)\in S_{n}$
satisfy
\[
\left[\phi\left(a\right)^{-1},\phi\left(b\right)^{-1}\right]\left[\phi\left(c\right)^{-1},\phi\left(d\right)^{-1}\right]=\left[\overline{\phi\left(a\right)},\overline{\phi\left(b\right)}\right]\left[\overline{\phi\left(c\right)},\overline{\phi\left(d\right)}\right]=1,
\]
or, equivalently (taking the inverse of the resulting permutation),
\[
\left[\phi(d)^{-1},\phi(c)^{-1}\right]\left[\phi(b)^{-1},\phi(a)^{-1}\right]=1.
\]
This means that the word $\left[d^{-1},c^{-1}\right]\left[b^{-1},a^{-1}\right]$
will appear below at several points. Note that the image of $\gamma\in\Gamma$
under $\overline{\phi}$ is the inverse of $\phi\left(\gamma\right)$.
But since a permutation and its inverse have the same cycle-structure
in $S_{n}$, this does not affect the statistics we study in this
paper.

\subsection{Tiled surfaces and cosets\label{subsec:Tiled-surfaces-and-double-cosets}}

We assume that $Y$ is a compact tiled surface. In this section we
assume $n\in\N$ with $n\geq\v$. We fix an arbitrary bijection $\J:Y^{(0)}\to[\v]$,
and view $(Y,\J)$ as fixed in this $\S$\ref{sec:The-probability-of-an-embedded-tiled-surface}.
We modify $\J$ slightly for technical reasons\footnote{The reason for using this modification comes from a convention in
the representation theoretic methods we use below. } by letting
\begin{equation}
\J_{n}:Y^{(0)}\to[n-\v+1,n],\quad\J_{n}(v)\eqdf\J(v)+n-\v.\label{eq:shifted-framing}
\end{equation}
Then $(Y,\J_{n})$ is a vertex-labeled tiled surface for each $n$.
We are interested in the quantity $\E_{n}^{\emb}(Y)$, but because
the uniform measure on $\X_{n}$ is invariant under conjugation by
$S_{n}$, and $S_{n}$ acts transitively on ordered tuples of size
$\v$ in $[n]$, we have
\begin{equation}
\E_{n}^{\emb}(Y)=\frac{n!}{\left(n-\v\right)!}\frac{|\X_{n}(Y,\J_{n})|}{|\X_{n}|}\label{eq:Enemb-to-Xn(Y,J)}
\end{equation}
where 
\[
\X_{n}(Y,\J_{n})\eqdf\left\{ \phi\in\X_{n}\,:\,\text{there is an embedding }Y\ensuremath{\hookrightarrow X_{\phi}}~\text{inducing~\ensuremath{\J_{n}}}\right\} .
\]
(Recall from $\S\S$\ref{subsec:Expectations-and-probabilities} that
the vertices of $X_{\phi}$ are labeled by $[n]$. Also, recall that
an embedding $Y\ensuremath{\hookrightarrow X_{\phi}}$ inducing $\J_{n}$,
if it exists, is unique.) Hence we are interested in the size of the
set $\X_{n}(Y,\J_{n})$. \emph{Henceforth, we use the map $\J_{n}$
and the previous labelings of the vertices of $X_{\phi}$ to identify
the vertex sets of $Y$ and $X_{\phi}$ with subsets of $\N$.} 

For each letter $f\in\{a,b,c,d\}$, let $\V_{f}^{-}=\V_{f}^{-}(Y)\subset[n-\v+1,n]$
be the subset of vertices of $Y$ with outgoing $f$-labeled edges,
and $\V_{f}^{+}\subset[n-\v+1,n]$ those vertices of $Y$ with incoming
$f$-labeled edges. Note that $\e_{f}=|\V_{f}^{\pm}|$. We let $G_{f}$
denote the subgroup of $S_{n}$ that fixes $\V_{f}^{-}$ and write
\begin{align*}
G & \eqdf G_{a}\times G_{b}\times G_{c}\times G_{d}\le S_{n}^{~4}.
\end{align*}
For each $f\in\{a,b,c,d\}$ we let $g_{f}^{0}\in S_{n}$ be a fixed
element with the property that for every pair of vertices $i,j$ of
$Y$ (so $i,j\in[n-\v+1,n]$) with a directed $f$-labeled edge from
$i$ to $j$, we have $g_{f}^{0}(i)=j$. Recall the notation $S'_{\v}$
for the subgroup of $S_{n}$ fixing $[n-\v]$ pointwise. We choose
the $g_{f}^{0}$ consistently for each $n$ in the sense that $g_{f}^{0}$
is chosen when $n=\v$ and then defined for arbitrary $n$ by the
isomorphisms $S_{\v}\cong S'_{\v}$ given in (\ref{eq:Sk-Skdash-isomorphism}).
We write $g^{0}\eqdf(g_{a}^{0},g_{b}^{0},g_{c}^{0},g_{d}^{0})$. Notice
that $g_{f}^{0}(\V_{f}^{-})=\V_{f}^{+}$.

\emph{In the rest of the paper, whenever we write an integral over
a group, it is performed with respect to the uniform probability measure.
}Let
\begin{equation}
\Theta_{\lambda}\left(Y,\J_{n}\right)\eqdf\int_{h_{f}\in G_{f}}\chi_{\lambda}\left(\left[\left(g_{d}^{0}h_{d}\right)^{-1},\left(g_{c}^{0}h_{c}\right)^{-1}\right]\left[\left(g_{b}^{0}h_{b}\right)^{-1},\left(g_{a}^{0}h_{a}\right)^{-1}\right]\right)\label{eq:Phi_lambda}
\end{equation}
where $\chi_{\lambda}$ is the character of $S_{n}$ corresponding
to the irreducible representation $V^{\lambda}$. We will calculate
the size of $\X_{n}(Y,\J_{n})$ using the following result.

\begin{prop}
\label{prop:probabilty-as-sum-over-irreps}We have 
\[
|\X_{n}(Y,\J_{n})|=\frac{\prod_{f\in a,b,c,d}(n-\e_{f})!}{n!}\sum_{\substack{\lambda\vdash n}
}d_{\lambda}\Theta_{\lambda}\left(Y,\J_{n}\right).
\]
\end{prop}

\begin{proof}
We begin by observing that with $g_{a}^{0},g_{b}^{0},g_{c}^{0},g_{d}^{0}$
as above, the map 
\[
\X_{n}\to S_{n}^{4},\quad\phi\mapsto\left(\phi(a),\phi(b),\phi(c),\phi(d)\right)
\]
restricts to a bijection between $\X_{n}(Y,\J_{n})$ and the tuples
$(g_{a},g_{b},g_{c},g_{d})\in S_{n}^{4}$ such that both $(g_{a},g_{b},g_{c},g_{d})$
is in the coset $(g_{a}^{0},g_{b}^{0},g_{c}^{0},g_{d}^{0})G$, and$\left[g_{d}^{-1},g_{c}^{-1}\right]\left[g_{b}^{-1},g_{a}^{-1}\right]=1$. 

Now let 
\[
I=\int_{h_{f}\in G_{f}}{\bf 1}\left\{ \left[\left(g_{d}^{0}h_{d}\right)^{-1},\left(g_{c}^{0}h_{c}\right)^{-1}\right]\left[\left(g_{b}^{0}h_{b}\right)^{-1},\left(g_{a}^{0}h_{a}\right)^{-1}\right]=1\right\} .
\]
Then it is immediate that 
\[
|\X_{n}(Y,\J_{n})|=|G|\cdot I=\prod_{f\in a,b,c,d}(n-\e_{f})!\cdot I.
\]
Finally, use Schur orthogonality to write, as functions on $S_{n}$,
\[
\mathbf{1}\{g=1\}=\frac{1}{n!}\sum_{\lambda\vdash n}d_{\lambda}\chi_{\lambda}(g),
\]
insert this into the definition of $I$, and interchange summation
and integration to complete the proof.
\end{proof}
In the next sections we will focus our attention on the quantities
$\Theta_{\lambda}(Y,\J_{n})$. 

\subsection{Construction of auxiliary permutations\label{subsec:Construction-of-auxiliary-numberings}}

In order to obtain an expression for $\Theta_{\lambda}(Y,\J_{n})$
that leads to good analytic estimates, we introduce further maps
\[
\sigma_{f}^{+},\sigma_{f}^{-},\tau_{f}^{+},\tau_{f}^{-}\in S'_{\v}\subset S_{n}
\]
for each $f\in\{a,b,c,d\}$. These should be thought of as orderings
of the vertices with indices from $[n-\v+1,n]$, other than the one
fixed by $\J_{n}$. We will first describe the construction of these
maps, and then note their properties.

Recall from $\S\S$\ref{subsec:tiled-surfaces-and-core-surfaces}
that $Y^{(1)}$ carries the structure of a ribbon graph and this gives
us a way to thicken it up to an oriented surface with boundary with
an embedded copy of the graph $Y^{(1)}$. Also recall, from $\S\S$\ref{subsec:tiled-surfaces-and-core-surfaces},
that we constructed a larger object $Y_{+}$ by adding extra hanging
half-edges to the vertices. The one-skeleton $Y_{+}^{(1)}$ also has
a cyclic ordering of the half-edges (hanging or otherwise) at each
vertex and so $Y_{+}^{(1)}$ can be thickened up to a `cut' ribbon
graph with some half-ribbon edges. In this picture, every edge is
thickened to a thin rectangle, and every hanging half-edge is thickened
up to a thin half-rectangle.

So every vertex of $Y$ has eight incident half-edges (hanging or
otherwise), and each of these half-edges has two sides. The $16$
maps $\left\{ \sigma_{f}^{\pm},\tau_{f}^{\pm}\right\} _{f\in\left\{ a,b,c,d\right\} }$
correspond to these $16$ sides of half-edges at each vertex: $\sigma_{f}^{-}$
and $\tau_{f}^{-}$ correspond to the sides of the outgoing $f$-half-edge,
while $\sigma_{f}^{+}$ and $\tau_{f}^{+}$ correspond to the sides
of the incoming $f$-half-edge. Finally, $\sigma_{f}^{\pm}$ correspond
to the left side of the outgoing and incoming $f$-half-edges, while
$\tau_{f}^{\pm}$ correspond to the right side of these $f$-half-edges,
where `left' and `right' here are with respect to the direction of
the half-edge. (We keep our convention from $\S\S$\ref{subsec:tiled-surfaces-and-core-surfaces}
that boundary cycles are oriented so that the object lies to the right.
In particular, the boundary of an octagon is $\left[a,b\right]\left[c,d\right]$
when followed in counter-clockwise direction.) See the left hand side
of Figure \ref{fig:construction of sigma,tau pm}.

The definition of $\sigma_{f}^{\pm},\tau_{f}^{\pm}$ is based on the
following choices:
\begin{itemize}
\item \textbf{Numbering octagons: }Number the $\f$ octagons of $Y$ by
distinct elements in $[\v-\f+1,\v]$.
\item \textbf{Numbering full edges at $\partial Y$:} For every $f\in\left\{ a,b,c,d\right\} $,
there are $\e_{f}-\f$ left-sides of full $f$-edges that belong to
the boundary $\partial Y$ (as compared with $\f$ left-sides of full
$f$-edges that meet octagons of $Y$). Number them by distinct values
in $\left[\v-\e_{f}+1,\v-\f\right]$. Similarly, number the $\e_{f}-\f$
right-sides of full $f$-edges belonging to $\partial Y$ by distinct
values in the same range $\left[\v-\e_{f}+1,\v-\f\right]$.
\item \textbf{Numbering hanging half-edges: }For each $f\in\{a,b,c,d\}$,
there are precisely $\v-\e_{f}$ outgoing $f$-labeled hanging half-edges,
and we number them by distinct numbers in $[\v-\e_{f}]$. Using the
matching determined by $g_{f}^{0}$ between outgoing and incoming
$f$-labeled hanging half-edges, the numbering we have just chosen
induces a numbering also of the incoming $f$-labeled hanging half-rectangles
by numbers in $[\v-\e_{f}]$.
\end{itemize}
We now define $\sigma_{f}^{\pm}$ and $\tau_{f}^{\pm}$ as follows.
For every vertex $v$ of $Y$ and a side of an incident half-edge
(hanging or otherwise), we need to determine the image of $v$ under
the permutation among $\sigma_{f}^{\pm},\tau_{f}^{\pm}$ corresponding
to this side-of-half-edge. 
\begin{itemize}
\item If the half-edge is part of a full-edge of $Y$, then,
\begin{itemize}
\item if the side in question meets an octagon numbered $i$, we map $v\mapsto n-\v+i$,
and
\item if the side in question belongs to $\partial Y$ and the full-edge
is numbered $j$, map $v\mapsto n-\v+j$.
\end{itemize}
\item If this is a \emph{hanging} half-edge numbered $k$, we map $v\mapsto n-\v+k$.
\end{itemize}
\begin{figure}
\centering{}\includegraphics[scale=0.9]{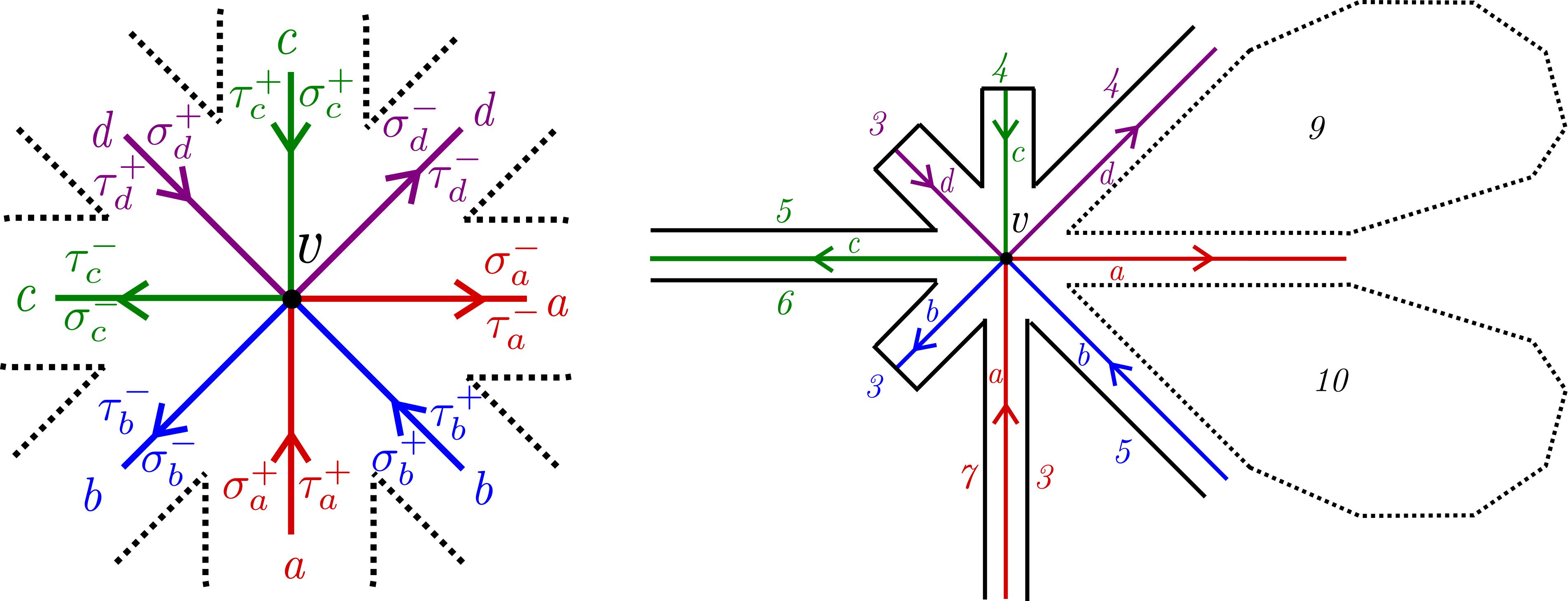}\caption{\label{fig:construction of sigma,tau pm}The figure on the left shows
a local picture of a vertex $v$ in the thick version of some tiled
surface with hanging half-edges $Y_{+}$, and the correspondence between
the $16$ maps $\sigma_{f}^{\pm},\tau_{f}^{\pm}$ and the $16$ sides
of half-edges incident to $v$. The figure on the right illustrates
how numbering of octagons, of exposed sides of full-edges, and of
hanging half-edges, determines the values of $\sigma_{f}^{\pm},\tau_{f}^{\pm}$
at $v$. In this figure, continuous black lines mark pieces of the
boundary of the thick version of $Y_{+}$, whereas dotted black lines
mark boundary pieces of $Y_{+}^{\left(1\right)}$ to which octagons
are glued in $Y_{+}$. The vertex $v$ in the center of the figure
is incident with two octagons, numbered $9$ and $10$; with three
hanging half-edges numbered $3$ (outgoing $b$ and incoming $d$)
and $4$ (incoming $c$); and with five half-edges belonging to full-edges,
with a total of six exposed sides, the numbering of which is described
in the figure. The images of this vertex under $\sigma_{f}^{\pm}$
and $\tau_{f}^{\pm}$ are listed in page \pageref{list of images in example}.}
\end{figure}

This is illustrated in the right hand side of Figure \ref{fig:construction of sigma,tau pm},
which shows some vertex $v$ of some tiled surface $Y$, and the numbering
of incident octagons, of exposed sides of full-edges and of hanging
half-edges. In that case, the images of $v$ under $\sigma_{f}^{\pm}$
and $\tau_{f}^{\pm}$ are the following:\label{list of images in example}
\[
\begin{array}{cccc}
\sigma_{a}^{-}\left(v\right)=n-\v+9 & \sigma_{b}^{-}\left(v\right)=n-\v+3 & \sigma_{c}^{-}\left(v\right)=n-\v+6 & \sigma_{d}^{-}\left(v\right)=n-\v+4\\
\sigma_{a}^{+}\left(v\right)=n-\v+7 & \sigma_{b}^{+}\left(v\right)=n-\v+5 & \sigma_{c}^{+}\left(v\right)=n-\v+4 & \sigma_{d}^{+}\left(v\right)=n-\v+3\\
\tau_{a}^{-}\left(v\right)=n-\v+10 & \tau_{b}^{-}\left(v\right)=n-\v+3 & \tau_{c}^{-}\left(v\right)=n-\v+5 & \tau_{d}^{-}\left(v\right)=n-\v+9\\
\tau_{a}^{+}\left(v\right)=n-\v+3 & \tau_{b}^{+}\left(v\right)=n-\v+10 & \tau_{c}^{+}\left(v\right)=n-\v+4 & \tau_{d}^{+}\left(v\right)=n-\v+3.
\end{array}
\]
The following properties of the maps we defined are all evident from
the construction.
\begin{lem}
When the vertices of $Y$ are identified with $\left[n-\v+1,n\right]$
according to $\J_{n}$, the $16$ maps $\sigma_{f}^{+},\sigma_{f}^{-},\tau_{f}^{+},\tau_{f}^{-}$
we defined indeed belong to $S'_{\v}\subset S_{n}$. Moreover, they
satisfy the following properties:
\begin{description}
\item [{P1}] For all $f\in\{a,b,c,d\}$, $\sigma_{f}^{\pm}(\V_{f}^{\pm})=\tau_{f}^{\pm}(\V_{f}^{\pm})=[n-\e_{f}+1,n]$.
\item [{P2}] For all $f\in\{a,b,c,d\}$, $(\sigma_{f}^{+})^{-1}\sigma_{f}^{-}=(\tau_{f}^{+})^{-1}\tau_{f}^{-}=g_{f}^{0}$. 
\item [{P3}] For all $f\in\{a,b,c,d\}$, $\sigma_{f}^{\pm}\lvert_{[n]\backslash\V_{f}^{\pm}}=\tau_{f}^{\pm}\lvert_{[n]\backslash\V_{f}^{\pm}}.$ 
\item [{P4}] Each of the following permutations fixes every element of
$[n-\f+1,n]$: 
\[
\sigma_{b}^{-}\left(\sigma_{a}^{+}\right)^{-1},\tau_{a}^{+}\left(\sigma_{b}^{+}\right)^{-1},\tau_{b}^{+}\left(\tau_{a}^{-}\right)^{-1},\sigma_{c}^{-}\left(\tau_{b}^{-}\right)^{-1},\sigma_{d}^{-}\left(\sigma_{c}^{+}\right)^{-1},\tau_{c}^{+}\left(\sigma_{d}^{+}\right)^{-1},\tau_{d}^{+}\left(\tau_{c}^{-}\right)^{-1},\sigma_{a}^{-}\left(\tau_{d}^{-}\right)^{-1}.
\]
\item [{P5}] The permutations $\sigma_{f}^{\pm},\tau_{f}^{\pm}$ are the
same for each $n$ in the sense that they change with $n$ via the
fixed isomorphisms between $S_{\v}$ and $S'_{\v}\leq S_{n}$ in (\ref{eq:Sk-Skdash-isomorphism}). 
\end{description}
\end{lem}

From now on, assume that we have fixed $\sigma_{f}^{\pm},\tau_{f}^{\pm}$
with properties \textbf{P1-P5. }We do this once and for all for every
tiled surface $Y$ (including the choice of $\J$).

\subsection{Integrating over cosets\label{subsec:Integrating-over-double-cosets}}

We briefly review some linear algebra. Recall that $\check{V}^{\lambda}$
is the vector space of complex linear functionals on $V^{\lambda}$.
If $V^{\lambda}$ has orthonormal basis $\{v_{i}\}$, then $\check{V}^{\lambda}$
has a dual basis $\{\check{v}_{i}\}$ defined by $\check{v}_{i}(v)\eqdf\langle v,v_{i}\rangle.$
Requiring the $\check{v_{i}}$ to be orthonormal defines a Hermitian
inner product on $\check{V}^{\lambda}$. The action of $S_{n}$ on
$\check{V}^{\lambda}$ is by $g[\phi](v)\eqdf\phi(g^{-1}v)$. If $A_{ji}\eqdf\langle gv_{i},v_{j}\rangle$
so that $g$ acts by the matrix $A=(A_{ij})$ on $V^{\lambda}$ in
this basis, then 
\[
g[\check{v}_{i}](v_{k})=\check{v}_{i}\left[\sum_{\ell}(A^{-1})_{\ell k}v_{\ell}\right]=\left(A^{-1}\right)_{ik},
\]
so 
\[
g[\check{v}_{i}]=\sum_{k}\left(A^{-1}\right)_{ik}\check{v}_{k}.
\]

We now give some motivation for what follows. We wish to integrate
the function 
\[
S_{n}^{4}\to\R,\quad\left(g_{a},g_{b},g_{c},g_{d}\right)\mapsto\chi_{\lambda}\left(\left[g_{d}^{-1},g_{c}^{-1}\right]\left[g_{b}^{-1},g_{a}^{-1}\right]\right),
\]
over a coset in $S_{n}^{4}$. This function can clearly be written
as a finite sum of finite products of matrix coefficients of the $g_{f}$
and $g_{f}^{-1}$ in $V^{\lambda}$. However, this is not the route
we wish to take. Instead, following a philosophy similar to that used
in the development of the Weingarten calculus (see for example, \cite{CS}),
we aim to write this function more holistically as (what is essentially)
\uline{one single matrix coefficient} in \uline{one single representation}.
To this end, consider the vector space\-\-
\begin{equation}
W^{\lambda}\eqdf V_{a}^{\lambda}\otimes\check{V_{a}^{\lambda}}\otimes V_{b}^{\lambda}\otimes\check{V_{b}^{\lambda}}\otimes V_{c}^{\lambda}\otimes\check{V_{c}^{\lambda}}\otimes V_{d}^{\lambda}\otimes\check{V_{d}^{\lambda}}\label{eq:def-W-lambda}
\end{equation}
as a unitary representation of $S_{n}^{4}.$ We write an element of
$S_{n}^{4}$ as $(g_{a},g_{b},g_{c},g_{d})$ and the subscripts above
indicate which coordinate acts on which factor. 

Let $B_{\lambda}\in\End(W^{\lambda})$ be defined via matrix coefficients
by the formula
\begin{align}
 & \left\langle B_{\lambda}\left(v_{1}\otimes\check{v}_{2}\otimes v_{3}\otimes\check{v_{4}}\otimes v_{5}\otimes\check{v}_{6}\otimes v_{7}\otimes\check{v_{8}}\right),w_{1}\otimes\check{w}_{2}\otimes w_{3}\otimes\check{w_{4}}\otimes w_{5}\otimes\check{w}_{6}\otimes w_{7}\otimes\check{w_{8}}\right\rangle \eqdf\nonumber \\
 & ~~~~~~~~~\left\langle v_{1},w_{3}\right\rangle \left\langle v_{3},v_{2}\right\rangle \left\langle w_{2},v_{4}\right\rangle \left\langle w_{4},w_{5}\right\rangle \left\langle v_{5},w_{7}\right\rangle \left\langle v_{7},v_{6}\right\rangle \left\langle w_{6},v_{8}\right\rangle \left\langle w_{8},w_{1}\right\rangle .\label{eq:B-lambda-def}
\end{align}

\begin{rem}
One could also order the tensor factors in (\ref{eq:def-W-lambda})
according to the order specified by the word $\left[a,b\right]\left[c,d\right]$,
namely $V_{a}^{\lambda}\otimes V_{b}^{\lambda}\check{\otimes V_{a}^{\lambda}}\otimes\check{V_{b}^{\lambda}}\otimes V_{c}^{\lambda}\otimes V_{d}^{\lambda}\otimes\check{V_{c}^{\lambda}}\otimes\check{V_{d}^{\lambda}}$.
In this case, the definition of $B_{\lambda}$ would be more natural:
$\left\langle v_{1},w_{2}\right\rangle \left\langle v_{2},v_{3}\right\rangle \left\langle w_{3},v_{4}\right\rangle \left\langle w_{4},w_{5}\right\rangle \left\langle v_{5},w_{6}\right\rangle \left\langle v_{6},v_{7}\right\rangle \left\langle w_{7},v_{8}\right\rangle \left\langle w_{8},w_{1}\right\rangle $
and easily generalizable to arbitrary words. We chose to stick with
the order in (\ref{eq:def-W-lambda}) for ease of notation in the
sequel, e.g.~in Lemma \ref{lem:twisted-orthonormal-basis}.
\end{rem}

\begin{lem}
\label{lem:B-lambda-property}For any $(g_{a},g_{b},g_{c},g_{d})\in S_{n}^{4}$,
we have 
\[
\mathrm{tr}_{W^{\lambda}}\left(B_{\lambda}\circ(g_{a},g_{b},g_{c},g_{d})\right)=\chi_{\lambda}\left(\left[g_{d}^{-1},g_{c}^{-1}\right]\left[g_{b}^{-1},g_{a}^{-1}\right]\right).
\]
\end{lem}

\begin{proof}
Let $v_{i}$ be any orthonormal basis of $V^{\lambda}$. Let $a_{ji}\eqdf\langle g_{a}v_{i},v_{j}\rangle$
be the matrix coefficients of the matrix $a=(a_{ij})$ by which $g$
acts on $V^{\lambda}$ with respect to $\{v_{i}\}$. Similarly define
matrices $b,c,d$ for $g_{b},g_{c},g_{d}$ in $V^{\lambda}$. We have
\begin{eqnarray*}
\mathrm{tr}_{W^{\lambda}}(B_{\lambda}\circ(g_{a},g_{b},g_{c},g_{d})) & = & \sum_{i_{1},\ldots,i_{8}}\langle B_{\lambda}\circ(g_{a},g_{b},g_{c},g_{d})v_{i_{1}}\otimes\check{v}_{i_{2}}\otimes v_{i_{3}}\otimes\check{v_{i_{4}}}\otimes v_{i_{5}}\otimes\check{v}_{i_{6}}\otimes v_{i_{7}}\otimes\check{v}_{i_{8}},\\
 &  & ~~~~~~~v_{i_{1}}\otimes\check{v}_{i_{2}}\otimes v_{i_{3}}\otimes\check{v}_{i_{4}}\otimes v_{i_{5}}\otimes\check{v}_{i_{6}}\otimes v_{i_{7}}\otimes\check{v}_{i_{8}}\rangle\\
 & = & \sum_{\substack{i_{1},\ldots,i_{8}\\
j_{1},\ldots,j_{8}
}
}a_{j_{1}i_{1}}(a^{-1})_{i_{2}j_{2}}b_{j_{3}i_{3}}(b^{-1})_{i_{4}j_{4}}c_{j_{5}i_{5}}(c^{-1})_{i_{6}j_{6}}d_{j_{7}i_{7}}(d^{-1})_{i_{8}j_{8}}\cdot\\
 &  & ~~~~~~~~~\left\langle B_{\lambda}v_{j_{1}}\otimes\check{v}_{j_{2}}\otimes v_{j_{3}}\otimes\check{v}_{j_{4}}\otimes v_{j_{5}}\otimes\check{v}_{j_{6}}\otimes v_{j_{7}}\otimes\check{v}_{j_{8}}\right.,\\
 &  & ~~~~~~~~~~~~~~~~v_{i_{1}}\otimes\check{v}_{i_{2}}\otimes v_{i_{3}}\otimes\check{v}_{i_{4}}\otimes v_{i_{5}}\otimes\check{v}_{i_{6}}\otimes v_{i_{7}}\otimes\check{v}_{i_{8}}\rangle
\end{eqnarray*}
which equals 
\begin{align*}
= & \sum_{j_{1},j_{2},j_{4},j_{5},j_{6},j_{8},i_{1},i_{4}}a_{j_{1}i_{1}}(a^{-1})_{j_{4}j_{2}}b_{j_{2}j_{1}}(b^{-1})_{i_{4}j_{4}}c_{j_{5}i_{4}}(c^{-1})_{j_{8}j_{6}}d_{j_{6}j_{5}}(d^{-1})_{i_{1}j_{8}}\\
= & \sum_{j_{1},j_{2},j_{4},j_{5},j_{6},j_{8},i_{1},i_{4}}(d^{-1})_{i_{1}j_{8}}(c^{-1})_{j_{8}j_{6}}d_{j_{6}j_{5}}c_{j_{5}i_{4}}(b^{-1})_{i_{4}j_{4}}(a^{-1})_{j_{4}j_{2}}b_{j_{2}j_{1}}a_{j_{1}i_{1}}\\
= & \chi_{\lambda}\left(\left[g_{d}^{-1},g_{c}^{-1}\right]\left[g_{b}^{-1},g_{a}^{-1}\right]\right).
\end{align*}
The third equality used (\ref{eq:B-lambda-def}).
\end{proof}
Using Lemma \ref{lem:B-lambda-property} allows us to relate $\Theta_{\lambda}(Y,\J_{n})$
to orthogonal projections in the space $W^{\lambda}$. For each $f\in\{a,b,c,d\}$
let $P_{f}$ be the orthogonal projection in $W^{\lambda}$ onto the
vectors that are invariant by $G_{f}$. We let $Q\eqdf P_{a}P_{b}P_{c}P_{d}.$ 
\begin{lem}
\label{lem:theta-as-trace-of-B}We have $\Theta_{\lambda}(Y,\J_{n})=\mathrm{tr}_{W^{\lambda}}\left(B_{\lambda}g^{0}Q\right)$.
\end{lem}

\begin{proof}
Using Lemma \ref{lem:B-lambda-property}, we can write
\begin{align*}
\Theta_{\lambda}(Y,\J_{n}) & =\int_{h_{_{f}}\in G_{f}}\chi_{\lambda}\left(\left[\left(g_{d}^{0}h_{d}\right)^{-1},\left(g_{c}^{0}h_{c}\right)^{-1}\right]\left[\left(g_{b}^{0}h_{b}\right)^{-1},\left(g_{a}^{0}h_{a}\right)^{-1}\right]\right)\\
 & =\mathrm{tr}_{W^{\lambda}}\left(B_{\lambda}g^{0}P_{a}P_{b}P_{c}P_{d}\right)=\mathrm{tr}_{W^{\lambda}}\left(B_{\lambda}g^{0}Q\right).
\end{align*}
\end{proof}
Hence, we now wish to calculate $\mathrm{tr}_{W^{\lambda}}\left(B_{\lambda}g^{0}Q\right)$.
For each $f\in\{a,b,c,d\}$ and $T\in\Tab(\lambda)$ let
\begin{align}
v_{T}^{\sigma_{f}^{\pm}} & \eqdf\left(\sigma_{f}^{\pm}\right)^{-1}\left(v_{T}\right),\quad v_{T}^{\tau_{f}^{\pm}}\eqdf\left(\tau_{f}^{\pm}\right)^{-1}\left(v_{T}\right).\label{eq:base-change}
\end{align}
Similarly, if $\nu\subset_{\v}\lambda$, recalling that if $T\in\Tab(\lambda/\nu)$,
$w_{T}$ denotes the corresponding Gelfand-Tsetlin basis element of
$V^{\lambda/\nu}$, we define
\[
w_{T}^{\sigma_{f}^{\pm}}\eqdf\left(\sigma_{f}^{\pm}\right)^{-1}\left(w_{T}\right),\quad w_{T}^{\tau_{f}^{\pm}}\eqdf\left(\tau_{f}^{\pm}\right)^{-1}\left(w_{T}\right);
\]
this makes sense as $\sigma_{f}^{\pm}$ and $\tau_{f}^{\pm}$ are
in $S_{\v}^{'}$. Recalling the notation $\EE_{\mu,R_{1},R_{2}}^{\lambda}$
from Lemma \ref{lem:commutant} (where $\mu\subset\lambda$), we define
\begin{align}
\EE_{\mu,R_{1},R_{2}}^{\lambda,f,\pm} & \eqdf\left(\sigma_{f}^{\pm}\right)^{-1}\otimes\left(\tau_{f}^{\pm}\right)^{-1}\left(\EE_{\mu,R_{1},R_{2}}^{\lambda}\right).\label{eq:twisted-E-def}
\end{align}

\begin{lem}
\label{lem:twisted-orthonormal-basis}For $\lambda\vdash n$, the
elements
\begin{equation}
\left\{ \,\EE_{\mu_{a},S_{a},T_{a}}^{\lambda,a,-}\otimes\EE_{\mu_{b},S_{b},T_{b}}^{\lambda,b,-}\otimes\EE_{\mu_{c},S_{c},T_{c}}^{\lambda,c,-}\otimes\EE_{\mu_{d},S_{d},T_{d}}^{\lambda,d,-}\,:\,\mu_{f}\subset_{\e_{f}}\lambda,\,S_{f},T_{f}\in\Tab(\lambda/\mu_{f})\,\right\} \label{eq:on-basis}
\end{equation}
are an orthonormal basis for the $G$-invariant vectors in $W^{\lambda}$.
\end{lem}

\begin{proof}
Consider $W^{\lambda}$ as a module for $\mathbf{G}\eqdf S_{a}^{(1)}\times S_{a}^{(2)}\times S_{b}^{(1)}\times S_{b}^{(2)}\times S_{c}^{(1)}\times S_{c}^{(2)}\times S_{d}^{(1)}\times S_{d}^{(2)}$
where all the $S_{f}^{(i)}$are isomorphic copies of $S_{n}$, $S_{f}^{(1)}$
acts on the $V_{f}^{\lambda}$ factor and $S_{f}^{(2)}$ acts on the
$\check{V}_{f}^{\lambda}$ factor of $W^{\lambda}$. Given subgroups
$H_{f}$ of $S_{n}$ for each $f\in\{a,b,c,d\}$, write $\Delta(H_{a},H_{b},H_{c},H_{d})$
for the subgroup consisting of tuples of the form $(g_{a},g_{a},g_{b},g_{b},g_{c},g_{c},g_{d},g_{d})$
with each $g_{f}\in H_{f}$. The statement of Lemma \ref{lem:twisted-orthonormal-basis}
is equivalent to the statement that the set given in (\ref{eq:on-basis})
is an orthonormal basis for the $\Delta(G_{a},G_{b},G_{c},G_{d})$-invariant
elements. 

We have 
\begin{align*}
 & \EE_{\mu_{a},S_{a},T_{a}}^{\lambda,a,-}\otimes\EE_{\mu_{b},S_{b},T_{b}}^{\lambda,b,-}\otimes\EE_{\mu_{c},S_{c},T_{c}}^{\lambda,c,-}\otimes\EE_{\mu_{d},S_{d},T_{d}}^{\lambda,d,-}=\\
 & \left(\sigma_{a}^{-},\tau_{a}^{-},\sigma_{b}^{-},\tau_{b}^{-},\sigma_{c}^{-},\tau_{c}^{-},\sigma_{d}^{-},\tau_{d}^{-}\right)^{-1}\EE_{\mu_{a},S_{a},T_{a}}^{\lambda}\otimes\EE_{\mu_{b},S_{b},T_{b}}^{\lambda}\otimes\EE_{\mu_{c},S_{c},T_{c}}^{\lambda}\otimes\EE_{\mu_{d},S_{d},T_{d}}^{\lambda}.
\end{align*}
We note that $(\sigma_{a}^{-},\tau_{a}^{-},\sigma_{b}^{-},\tau_{b}^{-},\sigma_{c}^{-},\tau_{c}^{-},\sigma_{d}^{-},\tau_{d}^{-})$
acts unitarily on $W^{\lambda}$, and by Lemma \ref{lem:commutant},
the vectors $\EE_{\mu_{a},S_{a},T_{a}}^{\lambda}\otimes\EE_{\mu_{b},S_{b},T_{b}}^{\lambda}\otimes\EE_{\mu_{c},S_{c},T_{c}}^{\lambda}\otimes\EE_{\mu_{d},S_{d},T_{d}}^{\lambda}$
are an orthonormal basis for the\\
 $\Delta(S_{n-\e_{a}},S_{n-\e_{b}},S_{n-\e_{c}},S_{n-\e_{d}})$-invariant
vectors in $W^{\lambda}$. Therefore the set given in (\ref{eq:on-basis})
is an orthonormal basis of invariant vectors for the group
\[
\left(\sigma_{a}^{-},\tau_{a}^{-},\sigma_{b}^{-},\tau_{b}^{-},\sigma_{c}^{-},\tau_{c}^{-},\sigma_{d}^{-},\tau_{d}^{-}\right)^{-1}\Delta\left(S_{n-\e_{a}},S_{n-\e_{b}},S_{n-\e_{c}},S_{n-\e_{d}}\right)\left(\sigma_{a}^{-},\tau_{a}^{-},\sigma_{b}^{-},\tau_{b}^{-},\sigma_{c}^{-},\tau_{c}^{-},\sigma_{d}^{-},\tau_{d}^{-}\right).
\]
It remains to prove that this group is $\Delta(G_{a},G_{b},G_{c},G_{d})$.
By property \textbf{P1}, this group is contained in $G_{a}\times G_{a}\times G_{b}\times G_{b}\times G_{c}\times G_{c}\times G_{d}\times G_{d}$.
Combining this with property \textbf{P3}, the group displayed above
is equal to $\Delta(G_{a},G_{b},G_{c},G_{d})$, as required.
\end{proof}
\begin{lem}
\label{lem:minus-to-plus-mapping}We have
\begin{align*}
g^{0}\left(\EE_{\mu_{a},S_{a},T_{a}}^{\lambda,a,-}\otimes\EE_{\mu_{b},S_{b},T_{b}}^{\lambda,b,-}\otimes\EE_{\mu_{c},S_{c},T_{c}}^{\lambda,c,-}\otimes\EE_{\mu_{d},S_{d},T_{d}}^{\lambda,d,-}\right)=\EE_{\mu_{a},S_{a},T_{a}}^{\lambda,a,+}\otimes\EE_{\mu_{b},S_{b},T_{b}}^{\lambda,b,+}\otimes\EE_{\mu_{c},S_{c},T_{c}}^{\lambda,c,+}\otimes\EE_{\mu_{d},S_{d},T_{d}}^{\lambda,d,+}.
\end{align*}
\end{lem}

\begin{proof}
This follows from property \textbf{P2} together with the definitions
of $\EE_{\mu_{f},S_{f},T_{f}}^{\lambda,f,\pm}$ in (\ref{eq:twisted-E-def}).
\end{proof}
\begin{prop}
\label{prop:theta-expression}Recalling the definition of $\Theta_{\lambda}(Y,\J_{n})$
from (\ref{eq:Phi_lambda}), we have
\begin{align}
\Theta_{\lambda}(Y,\J_{n})= & \sum_{\nu\subset_{\v-\f}\lambda'\subset_{\f}\lambda}d_{\lambda/\lambda'}d_{\nu}\sum_{\nu\subset\mu_{f}\subset_{\e_{f}-\f}\lambda'}\frac{1}{d_{\mu_{a}}d_{\mu_{b}}d_{\mu_{c}}d_{\mu_{d}}}\Upsilon_{n}\left(\left\{ \sigma_{f}^{\pm},\tau_{f}^{\pm}\right\} ,\nu,\left\{ \mu_{f}\right\} ,\lambda'\right),\label{eq:theta-expression-1}
\end{align}
where 
\begin{align}
\Upsilon_{n}\left(\left\{ \sigma_{f}^{\pm},\tau_{f}^{\pm}\right\} ,\nu,\left\{ \mu_{f}\right\} ,\lambda'\right)\eqdf\sum_{\begin{gathered}r_{f}^{+},r_{f}^{-}\in\Tab\left(\mu_{f}/\nu\right)\\
s_{f},t_{f}\in\Tab\left(\lambda'/\mu_{f}\right)
\end{gathered}
}\M\left(\left\{ \sigma_{f}^{\pm},\tau_{f}^{\pm},r_{f}^{\pm},s_{f},t_{f}\right\} \right)\label{eq:Upsilon-def}
\end{align}
and $\M\left(\left\{ \sigma_{f}^{\pm},\tau_{f}^{\pm},r_{f}^{\pm},s_{f},t_{f}\right\} \right)$
is the following product of matrix coefficients
\begin{eqnarray}
\M\left(\left\{ \sigma_{f}^{\pm},\tau_{f}^{\pm},r_{f}^{\pm},s_{f},t_{f}\right\} \right) & \eqdf & \left\langle \sigma_{b}^{-}\left(\sigma_{a}^{+}\right)^{-1}w_{r_{a}^{+}\sqcup s_{a}},w_{r_{b}^{-}\sqcup s{}_{b}}\right\rangle \left\langle \tau_{a}^{+}\left(\sigma_{b}^{+}\right)^{-1}w_{r_{b}^{+}\sqcup s_{b}},w_{r_{a}^{+}\sqcup t_{a}}\right\rangle \cdot\nonumber \\
 &  & \left\langle \tau_{b}^{+}\left(\tau_{a}^{-}\right)^{-1}w_{r_{a}^{-}\sqcup t{}_{a}},w_{r_{b}^{+}\sqcup t_{b}}\right\rangle \left\langle \sigma_{c}^{-}\left(\tau_{b}^{-}\right)^{-1}w_{r_{b}^{-}\sqcup t{}_{b}},w_{r_{c}^{-}\sqcup s{}_{c}}\right\rangle \cdot\nonumber \\
 &  & \left\langle \sigma_{d}^{-}\left(\sigma_{c}^{+}\right)^{-1}w_{r_{c}^{+}\sqcup s_{c}},w_{r_{d}^{-}\sqcup s{}_{d}}\right\rangle \left\langle \tau_{c}^{+}\left(\sigma_{d}^{+}\right)^{-1}w_{r_{d}^{+}\sqcup s_{d}},w_{r_{c}^{+}\sqcup t_{c}}\right\rangle \cdot\nonumber \\
 &  & \left\langle \tau_{d}^{+}\left(\tau_{c}^{-}\right)^{-1}w_{r_{c}^{-}\sqcup t{}_{c}},w_{r_{d}^{+}\sqcup t_{d}}\right\rangle \left\langle \sigma_{a}^{-}\left(\tau_{d}^{-}\right)^{-1}w_{r_{d}^{-}\sqcup t{}_{d}},w_{r_{a}^{-}\sqcup s{}_{a}}\right\rangle .\label{eq:M-def}
\end{eqnarray}
\end{prop}

Before proving Proposition \ref{prop:theta-expression}, we say a
word about the interpretation of the formula. Recall that the permutations
$\sigma_{f}^{\pm}$ and $\tau_{f}^{\pm}$ all belong to $S'_{\v}\le S_{n}$.
But by property \textbf{P4},\textbf{ }the eight permutations appearing
in (\ref{eq:M-def}) all restrict to the identity on $\left[n-\f+1,n\right]$,
and so can be seen as permutations on $\left[n-\v+1,n-\f\right]$.
For every $\pi\in\left\{ \sigma_{f}^{\pm},\tau_{f}^{\pm}\right\} $,
$\pi^{-1}\left(\left[n-\v+1,n-\f\right]\right)$ correspond to vertices
where the corresponding side-of-half-$f$-edge belongs either to a
hanging half-edge, or to an exposed side of a full-edge. One should
think of $r_{f}^{-}$ as the skew Young tableau consisting of indices
of outgoing hanging half-edges labeled $f$, of $r_{f}^{+}$ as the
tableau of incoming hanging half-edges labeled $f$, of $s_{f}$ as
the tableau of exposed left-sides of full $f$-edges and of $t_{f}$
as the tableau of exposed right-sides of full $f$-edges. Then indeed,
for example, the indices corresponding to $\sigma_{f}^{-}$ are $r_{f}^{-}\sqcup s_{f}$,
those corresponding to $\sigma_{f}^{+}$ are $r_{f}^{+}\sqcup s_{f}$,
those corresponding to $\tau_{f}^{-}$ are $r_{f}^{-}\sqcup t_{f}$,
and those corresponding to $\tau_{f}^{+}$ are $r_{f}^{+}\sqcup t_{f}$.
\begin{proof}[Proof of Proposition \ref{prop:theta-expression}]
By Lemmas \ref{lem:theta-as-trace-of-B}, \ref{lem:twisted-orthonormal-basis},
and \ref{lem:minus-to-plus-mapping}, we have
\begin{eqnarray}
\Theta_{\lambda}\left(Y,\J_{n}\right) & = & \trace_{W^{\lambda}}\left(B_{\lambda}g^{0}Q\right)\nonumber \\
 & = & \sum_{\begin{gathered}\mu_{f}\subset_{\e_{f}}\lambda,\\
S_{f},T_{f}\in\Tab\left(\lambda/\mu_{f}\right)
\end{gathered}
}\begin{gathered}\big\langle B_{\lambda}g^{0}\left[\EE_{\mu_{a},S_{a},T_{a}}^{\lambda,a,-}\otimes\EE_{\mu_{b},S_{b},T_{b}}^{\lambda,b,-}\otimes\EE_{\mu_{c},S_{c},T_{c}}^{\lambda,c,-}\otimes\EE_{\mu_{d},S_{d},T_{d}}^{\lambda,d,-}\right],\\
\EE_{\mu_{a},S_{a},T_{a}}^{\lambda,a,-}\otimes\EE_{\mu_{b},S_{b},T_{b}}^{\lambda,b,-}\otimes\EE_{\mu_{c},S_{c},T_{c}}^{\lambda,c,-}\otimes\EE_{\mu_{d},S_{d},T_{d}}^{\lambda,d,-}\big\rangle
\end{gathered}
\nonumber \\
 & = & \sum_{\begin{gathered}\mu_{f}\subset_{\e_{f}}\lambda,\\
S_{f},T_{f}\in\Tab\left(\lambda/\mu_{f}\right)
\end{gathered}
}\begin{gathered}\big\langle B_{\lambda}\left[\EE_{\mu_{a},S_{a},T_{a}}^{\lambda,a,+}\otimes\EE_{\mu_{b},S_{b},T_{b}}^{\lambda,b,+}\otimes\EE_{\mu_{c},S_{c},T_{c}}^{\lambda,c,+}\otimes\EE_{\mu_{d},S_{d},T_{d}}^{\lambda,d,+}\right],\\
\EE_{\mu_{a},S_{a},T_{a}}^{\lambda,a,-}\otimes\EE_{\mu_{b},S_{b},T_{b}}^{\lambda,b,-}\otimes\EE_{\mu_{c},S_{c},T_{c}}^{\lambda,c,-}\otimes\EE_{\mu_{d},S_{d},T_{d}}^{\lambda,d,-}\big\rangle.
\end{gathered}
\label{eq:trace-calc-1}
\end{eqnarray}
Using (\ref{eq:E-def}), (\ref{eq:B-lambda-def}) and (\ref{eq:twisted-E-def}),
we obtain
\begin{align*}
 & \left\langle B_{\lambda}\left[\EE_{\mu_{a},S_{a},T_{a}}^{\lambda,a,+}\otimes\EE_{\mu_{b},S_{b},T_{b}}^{\lambda,b,+}\otimes\EE_{\mu_{c},S_{c},T_{c}}^{\lambda,c,+}\otimes\EE_{\mu_{d},S_{d},T_{d}}^{\lambda,d,+}\right],\EE_{\mu{}_{a},S_{a},T_{a}}^{\lambda,a,-}\otimes\EE_{\mu{}_{b},S_{b},T_{b}}^{\lambda,b,-}\otimes\EE_{\mu{}_{c},S_{c},T_{c}}^{\lambda,c,-}\otimes\EE_{\mu{}_{d},S_{d},T_{d}}^{\lambda,d,-}\right\rangle \\
= & \frac{1}{d_{\mu_{a}}d_{\mu_{b}}d_{\mu_{c}}d_{\mu_{d}}}\sum_{R_{f}^{+},R_{f}^{-}\in\Tab\left(\mu_{f}\right)}\\
 & \left\langle B_{\lambda}\left[v_{R_{a}^{+}\sqcup S_{a}}^{\sigma_{a}^{+}}\otimes\check{v}_{R_{a}^{+}\sqcup T_{a}}^{\tau_{a}^{+}}\otimes v_{R_{b}^{+}\sqcup S_{b}}^{\sigma_{b}^{+}}\otimes\check{v}_{R_{b}^{+}\sqcup T_{b}}^{\tau_{b}^{+}}\otimes v_{R_{c}^{+}\sqcup S_{c}}^{\sigma_{c}^{+}}\otimes\check{v}_{R_{c}^{+}\sqcup T_{c}}^{\tau_{c}^{+}}\otimes v_{R_{d}^{+}\sqcup S_{d}}^{\sigma_{d}^{+}}\otimes\check{v}_{R_{d}^{+}\sqcup T_{d}}^{\tau_{d}^{+}}\right]\right.,\\
 & \left.~~v_{R_{a}^{-}\sqcup S_{a}}^{\sigma_{a}^{-}}\otimes\check{v}_{R_{a}^{-}\sqcup T_{a}}^{\tau_{a}^{-}}\otimes v_{R_{b}^{-}\sqcup S_{b}}^{\sigma_{b}^{-}}\otimes\check{v}_{R_{b}^{-}\sqcup T_{b}}^{\tau_{b}^{-}}\otimes v_{R_{c}^{-}\sqcup S_{c}}^{\sigma_{c}^{-}}\otimes\check{v}_{R_{c}^{-}\sqcup T_{c}}^{\tau_{c}^{-}}\otimes v_{R_{d}^{-}\sqcup S_{d}}^{\sigma_{d}^{-}}\otimes\check{v}_{R_{d}^{-}\sqcup T_{d}}^{\tau_{d}^{-}}\right\rangle \\
= & \frac{1}{d_{\mu_{a}}d_{\mu_{b}}d_{\mu_{c}}d_{\mu_{d}}}\sum_{R_{f}^{\pm}\in\Tab\left(\mu_{f}\right)}\left\langle v_{R_{a}^{+}\sqcup S_{a}}^{\sigma_{a}^{+}},v_{R_{b}^{-}\sqcup S_{b}}^{\sigma_{b}^{-}}\right\rangle \left\langle v_{R_{b}^{+}\sqcup S_{b}}^{\sigma_{b}^{+}},v_{R_{a}^{+}\sqcup T_{a}}^{\tau_{a}^{+}}\right\rangle \left\langle v_{R_{a}^{-}\sqcup T_{a}}^{\tau_{a}^{-}},v_{R_{b}^{+}\sqcup T_{b}}^{\tau_{b}^{+}}\right\rangle \cdot\\
 & ~~~~~\left\langle v_{R_{b}^{-}\sqcup T_{b}}^{\tau_{b}^{-}},v_{R_{c}^{-}\sqcup S_{c}}^{\sigma_{c}^{-}}\right\rangle \left\langle v_{R_{c}^{+}\sqcup S_{c}}^{\sigma_{c}^{+}},v_{R_{d}^{-}\sqcup S_{d}}^{\sigma_{d}^{-}}\right\rangle \left\langle v_{R_{d}^{+}\sqcup S_{d}}^{\sigma_{d}^{+}},v_{R_{c}^{+}\sqcup T_{c}}^{\tau_{c}^{+}}\right\rangle \left\langle v_{R_{c}^{-}\sqcup T_{c}}^{\tau_{c}^{-}},v_{R_{d}^{+}\sqcup T_{d}}^{\tau_{d}^{+}}\right\rangle \left\langle v_{R_{d}^{-}\sqcup T_{d}}^{\tau_{d}^{-}},v_{R_{a}^{-}\sqcup S_{a}}^{\sigma_{a}^{-}}\right\rangle .
\end{align*}
Since $\sigma_{f}^{\pm},\tau_{f}^{\pm}\in S'_{\v}$ for all $f\in\{a,b,c,d\}$,
the only way the product of matrix coefficients above can be non-zero
is if there is $\nu\vdash n-\v$ such that $\nu\subset\mu_{f}$ for
all $f\in\{a,b,c,d\}$, and all $R_{f}^{+}\lvert_{\le n-\v}$, $R_{f}^{-}\lvert_{\le n-\v}$
are equal and of shape $\nu$. Also, recall from Section \ref{subsec:Representations-of-symmetric}
that the action of $\sigma\in S'_{\v}$ on a tableau of shape $\lambda\vdash n$
depends only on the boxes with numbers from $\left[n-\v+1,n\right]$.
This gives 
\begin{align}
 & \left\langle B_{\lambda}\left[\EE_{\mu_{a},S_{a},T_{a}}^{\lambda,a,+}\otimes\EE_{\mu_{b},S_{b},T_{b}}^{\lambda,b,+}\otimes\EE_{\mu_{c},S_{c},T_{c}}^{\lambda,c,+}\otimes\EE_{\mu_{d},S_{d},T_{d}}^{\lambda,d,+}\right],\EE_{\mu{}_{a},S_{a},T_{a}}^{\lambda,a,-}\otimes\EE_{\mu{}_{b},S_{b},T_{b}}^{\lambda,b,-}\otimes\EE_{\mu{}_{c},S_{c},T_{c}}^{\lambda,c,-}\otimes\EE_{\mu{}_{d},S_{d},T_{d}}^{\lambda,d,-}\right\rangle \nonumber \\
= & \sum_{\nu\subset_{\v}\lambda}\frac{d_{\nu}}{d_{\mu_{a}}d_{\mu_{b}}d_{\mu_{c}}d_{\mu_{d}}}\sum_{r_{f}^{+},r_{f}^{-}\in\Tab\left(\mu_{f}/\nu\right)}\nonumber \\
 & \left\langle w_{r_{a}^{+}\sqcup S_{a}}^{\sigma_{a}^{+}},w_{r_{b}^{-}\sqcup S_{b}}^{\sigma_{b}^{-}}\right\rangle \cdot\left\langle w_{r_{b}^{+}\sqcup S_{b}}^{\sigma_{b}^{+}},w_{r_{a}^{+}\sqcup T_{a}}^{\tau_{a}^{+}}\right\rangle \cdot\left\langle w_{r_{a}^{-}\sqcup T_{a}}^{\tau_{a}^{-}},w_{r_{b}^{+}\sqcup T_{b}}^{\tau_{b}^{+}}\right\rangle \cdot\left\langle w_{r_{b}^{-}\sqcup T_{b}}^{\tau_{b}^{-}},w_{r_{c}^{-}\sqcup S_{c}}^{\sigma_{c}^{-}}\right\rangle \cdot\nonumber \\
 & \left\langle w_{r_{c}^{+}\sqcup S_{c}}^{\sigma_{c}^{+}},w_{r_{d}^{-}\sqcup S_{d}}^{\sigma_{d}^{-}}\right\rangle \cdot\left\langle w_{r_{d}^{+}\sqcup S_{d}}^{\sigma_{d}^{+}},w_{r_{c}^{+}\sqcup T_{c}}^{\tau_{c}^{+}}\right\rangle \cdot\left\langle w_{r_{c}^{-}\sqcup T_{c}}^{\tau_{c}^{-}},w_{r_{d}^{+}\sqcup T_{d}}^{\tau_{d}^{+}}\right\rangle \cdot\left\langle w_{r_{d}^{-}\sqcup T_{d}}^{\tau_{d}^{-}},w_{r_{a}^{-}\sqcup S_{a}}^{\sigma_{a}^{-}}\right\rangle .\label{eq:trace-calc-2}
\end{align}
Putting (\ref{eq:trace-calc-1}) and (\ref{eq:trace-calc-2}) together
yields
\begin{eqnarray*}
\Theta_{\lambda}\left(Y,\J_{n}\right) & = & \sum_{\nu\subset_{\v}\lambda}d_{\nu}\sum_{\nu\subset\mu_{f}\subset_{\e_{f}}\lambda}\frac{1}{d_{\mu_{a}}d_{\mu_{b}}d_{\mu_{c}}d_{\mu_{d}}}\sum_{r_{f}^{+},r_{f}^{-}\in\Tab\left(\mu_{f}/\nu\right)}\sum_{S_{f},T_{f}\in\Tab\left(\lambda/\mu_{f}\right)}\\
 &  & \left\langle w_{r_{a}^{+}\sqcup S_{a}}^{\sigma_{a}^{+}},w_{r_{b}^{-}\sqcup S_{b}}^{\sigma_{b}^{-}}\right\rangle \left\langle w_{r_{b}^{+}\sqcup S_{b}}^{\sigma_{b}^{+}},w_{r_{a}^{+}\sqcup T_{a}}^{\tau_{a}^{+}}\right\rangle \left\langle w_{r_{a}^{-}\sqcup T_{a}}^{\tau_{a}^{-}},w_{r_{b}^{+}\sqcup T_{b}}^{\tau_{b}^{+}}\right\rangle \left\langle w_{r_{b}^{-}\sqcup T_{b}}^{\tau_{b}^{-}},w_{r_{c}^{-}\sqcup S_{c}}^{\sigma_{c}^{-}}\right\rangle \\
 &  & \left\langle w_{r_{c}^{+}\sqcup S_{c}}^{\sigma_{c}^{+}},w_{r_{d}^{-}\sqcup S_{d}}^{\sigma_{d}^{-}}\right\rangle \left\langle w_{r_{d}^{+}\sqcup S_{d}}^{\sigma_{d}^{+}},w_{r_{c}^{+}\sqcup T_{c}}^{\tau_{c}^{+}}\right\rangle \left\langle w_{r_{c}^{-}\sqcup T_{c}}^{\tau_{c}^{-}},w_{r_{d}^{+}\sqcup T_{d}}^{\tau_{d}^{+}}\right\rangle \left\langle w_{r_{d}^{-}\sqcup T_{d}}^{\tau_{d}^{-}},w_{r_{a}^{-}\sqcup S_{a}}^{\sigma_{a}^{-}}\right\rangle .
\end{eqnarray*}
Now, $\langle w_{r_{a}^{+}\sqcup S_{a}}^{\sigma_{a}^{+}},w_{r_{b}^{-}\sqcup S_{b}}^{\sigma_{b}^{-}}\rangle=\langle\sigma_{b}^{-}\left(\sigma_{a}^{+}\right)^{-1}w_{r_{a}^{+}\sqcup S_{a}},w_{r_{b}^{-}\sqcup S_{b}}\rangle$
and so on, and property \textbf{P4 }implies that each pair $R_{f_{1},i_{1}}$
and $R_{f_{2},i_{1}}$ occurring in the same matrix coefficient above
have the elements $[n-\f+1,n]$ in the same boxes, if the matrix coefficient
is non-zero. This implies that if the product of matrix coefficients
is non-zero then all the $R_{f,i}$ above have the elements $[n-\f+1,n]$
in the same boxes and there is $\lambda'\subset\lambda$ such that
$\mu_{f}\subset_{\e_{f}-\f}\lambda'$ for all $f$. Therefore the
above is equal to 
\begin{eqnarray*}
\Theta_{\lambda}\left(Y,\J_{n}\right) & = & \sum_{\nu\subset_{\v-\f}\lambda'\subset_{\f}\lambda}d_{\lambda/\lambda'}d_{\nu}\sum_{\nu\subset\mu_{f}\subset_{\e_{f}-\f}\lambda'}\frac{1}{d_{\mu_{a}}d_{\mu_{b}}d_{\mu_{c}}d_{\mu_{d}}}\sum_{r_{f}^{+},r_{f}^{-}\in\Tab\left(\mu_{f}/\nu\right)}\sum_{s_{f},t_{f}\in\Tab\left(\lambda'/\mu_{f}\right)}\\
 &  & \left\langle \sigma_{b}^{-}\left(\sigma_{a}^{+}\right)^{-1}w_{r_{a}^{+}\sqcup s_{a}},w_{r_{b}^{-}\sqcup s{}_{b}}\right\rangle \cdot\left\langle \tau_{a}^{+}\left(\sigma_{b}^{+}\right)^{-1}w_{r_{b}^{+}\sqcup s_{b}},w_{r_{a}^{+}\sqcup t_{a}}\right\rangle \cdot\\
 &  & \left\langle \tau_{b}^{+}\left(\tau_{a}^{-}\right)^{-1}w_{r_{a}^{-}\sqcup t{}_{a}},w_{r_{b}^{+}\sqcup t_{b}}\right\rangle \cdot\left\langle \sigma_{c}^{-}\left(\tau_{b}^{-}\right)^{-1}w_{r_{b}^{-}\sqcup t{}_{b}},w_{r_{c}^{-}\sqcup s{}_{c}}\right\rangle \cdot\\
 &  & \left\langle \sigma_{d}^{-}\left(\sigma_{c}^{+}\right)^{-1}w_{r_{c}^{+}\sqcup s_{c}},w_{r_{d}^{-}\sqcup s{}_{d}}\right\rangle \cdot\left\langle \tau_{c}^{+}\left(\sigma_{d}^{+}\right)^{-1}w_{r_{d}^{+}\sqcup s_{d}},w_{r_{c}^{+}\sqcup t_{c}}\right\rangle \cdot\\
 &  & \left\langle \tau_{d}^{+}\left(\tau_{c}^{-}\right)^{-1}w_{r_{c}^{-}\sqcup t{}_{c}},w_{r_{d}^{+}\sqcup t_{d}}\right\rangle \cdot\left\langle \sigma_{a}^{-}\left(\tau_{d}^{-}\right)^{-1}w_{r_{d}^{-}\sqcup t{}_{d}},w_{r_{a}^{-}\sqcup s{}_{a}}\right\rangle .
\end{eqnarray*}
This finishes the proof.
\end{proof}
It is also useful to know the following.
\begin{lem}
\label{lem:upsilon-sym}We have $\Upsilon_{n}\left(\left\{ \sigma_{f}^{\pm},\tau_{f}^{\pm}\right\} ,\nu,\{\mu_{f}\},\lambda'\right)=\Upsilon_{n}\left(\left\{ \sigma_{f}^{\pm},\tau_{f}^{\pm}\right\} ,\check{\nu},\{\check{\mu}_{f}\},\check{\lambda}'\right)$.
\end{lem}

\begin{proof}
This uses that as $S_{n-\f}$ modules, $V^{\lambda'}$ and $V^{\check{\lambda'}}\otimes\mathrm{sign}$
are isomorphic by the map $w_{T}\mapsto w_{\check{T}}$. This gives
\begin{align*}
\M\left(\left\{ \sigma_{f}^{\pm},\tau_{f}^{\pm},r_{f}^{\pm},s_{f},t_{f}\right\} \right)= & \mathrm{sign}\left(\sigma_{d}^{-}\left(\sigma_{c}^{+}\right)^{-1}\right)\mathrm{sign}\left(\tau_{c}^{+}\left(\sigma_{d}^{+}\right)^{-1}\right)\mathrm{sign}\left(\tau_{d}^{+}\left(\tau_{c}^{-}\right)^{-1}\right)\mathrm{sign}\left(\sigma_{a}^{-}\left(\tau_{d}^{-}\right)^{-1}\right)\\
 & \cdot\mathrm{sign}\left(\sigma_{b}^{-}\left(\sigma_{a}^{+}\right)^{-1}\right)\mathrm{sign}\left(\tau_{a}^{+}\left(\sigma_{b}^{+}\right)^{-1}\right)\mathrm{sign}\left(\tau_{b}^{+}\left(\tau_{a}^{-}\right)^{-1}\right)\mathrm{sign}\left(\sigma_{c}^{-}\left(\tau_{b}^{-}\right)^{-1}\right)\cdot\\
 & \cdot\M\left(\left\{ \sigma_{f}^{\pm},\tau_{f}^{\pm},\check{r}_{f}^{\pm},\check{s}_{f},\check{t}_{f}\right\} \right)\\
= & \M\left(\left\{ \sigma_{f}^{\pm},\tau_{f}^{\pm},\check{r}_{f}^{\pm},\check{s}_{f},\check{t}_{f}\right\} \right)
\end{align*}
where the last line used \textbf{P2 }to get $(\tau_{f}^{+})^{-1}\tau_{f}^{-}=(\sigma_{f}^{+})^{-1}\sigma_{f}^{-}=g_{0}^{f}$.
Using this identity gives the result.
\end{proof}
We are now ready to give an exact expression for $\E_{n}^{\emb}(Y)$,
which is the main result of this $\S\S$\ref{subsec:Integrating-over-double-cosets}.
\begin{thm}
\label{thm:E_n-emb-exact-expression}For $n\geq\v$ we have
\begin{equation}
\E_{n}^{\emb}(Y)=\frac{\left(n!\right)^{3}}{\left|\X_{n}\right|}\cdot\frac{\left(n\right)_{\v}\left(n\right)_{\f}}{\prod_{f}\left(n\right)_{\e_{f}}}\Xi_{n}(Y)\label{eq:prob-to-xi-relation}
\end{equation}
where
\begin{equation}
\Xi_{n}(Y)\eqdf\sum_{\substack{\nu\subset_{\v-\f}\lambda'\vdash\,n-\f}
}d_{\lambda'}d_{\nu}\sum_{\nu\subset\mu_{f}\subset_{\e_{f}-\f}\lambda'}\frac{1}{d_{\mu_{a}}d_{\mu_{b}}d_{\mu_{c}}d_{\mu_{d}}}\Upsilon_{n}\left(\left\{ \sigma_{f}^{\pm},\tau_{f}^{\pm}\right\} ,\nu,\{\mu_{f}\},\lambda'\right).\label{eq:Xi-def}
\end{equation}
\end{thm}

\begin{rem}
\label{rem:Xi depends only on Y}Although the expression (\ref{eq:Xi-def})
seems to depend on the choices of $\sigma_{f}^{\pm}$ etc.~that we
have made already, the relation (\ref{eq:prob-to-xi-relation}) shows
that it only depends on $Y$.
\end{rem}

\begin{proof}[Proof of Theorem \ref{thm:E_n-emb-exact-expression}]
Recall from (\ref{eq:Enemb-to-Xn(Y,J)}) that $\E_{n}^{\emb}(Y)=\frac{n!}{(n-\v)!}\frac{|\X_{n}(Y,\J_{n})|}{|\X_{n}|}$.
Combining this with Propositions \ref{prop:probabilty-as-sum-over-irreps}
and \ref{prop:theta-expression} gives
\begin{align*}
\E_{n}^{\emb}(Y) & =\frac{\prod_{f\in a,b,c,d}(n-\e_{f})!}{(n-\v)!\left|\X_{n}\right|}\sum_{\nu\subset_{\v-\f}\lambda'\subset_{\f}\lambda\vdash n}d_{\lambda}d_{\lambda/\lambda'}d_{\nu}\sum_{\nu\subset\mu_{f}\subset_{\e_{f}-\f}\lambda'}\frac{1}{d_{\mu_{a}}d_{\mu_{b}}d_{\mu_{c}}d_{\mu_{d}}}\text{\ensuremath{\Upsilon_{n}\left(\left\{ \sigma_{f}^{\pm},\tau_{f}^{\pm}\right\} ,\nu,\{\mu_{f}\},\lambda'\right).}}
\end{align*}
Applying Lemma \ref{lem:induced-rep-dimension} to the summation over
$\lambda,$ for fixed $\lambda'$, yields
\[
\E_{n}^{\emb}(Y)=\frac{n!\prod_{f\in a,b,c,d}(n-\e_{f})!}{(n-\v)!(n-\f)!\left|\X_{n}\right|}\sum_{\nu\subset_{\v-\f}\lambda'\vdash\,n-\f}d_{\lambda'}d_{\nu}\sum_{\nu\subset\mu_{f}\subset_{\e_{f}-\f}\lambda'}\frac{1}{d_{\mu_{a}}d_{\mu_{b}}d_{\mu_{c}}d_{\mu_{d}}}\text{\ensuremath{\Upsilon_{n}\left(\left\{ \sigma_{f}^{\pm},\tau_{f}^{\pm}\right\} ,\nu,\{\mu_{f}\},\lambda'\right).}}
\]
\end{proof}
In view of Theorem \ref{thm:E_n-emb-exact-expression}, from now on
we will not need to refer to the partition $\lambda\vdash n$, but
only to the partitions $\nu\subset_{\v-\e_{f}}\mu_{f}\subset_{\e_{f}-\f}\lambda'\vdash n-\f$.
For ease of notation, from now on we shall abuse notation and \textbf{write
$\lambda$ instead of $\lambda'$}.

Before moving on, we prove that the recently defined functions $\Upsilon_{n}(\{\sigma_{f}^{\pm},\tau_{f}^{\pm}\},\nu,\{\mu_{f}\},\lambda)$
are analytic functions of $n^{-1}$ when $\nu,\mu_{f},\lambda$ each
vary in a family of Young diagrams. Recall the notation $\lambda\left(n\right)$
and $T\left(n\right)$ from Section \ref{subsec:Families-of-YD-and-zeta}.
\begin{lem}
\label{lem:Upsilon-analytic}Still assume that $(Y,\J)$, $\sigma_{f}^{\pm}$
and $\tau_{f}^{\pm}$ are all fixed in the sense of Sections \ref{subsec:Integrating-over-double-cosets}
and \ref{subsec:Construction-of-labelings}. Suppose that we are given
YDs $\nu\subset_{\v-\f}\lambda$ and for each $f\in\{a,b,c,d\}$ a
YD $\mu_{f}$ with 
\[
\nu\subset_{\v-\e_{f}}\mu_{f}\subset_{\e_{f}-\f}\lambda.
\]
There is a function $\Upsilon^{*}\left(\nu,\{\mu_{f}\},\lambda,\bullet\right)$
that is holomorphic in some open disc in $\C$ with center $0$ such
that for all $n$ sufficiently large (depending on $Y$), $\lambda(n-\f)$,
$\mu_{f}(n-\e_{f})$, and $\nu(n-\v)$ all exist and
\[
\Upsilon_{n}\left(\left\{ \sigma_{f}^{\pm},\tau_{f}^{\pm}\right\} ,\nu(n-\v),\left\{ \mu_{f}(n-\e_{f})\right\} ,\lambda(n-\f)\right)=\Upsilon^{*}\left(\nu,\{\mu_{f}\},\lambda,n^{-1}\right).
\]
In addition, the coefficients of the Taylor series of $\Upsilon^{*}\left(\nu,\{\mu_{f}\},\lambda,\bullet\right)$
are in $\mathbf{Q}.$
\end{lem}

\begin{proof}
The proof relies crucially on property \textbf{P5}\textbf{\emph{ }}of
the permutations $\sigma_{f}^{\pm},\tau_{f}^{\pm}$ stating that they
are obtained from fixed permutations in $S_{\v}$. This means that
each of the summands 
\[
\M\left(\left\{ \sigma_{f}^{\pm},\tau_{f}^{\pm},s_{f}\left(n-\f\right),t_{f}\left(n-\f\right),r_{f}^{-}\left(n-\e_{f}\right),r_{f}^{+}\left(n-\e_{f}\right)\right\} \right)
\]
of 
\[
\Upsilon_{n}\left(\left\{ \sigma_{f}^{\pm},\tau_{f}^{\pm}\right\} ,\nu(n-\v),\left\{ \mu_{f}(n-\e_{f})\right\} ,\lambda(n-\f)\right)
\]
(cf. (\ref{eq:Upsilon-def})) agrees with a function of $n^{-1}$
that is holomorphic in an open disc with center $0$ and with rational
coefficients of its Taylor series by Proposition \ref{prop:holomorphic-matrix-coefficients}.
Since there are only finitely many summands in (\ref{eq:Upsilon-def}),
this proves the lemma.
\end{proof}
We also give a coarse bound for the quantities $\Upsilon_{n}(\{\sigma_{f}^{\pm},\tau_{f}^{\pm}\},\nu,\{\mu_{f}\},\lambda)$;
this will be improved later in Proposition \ref{prop:Upsilon-non-trivial-bound}.
\begin{lem}
\label{lem:coarse-bound-on-Upsilon}We have 
\[
\left|\Upsilon_{n}\left(\left\{ \sigma_{f}^{\pm},\tau_{f}^{\pm}\right\} ,\nu,\left\{ \mu_{f}\right\} ,\lambda\right)\right|\leq\left(d_{\lambda/\nu}\right)^{8}\leq\left((\v-\f)!\right)^{8}.
\]
\end{lem}

\begin{proof}
For fixed $\nu,\{\mu_{f}\},\lambda$, the range of summation in (\ref{eq:Upsilon-def})
is parameterized $1:1$ by the 8 tableaux of the form
\[
r_{f}^{+}\sqcup s_{f},r_{f}^{-}\sqcup t_{f}\in\Tab(\lambda/\nu).
\]
Also, since the matrix coefficients in (\ref{eq:M-def}) involve unit
vectors in a unitary representation, each summand in (\ref{eq:M-def})
is $\leq1$ in absolute value. Hence the lemma follows.
\end{proof}

\subsection{A geometric bound for products of matrix coefficients}

We continue to keep all the notations and assumptions of $\S\S$\ref{subsec:Integrating-over-double-cosets}.
We will show that we can give improved bounds for the product of matrix
coefficients $\M(\{\sigma_{f}^{\pm},\tau_{f}^{\pm},r_{f}^{\pm},s_{f},t_{f}\})$
defined in (\ref{eq:M-def}) in terms of geometric properties of $Y$.
Recall the definitions of the functions $\tp,\lf,d$ from $\S\S\ref{subsec:An-estimate-for-mat-coefs-in-skew-modules}$.
We define
\begin{align}
 & D_{\tp}\left(\left\{ \sigma_{f}^{\pm},\tau_{f}^{\pm},r_{f}^{\pm},s_{f},t_{f}\right\} \right)\eqdf\label{eq:Dtopdef}\\
 & ~~~d\left(\sigma_{b}^{-}\left(\sigma_{a}^{+}\right)^{-1}\tp(r_{a}^{+}\sqcup s_{a}),\tp(r_{b}^{-}\sqcup s{}_{b})\right)+d\left(\tau_{a}^{+}\left(\sigma_{b}^{+}\right)^{-1}\tp(r_{b}^{+}\sqcup s_{b}),\tp(r_{a}^{+}\sqcup t_{a})\right)+\nonumber \\
 & ~~~d\left(\tau_{b}^{+}\left(\tau_{a}^{-}\right)^{-1}\tp(r_{a}^{-}\sqcup t{}_{a}),\tp(r_{b}^{+}\sqcup t_{b})\right)+d\left(\sigma_{c}^{-}\left(\tau_{b}^{-}\right)^{-1}\tp(r_{b}^{-}\sqcup t{}_{b}),\tp(r_{c}^{-}\sqcup s{}_{c})\right)+\nonumber \\
 & ~~~d\left(\sigma_{d}^{-}\left(\sigma_{c}^{+}\right)^{-1}\tp(r_{c}^{+}\sqcup s_{c}),\tp(r_{d}^{-}\sqcup s{}_{d})\right)+d\left(\tau_{c}^{+}\left(\sigma_{d}^{+}\right)^{-1}\tp(r_{d}^{+}\sqcup s_{d}),\tp(r_{c}^{+}\sqcup t_{c})\right)+\nonumber \\
 & ~~~d\left(\tau_{d}^{+}\left(\tau_{c}^{-}\right)^{-1}\tp(r_{c}^{-}\sqcup t{}_{c}),\tp(r_{d}^{+}\sqcup t_{d})\right)+d\left(\sigma_{a}^{-}\left(\tau_{d}^{-}\right)^{-1}\tp(r_{d}^{-}\sqcup t{}_{d}),\tp(r_{a}^{-}\sqcup s{}_{a})\right).\nonumber 
\end{align}
Proposition \ref{prop:bound-for-matrix-coef-in-terms-of-D} directly
implies the following result.
\begin{lem}
~\label{lem:product-of-matrix-coefs-bounded-using-D}If $\lambda_{1}+\nu_{1}>n-\f+\left(\v-\f\right)^{2}$,
then 
\[
\left|\M\left(\left\{ \sigma_{f}^{\pm},\tau_{f}^{\pm},r_{f}^{\pm},s_{f},t_{f}\right\} \right)\right|\leq\left(\frac{(\v-\f)^{2}}{\lambda_{1}+\nu_{1}-(n-\f)}\right)^{D_{\tp}\left(\left\{ \sigma_{f}^{\pm},\tau_{f}^{\pm},r_{f}^{\pm},s_{f},t_{f}\right\} \right)}.
\]
\end{lem}

\begin{rem}
If $\nu$ has a fixed bound on the number of boxes outside its first
row, and $Y$ is fixed, then the hypothesis of Lemma \ref{lem:product-of-matrix-coefs-bounded-using-D}
is satisfied for sufficiently large $n$.
\end{rem}

The quantities $D_{\mathrm{top}}(\{\sigma_{f}^{\pm},\tau_{f}^{\pm},r_{f}^{\pm},s_{f},t_{f}\})$
have a useful interpretation in terms of the combinatorics of the
boundary cycles of $Y$. To explain this, we construct from the data
$\{\sigma_{f}^{\pm},\tau_{f}^{\pm},r_{f}^{\pm},s_{f},t_{f}\}$ a labeling
of the sides/half-sides of rectangles/half-rectangles in $Y_{+}^{(1)}$.

\subsection{Construction of labelings of tiled surfaces from collections of tableaux\label{subsec:Construction-of-labelings}}

In this section, we keep all of the notation from the previous sections.
In particular, we have a fixed vertex-labeled compact tiled surface
$(Y,\J_{n})$ with $\v$ vertices, $\e_{f}$ $f$-labeled edges for
each $f\in\{a,b,c,d\}$, and $\f$ octagons. We fix the data
\begin{align}
\nu\vdash n-\v,\lambda\vdash n-\f\nonumber \\
\nu\subset\mu_{f}\subset_{\e_{f}-\f}\lambda & \quad\forall f\in\{a,b,c,d\}\nonumber \\
r_{f}^{+},r_{f}^{-}\in\Tab(\mu_{f}/\nu),\,s_{f},t_{f}\in\Tab(\lambda/\mu_{f}) & \quad\forall f\in\{a,b,c,d\}\label{eq:data}
\end{align}
All this data uniquely determines one summand of $\Upsilon_{n}(\{\sigma_{f}^{\pm},\tau_{f}^{\pm}\},\nu,\{\mu_{f}\},\lambda)$
as in (\ref{eq:M-def}), and hence also of $\Xi_{n}(Y)$ as in (\ref{eq:Xi-def}). 

Recall that in $\S\S$\ref{subsec:Construction-of-auxiliary-numberings}
we constructed the maps $\sigma_{f}^{\pm}$ and $\tau_{f}^{\pm}$
according to numbering in $\left[\v-\f\right]$ of octagons, of exposed
sides of full-edges and of hanging half-edges of $Y_{+}$. By adding
$n-\v,$ these numbers are in $\left[n-\v+1,n-\f\right]$, give rise
to the images of the corresponding vertices of $Y$ through $\sigma_{f}^{\pm},\tau_{f}^{\pm}$,
and are the elements in the tableaux $r_{f}^{\pm},s_{f},t_{f}$. Given
these tableaux, we assign a `top' label to the hanging half-edges
and exposed sides of full-edges which appear in the top row of the
corresponding tableau. Namely,
\begin{itemize}
\item Every exposed left-side (resp.~right-side) of an $f$-full-edge is
labeled `top' if the corresponding element in $s_{f}$ (resp.~$t_{f}$)
lies in the top row\footnote{To be sure, the top row in this case is row number one of $\lambda/\mu_{f}$,
which may be empty: its length is $\lambda_{1}-\left(\mu_{f}\right)_{1}$.}. 
\item Every outgoing (resp.~incoming) hanging $f$-half-edge is labeled
`top' if the corresponding element in $r_{f}^{-}$ (resp.~$r_{f}^{+}$)
lies in the top row.
\end{itemize}
This labeling scheme is illustrated in Figure \ref{fig:labelling-example}. 

\begin{figure}
\begin{centering}
\includegraphics[clip,scale=0.8]{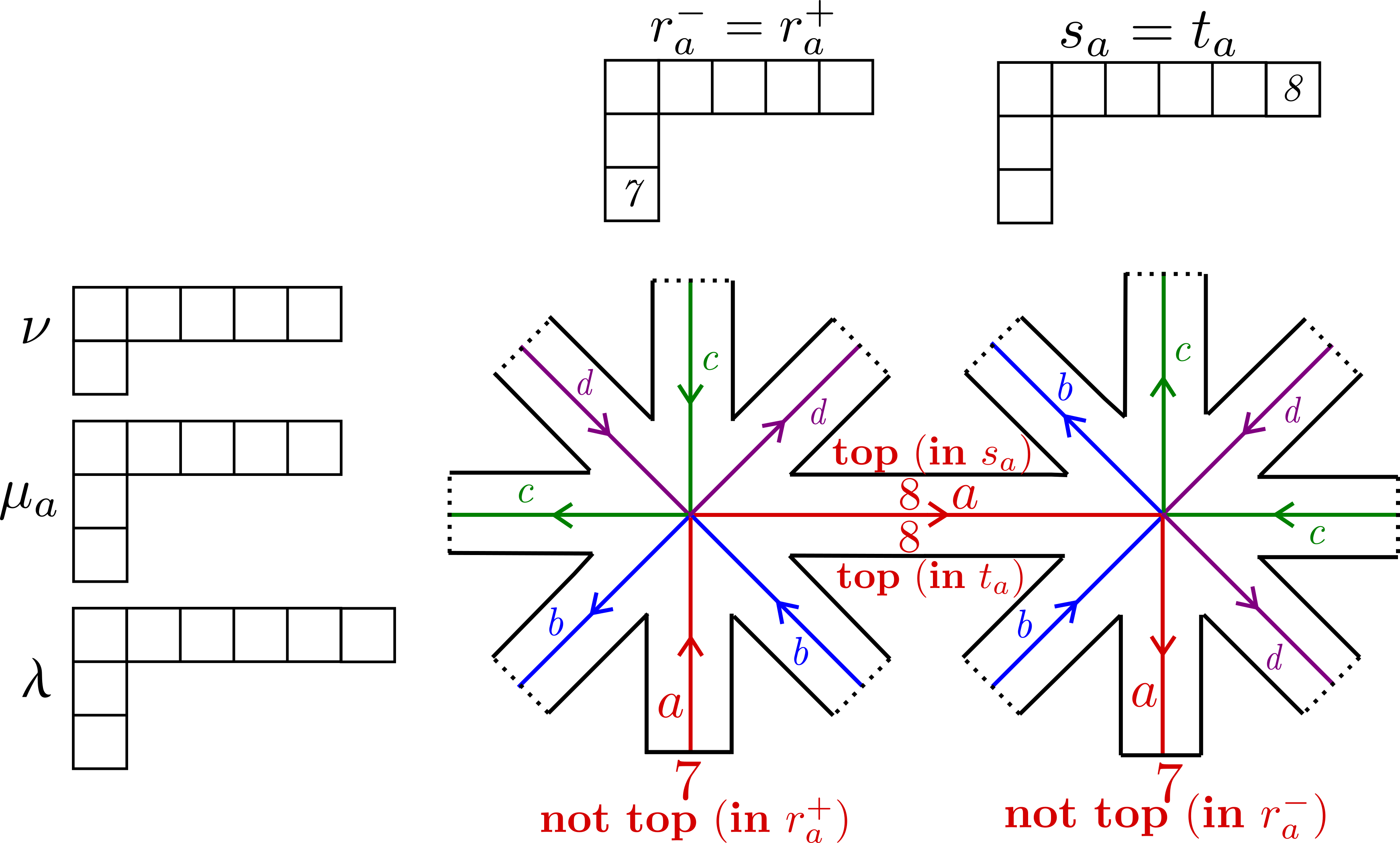}
\par\end{centering}
\caption{Illustration of how the tableaux $r_{f}^{-},r_{f}^{+},s_{f},t_{f}$
induce the `top' labeling.\label{fig:labelling-example}}
\end{figure}

The purpose of introducing these labelings is the following diagrammatic
interpretation of $D_{\tp}$. We view the boundary $\partial Y_{+}$
of (the thick version of) $Y_{+}$ to consist of hanging half-edges
and of exposed sides of full-edges. 
\begin{lem}
\label{lem:D-diagrammatic}The quantity $D_{\tp}(\{\sigma_{f}^{\pm},\tau_{f}^{\pm},r_{f}^{\pm},s_{f},t_{f}\})$
is half the number of incidences between two consecutive parts of
$\partial Y_{+}$ among which one is labeled `top' and other one is
not. 
\end{lem}

\begin{proof}
This follows simply by careful consideration of the definition (\ref{eq:Dtopdef})
of $D_{\tp}$, as the permutations appearing in that definition map
the index of one part of $\partial Y_{+}$ to a neighboring part.
For example, $\sigma_{b}^{-}\left(\sigma_{a}^{+}\right)^{-1}$ maps
the index on an exposed left-side of an $a$-full-edge to the neighboring
index which either belongs to an exposed left-side of a $b$-full-edge
or to a hanging outgoing $b$-half-edge. 

The reason for the factor $\frac{1}{2}$ is that for $A,B$ of the
same size, $d(A,B)$ is half the size of the symmetric difference
of $A$ and $B$. While in the definition of $D_{\tp}$ we count differences
$d\left(A,B\right)$, in counting switches between `top' parts to
`non-top' parts of $\partial Y_{+}$ and vice-versa, we refer to symmetric
differences.
\end{proof}
The benefit of Lemma \ref{lem:D-diagrammatic} is that it allows us
to connect the possible properties of $Y$ being boundary reduced
or strongly boundary reduced to non-trivial bounds for matrix coefficients. 

A \emph{piece }of $\partial Y_{+}$ is a contiguous collection of
exposed sides of full-edges and of hanging half-edges. Given a piece
$P$, we write $\e(P)$ for the number of exposed sides-of-full-edges
in $P$, $\he(P)$ for the number of hanging half-edges in $P$, and
$\chi\left(P\right)$ for the Euler characteristic of $P$, which
may by $0$ (if $P$ is a circle) or $1$ (if $P$ is topologically
a path). A \emph{piece collection }$\mathcal{P}$ of $\partial Y_{+}$
is a collection of \emph{disjoint} pieces of $\partial Y_{+}$ (without
intersection even of endpoints). 
\begin{defn}
\label{def:defect and max defect}Define
\[
\defect(P)\eqdf\e(P)-3\he(P)
\]
and\footnote{For general genus $g\ge2$, the definitions are $\defect\left(P\right)=\e\left(P\right)-\left(2g-1\right)\he\left(P\right)$
and $\sdefect\left(Y\right)=\max_{{\cal P}}\defect\left(P\right)-4g\chi\left(P\right)$.} 
\begin{align*}
\sdefect(Y) & \eqdf\max_{\mathcal{P\neq\emptyset}}\sum_{P\in\mathcal{P}}\defect(P)-8\chi(P)
\end{align*}
where the maximum is over all \uline{nonempty} piece collections
of $\partial Y_{+}$. 
\end{defn}

\begin{lem}
\label{lem:max-defect of BR and SBR} If $Y$ is boundary reduced,
then 
\begin{equation}
\sdefect(Y)\leq0.\label{eq:sup-defect-br}
\end{equation}
If $Y$ is strongly boundary reduced, then
\begin{equation}
\sdefect(Y)\leq-2.\label{eq:sup-defect-sbr}
\end{equation}
\end{lem}

\begin{proof}
Assume first that $Y$ is boundary reduced. Recall that $Y$ has no
long blocks, so no blocks of size $>4$, and no long chains, so between
every two blocks of size $4$ in the same piece, there must be either
two consecutive hanging edges or one block of size $\le2$. As a result,
for every piece $P$ of $\partial Y_{+}$ that is a circle, we have
$\defect(P)\leq0$ so $\defect(P)-8\chi(P)\leq0$. For every piece
$P$ of $\partial Y_{+}$ that is a path, we have $\defect(P)\leq4$:
this bound is attained for example when $P$ corresponds to a block
of size $4$ or to two consecutive blocks of size $4$ and $3$. Therefore,
$\defect(P)-8\chi(P)\leq-4$. Hence all contributions to $\sdefect(Y)$
are non-positive, and we obtain (\ref{eq:sup-defect-br}).

If $Y$ is strongly boundary reduced, there are no blocks in $\partial Y$
of length $>3$, and so every piece $P$ of $\partial Y_{+}$ which
is a path satisfies $\defect(P)\leq3$ and $\defect(P)-8\chi(P)\leq-5$:
this bound is attained when $P$ corresponds to a chain of consecutive
blocks of size $3$ each. If $P$ is a piece of $\partial Y_{+}$
that is a circle but not a cyclic chain, then there are two consecutive
hanging half-edges and $\defect(P)=\defect(P)-\chi(P)\leq-3$. If
$P$ is a cyclic chain, then by \cite[Lemma 3.6]{MPcore} it cannot
have only one block of size $2$, so $\defect(P)=\defect(P)-\chi(P)\leq-2$.
This proves (\ref{eq:sup-defect-sbr}).
\end{proof}
The following result relates the structure of the pieces of $\partial Y$
to the quantities $D_{\mathrm{top}}$ and $D_{\mathrm{\mathrm{left}}}$
appearing in our previous bound (Lemma \ref{lem:product-of-matrix-coefs-bounded-using-D})
for matrix coefficients. Given a skew YD $\lambda/\nu$, we write
$b_{\lambda/\nu}$ for the number of boxes of $\lambda/\nu$ outside
the first row.
\begin{prop}
\label{prop:b-D-defect-ineq}Suppose we are given $\nu,\mu_{f},\lambda,r_{f}^{\pm},s_{f},t_{f}$
as in (\ref{eq:data}). If $b_{\lambda/\nu}>0$, then\footnote{For general $g\ge2$ the bound is $\frac{1}{4g}\sdefect\left(Y\right)$.}
\[
b_{\lambda/\nu}-b_{\mu_{a}/\nu}-b_{\mu_{b}/\nu}-b_{\mu_{c}/\nu}-b_{\mu_{d}/\nu}-D_{\tp}\left(\left\{ \{\sigma_{f}^{\pm},\tau_{f}^{\pm},r_{f}^{\pm},s_{f},t_{f}\right\} \right)\leq\frac{1}{8}\sdefect(Y).
\]
\end{prop}

\begin{proof}
We define a collection of pieces $\mathcal{P}$ of $\partial Y_{+}$
according to the `top' labels: pieces are contiguous segments of $\partial Y_{+}$
(hanging half-edges or exposed sides of full edges) which are \textbf{not
}labeled `top'. This collection is non-empty if and only if $b_{\lambda/\nu}>0$,
which holds by assumption. Let $\mathcal{P}_{0}$ denote the collection
of such pieces that are circles, $\mathcal{P}_{1}$ denote the collection
of such pieces that are paths, and $\mathcal{P}=\mathcal{P}_{0}\sqcup\mathcal{P}_{1}$.
It follows from Lemma \ref{lem:D-diagrammatic} that
\[
D_{\mathrm{top}}\eqdf D_{\tp}\left(\left\{ \sigma_{f}^{\pm},\tau_{f}^{\pm},r_{f}^{\pm},s_{f},t_{f}\right\} \right)=\left|\mathcal{P}_{1}\right|.
\]
Now $b_{\lambda/\nu}=\frac{1}{8}\sum_{P\in\mathcal{P}}\left[\e\left(P\right)+\he\left(P\right)\right]$
because for every $f\in\left\{ a,b,c,d\right\} $, every square outside
the top row of $\lambda/\nu$ corresponds either to two hanging half-edges
(in case this square belongs to $\mu_{f}/\nu$), or to two exposed-sides-of-full-edges
(in case this square belongs to $\lambda/\mu_{f}$). Similarly, $b_{\mu_{a}/\nu}+b_{\mu_{b}/\nu}+b_{\mu_{c}/\nu}+b_{\mu_{d}/\nu}=\frac{1}{2}\sum_{P\in\mathcal{P}}\he\left(P\right)$.
Thus,

\begin{align*}
b_{\lambda/\nu}-b_{\mu_{a}/\nu}-b_{\mu_{b}/\nu}-b_{\mu_{c}/\nu}-b_{\mu_{d}/\nu}-D_{\mathrm{top}} & =\sum_{P\in\mathcal{\mathcal{P}}}\left(\frac{1}{8}\left(\e(P)+\he(P)\right)-\frac{1}{2}\he(P)\right)-\sum_{P\in\mathcal{P}_{1}}1\\
 & =\sum_{P\in\mathcal{P}_{0}}\frac{1}{8}\defect(P)+\sum_{P\in\mathcal{P}_{1}}\frac{1}{8}\left(\defect(P)-8\right)\\
 & =\sum_{P\in\mathcal{P}}\frac{1}{8}\left(\defect(P)-8\chi(P)\right)\leq\frac{1}{8}\sdefect(Y).
\end{align*}
\end{proof}
\begin{prop}
\label{prop:geometric-D-bounds}Suppose we are given $\nu,\mu_{f},\lambda,r_{f}^{-},r_{f}^{+},s_{f},t_{f}$
as in (\ref{eq:data}). 
\begin{enumerate}
\item If $Y$ is boundary reduced, 
\begin{equation}
D_{\tp}\left(\left\{ \sigma_{f}^{\pm},\tau_{f}^{\pm},r_{f}^{\pm},s_{f},t_{f}\right\} \right)\geq b_{\lambda/\nu}-b_{\mu_{a}/\nu}-b_{\mu_{b}/\nu}-b_{\mu_{c}/\nu}-b_{\mu_{d}/\nu}.\label{eq:D-bound-simple}
\end{equation}
\item If $Y$ is strongly boundary reduced, then (\ref{eq:D-bound-simple})
becomes an equality if and only if 
\begin{equation}
D_{\tp}\left(\left\{ \sigma_{f}^{\pm},\tau_{f}^{\pm},r_{f}^{\pm},s_{f},t_{f}\right\} \right)=b_{\lambda/\nu}=b_{\mu_{a}/\nu}=b_{\mu_{b}/\nu}=b_{\mu_{c}/\nu}=b_{\mu_{d}/\nu}=0.\label{eq:D-strong-equality}
\end{equation}
Otherwise, 
\begin{equation}
D_{\tp}\left(\left\{ \sigma_{f}^{\pm},\tau_{f}^{\pm},r_{f}^{\pm},s_{f},t_{f}\right\} \right)\geq b_{\lambda/\nu}-b_{\mu_{a}/\nu}-b_{\mu_{b}/\nu}-b_{\mu_{c}/\nu}-b_{\mu_{d}/\nu}+1.\label{eq:D-bound-plus-one}
\end{equation}
\end{enumerate}
\end{prop}

\begin{proof}
Note that for any $Y$, if $b_{\lambda/\nu}=0$ then $D_{\tp}(\{\sigma_{f}^{\pm},\tau_{f}^{\pm},r_{f}^{\pm},s_{f},t_{f}\})=b_{\mu_{a}/\nu}=b_{\mu_{b}/\nu}=b_{\mu_{c}/\nu}=b_{\mu_{d}/\nu}=0$.
Otherwise, Proposition \ref{prop:b-D-defect-ineq} applies, and we
obtain the statement of the proposition by combining Proposition \ref{prop:b-D-defect-ineq}
with the inequalities (\ref{eq:sup-defect-br}) and (\ref{eq:sup-defect-sbr}).
To obtain (\ref{eq:D-bound-plus-one}) from (\ref{eq:sup-defect-sbr})
(instead of a bound featuring $\frac{1}{4}$), one uses that all $b_{\bullet}$
quantities and $D_{\tp}(\{\sigma_{f}^{\pm},\tau_{f}^{\pm},r_{f}^{\pm},s_{f},t_{f}\})$
are integer valued.
\end{proof}
Proposition \ref{prop:geometric-D-bounds} together with Lemma \ref{lem:D-diagrammatic}
have the following important consequence for the quantities $\Upsilon_{n}(\{\sigma_{f}^{\pm},\tau_{f}^{\pm}\},\nu,\{\mu_{f}\},\lambda)$.
\begin{prop}
\label{prop:Upsilon-non-trivial-bound}Suppose that $\nu,\{\mu_{f}\},\lambda$
are as in (\ref{eq:data}), and that $\lambda_{1}+\nu_{1}>n-\f+(\v-\f)^{2}$. 
\begin{enumerate}
\item If $Y$ is boundary reduced, then
\begin{align}
\left|\Upsilon_{n}\left(\left\{ \sigma_{f}^{\pm},\tau_{f}^{\pm}\right\} ,\nu,\{\mu_{f}\},\lambda\right)\right| & \leq\left((\v-\f)!\right)^{8}\left(\frac{(\v-\f)^{2}}{\nu_{1}+\lambda_{1}-(n-\f)}\right)^{b_{\lambda/\nu}-b_{\mu_{a}/\nu}-b_{\mu_{b}/\nu}-b_{\mu_{c}/\nu}-b_{\mu_{d}/\nu}}.\label{eq:Upsilon-bound-a}
\end{align}
\item If $Y$ is strongly boundary reduced and $b_{\text{\ensuremath{\lambda/}}\nu}>0$,
then
\begin{equation}
\left|\Upsilon_{n}\left(\left\{ \sigma_{f}^{\pm},\tau_{f}^{\pm}\right\} ,\nu,\{\mu_{f}\},\lambda\right)\right|\leq\left((\v-\f)!\right)^{8}\left(\frac{(\v-\f)^{2}}{\nu_{1}+\lambda_{1}-(n-\f)}\right)^{1+b_{\lambda/\nu}-b_{\mu_{a}/\nu}-b_{\mu_{b}/\nu}-b_{\mu_{c}/\nu}-b_{\mu_{d}/\nu}}.\label{eq:Upsilon-bound-b}
\end{equation}
\item For any $Y$, if $b_{\lambda/\nu}=0$, then
\[
\Upsilon_{n}\left(\left\{ \sigma_{f}^{\pm},\tau_{f}^{\pm}\right\} ,\nu,\{\mu_{f}\},\lambda\right)=1.
\]
\end{enumerate}
\end{prop}

\begin{proof}
\emph{Part 1. }Suppose that $Y$ is boundary reduced. As argued in
the proof of \ref{lem:coarse-bound-on-Upsilon}, there are at most
$(d_{\lambda/\nu})^{8}\leq((\v-\f)!)^{8}$ summands in the definition
(\ref{eq:Upsilon-def}) of $\Upsilon_{n}(\{\sigma_{f}^{\pm},\tau_{f}^{\pm}\},\nu,\{\mu_{f}\},\lambda)$.
Each summand is some $\M(\{\sigma_{f}^{\pm},\tau_{f}^{\pm},r_{f}^{\pm},s_{f},t_{f}\})$,
so by Lemma \ref{lem:product-of-matrix-coefs-bounded-using-D} and
Proposition \ref{prop:geometric-D-bounds} Part 1, we get that each
summand of (\ref{eq:Upsilon-def}) has absolute value
\[
\leq\left(\frac{(\v-\f)^{2}}{\nu_{1}+\lambda_{1}-(n-\f)}\right)^{b_{\lambda/\nu}-b_{\mu_{a}/\nu}-b_{\mu_{b}/\nu}-b_{\mu_{c}/\nu}-b_{\mu_{d}/\nu}}.
\]
This proves (\ref{eq:Upsilon-bound-a}). 

\emph{Part 2. }Suppose now that $Y$ is strongly boundary reduced
and that $b_{\lambda/\nu}>0$. This time, Lemma \ref{lem:product-of-matrix-coefs-bounded-using-D}
and Proposition \ref{prop:geometric-D-bounds} give that each summand
of $\Upsilon_{n}(\{\sigma_{f}^{\pm},\tau_{f}^{\pm}\},\nu,\{\mu_{f}\},\lambda)$
in (\ref{eq:Upsilon-def}) has absolute value
\[
\leq\left(\frac{(\v-\f)^{2}}{\nu_{1}+\lambda_{1}-(n-\f)}\right)^{1+b_{\lambda/\nu}-b_{\mu_{a}/\nu}-b_{\mu_{b}/\nu}-b_{\mu_{c}/\nu}-b_{\mu_{d}/\nu}}.
\]
As in Part 1, there are $\leq((\v-\f)!)^{8}$ summands of (\ref{eq:Upsilon-def}),
so this proves (\ref{eq:Upsilon-bound-b}). 

\emph{Part 3. }Suppose that $b_{\lambda/\nu}=0$. Then there is only
one possible choice for the tableaux $r_{f}^{-},r_{f}^{+},s_{f},t_{f}$
in (\ref{eq:Upsilon-def}). Moreover, $V^{\lambda/\nu}$ is the trivial
module for the relevant copy of $S_{\v-\f}$, hence the product of
matrix coefficients appearing in (\ref{eq:M-def}) is equal to $1$,
and $\Upsilon_{n}(\{\sigma_{f}^{\pm},\tau_{f}^{\pm}\},\nu,\{\mu_{f}\},\lambda)=1$. 
\end{proof}

\subsection{Stronger bounds for matrix coefficients\label{subsec:Stronger-bounds-for-matrix-coefficients}}

In this section we give a strengthening of Proposition \ref{prop:b-D-defect-ineq}
that is used in a sequel to this paper \cite{magee2020random}. A
reader who is only interested in the current paper may skip this short
$\S\S$\ref{subsec:Stronger-bounds-for-matrix-coefficients}.

\begin{prop}
\label{prop:b-D-defect-ineq-stronger}Suppose we are given Young diagrams
$\nu\vdash n-\v$, $\lambda\vdash n-\f$, and $\mu_{f}$ such that
\[
\nu\subset\mu_{f}\subset_{\e_{f}-\f}\lambda,\quad\forall f\in\{a,b,c,d\},
\]
and tableaux $r_{f}^{+},r_{f}^{-}\in\Tab(\mu_{f}/\nu)$ and $s_{f},t_{f}\in\Tab(\lambda/\mu_{f})$.
Fix $\varepsilon\ge0$ and suppose in addition that for every piece
$P$ of $\partial Y$ we have 
\[
\defect(P)-4\chi(P)\leq-\varepsilon\left(\e(P)+\he(P)\right).
\]
Then
\[
b_{\lambda/\nu}-b_{\mu_{a}/\nu}-b_{\mu_{b}/\nu}-b_{\mu_{c}/\nu}-b_{\mu_{d}/\nu}-D_{\tp}\left(\left\{ \sigma_{f}^{\pm},\tau_{f}^{\pm},r_{f}^{\pm},s_{f},t_{f}\right\} \right)\leq-\varepsilon b_{\lambda/\nu}.
\]
\end{prop}

\begin{proof}
We follow the same construction of a piece collection $\mathcal{P}$
as in the proof of Proposition \ref{prop:b-D-defect-ineq}. This leads
to
\begin{eqnarray*}
b_{\lambda/\nu}-b_{\mu_{a}/\nu}-b_{\mu_{b}/\nu}-b_{\mu_{c}/\nu}-b_{\mu_{d}/\nu}-D_{\mathrm{top}} & = & \frac{1}{8}\sum_{P\in\mathcal{P}}\defect(P)-8\chi(P)\le\frac{1}{8}\sum_{P\in\mathcal{P}}\defect(P)-4\chi(P)\\
 & \le & -\frac{\varepsilon}{8}\sum_{P\in\mathcal{P}}\left(\e(P)+\he(P)\right)=-\varepsilon b_{\lambda/\nu},
\end{eqnarray*}
as required. 
\end{proof}

\subsection{Approximating $\Xi_{n}(Y)$ by Laurent polynomials\label{subsec:Approximating Xi by Laurent}}

In this section we keep all the notations and assumptions of $\S\S$\ref{subsec:Integrating-over-double-cosets}.
We want to show that we can replace the summation over $\nu,\mu_{f}$,
and $\lambda$ given the definition of $\Xi_{n}\left(Y\right)$ in
(\ref{eq:Xi-def}) by a sum of finite size, independent of $n$, at
the cost of a controllable error term. To state this precisely, recalling
the definition of $\Lambda(n,b)$ from $\S\S$\ref{subsec:Families-of-YD-and-zeta},
and letting $b\in\N$, we introduce the quantity
\begin{equation}
\Xi_{n}^{(b)}\left(Y\right)\eqdf\sum_{\substack{\nu\subset_{\v-\f}\lambda\vdash\,n-\f\\
\nu\notin\Lambda(n-\v,b)
}
}d_{\lambda}d_{\nu}\sum_{\nu\subset\mu_{f}\subset_{\e_{f}-\f}\lambda}\frac{1}{d_{\mu_{a}}d_{\mu_{b}}d_{\mu_{c}}d_{\mu_{d}}}\Upsilon_{n}\left(\left\{ \sigma_{f}^{\pm},\tau_{f}^{\pm}\right\} ,\nu,\{\mu_{f}\},\lambda\right).\label{eq:xi-b-def}
\end{equation}
In this summation, we restrict to $\nu$ that have less than $b$
boxes either outside the first row or outside the first column. Note
that whereas $\Xi_{n}\left(Y\right)$ does not depend on any of the
choices of $g^{0}$ and numberings we made in Sections \ref{subsec:Tiled-surfaces-and-double-cosets}
and \ref{subsec:Construction-of-auxiliary-numberings} (see Remark
\ref{rem:Xi depends only on Y}), $\Xi_{n}^{\left(b\right)}\left(Y\right)$
may depend on these choices.
\begin{lem}
\label{lem:truncated Xi}For a fixed tiled surface $Y$ and $b\in\N$,
for any vertex-ordering $\J$ of $Y$ as above, we have
\[
\Xi_{n}(Y)=\Xi_{n}^{(b)}(Y)+O\left(n^{\v-\f-2b}\right)
\]
as $n\to\infty$. The implied constant depends on $b$, $\v$, and
$\f$.
\end{lem}

\begin{proof}
Note that $d_{\mu_{f}}\geq d_{\nu}$, and for fixed $\nu$ and $\lambda$
the number of $\mu_{f}$ with $\nu\subset\mu_{f}\subset_{\e_{f}-\f}\lambda$
is $\leq(\v-\f)!$ (there is an injection from the collection of such
$\mu_{f}$ to $\Tab(\lambda/\nu)$ by filling the boxes of $\mu_{f}/\nu$
with $[n-\v+1,n-\e_{f}]$ and the other boxes of $\lambda/\nu$ arbitrarily).
Using these observations together with Lemma \ref{lem:coarse-bound-on-Upsilon}
we obtain
\begin{align*}
\left|\Xi_{n}(Y)-\Xi_{n}^{(b)}(Y)\right| & =\left|\sum_{\substack{\nu\subset_{\v-\f}\lambda\vdash\,n-\f\\
\nu\in\Lambda(n-\v,b)
}
}d_{\lambda}d_{\nu}\sum_{\nu\subset\mu_{f}\subset_{\e_{f}-\f}\lambda}\frac{1}{d_{\mu_{a}}d_{\mu_{b}}d_{\mu_{c}}d_{\mu_{d}}}\Upsilon\left(\left\{ \sigma_{f}^{\pm},\tau_{f}^{\pm}\right\} ,\nu,\{\mu_{f}\},\lambda\right)\right|\\
 & \leq\left((\v-\f)!\right)^{12}\sum_{\substack{\nu\subset_{\v-\f}\lambda\vdash\,n-\f\\
\nu\in\Lambda(n-\v,b)
}
}\frac{d_{\lambda}}{d_{\nu}^{3}}.
\end{align*}
By Lemma \ref{lem:induced-rep-dimension}, for a fixed $\nu\vdash n-\v$,
we have
\[
\sum_{\lambda:\,\nu\subset_{\v-\f}\lambda}d_{\lambda}\leq\sum_{\lambda:\nu\subset_{\v-\f}\lambda}d_{\lambda}d_{\lambda/\nu}=\frac{(n-\f)!}{(n-\v)!}d_{\nu}.
\]
Therefore, by Proposition \ref{prop:Liebeck-Shalev-non-effective},
\begin{align*}
\left|\Xi_{n}(Y)-\Xi_{n}^{(b)}(Y)\right| & \leq\left((\v-\f)!\right)^{12}\frac{(n-\f)!}{(n-\v)!}\sum_{\nu\in\Lambda(n-\v,b)}\frac{1}{d_{\nu}^{2}}\\
 & =O_{b,\v,\f}\left(n^{\v-\f-2b}\right).
\end{align*}
\end{proof}
\begin{prop}
\label{prop:Xi-Laurent}For any $M\in\N$, there is a Laurent polynomial
$\Xi_{M}^{*}(Y)\in\mathbf{Q}\left[t,t^{-1}\right]$ such that as $n\to\infty$
\[
\Xi_{n}(Y)=\Xi_{M}^{*}(Y)[n]+O\left(n^{-M}\right).
\]
\end{prop}

\begin{proof}
Let $b=\left\lceil \frac{\v-\f+M}{2}\right\rceil $. Then Lemma \ref{lem:truncated Xi}
yields that as $n\to\infty$, 
\begin{equation}
\Xi_{n}\left(Y\right)=\Xi_{n}^{(b)}(Y)+O\left(n^{-M}\right).\label{eq:xi-approx}
\end{equation}
Similarly to the proof of Proposition \ref{prop:zeta-polynomial},
we note that for $n-\v>2b$, the collection of $\nu\vdash n-\v$ such
that $\nu\notin\Lambda(n-\v,b)$ is the disjoint union of $\Lambda_{b_{\lambda}<b}(n-\v)=\left\{ \nu\vdash n-\v\,\middle|\,\nu_{1}>n-\v-b\right\} $
and the dual partitions $\{\check{\nu}\,|\,\nu\in\Lambda_{b_{\lambda}<b}\left(n-\v\right)\}$.
Because each $\mu_{f}$ and $\lambda$ in (\ref{eq:xi-b-def}) extend
$\nu$ by a fixed number of boxes, there is a finite number $\ell$
of tuples of YDs
\[
\left(\nu^{i},\mu_{a}^{i},\mu_{b}^{i},\mu_{c}^{i},\mu_{d}^{i},\lambda{}^{i}\right),\quad1\leq i\leq\ell
\]
with $\nu^{i}\in\Lambda_{b_{\lambda}<b}(2b)$ such that for each $1\leq i\leq\ell$
and $f\in\{a,b,c,d\}$
\[
\nu^{i}\subset_{\v-\e_{f}}\mu_{f}^{i}\subset_{\e_{f}-\f}\lambda^{i}.
\]
Thus (\ref{eq:xi-b-def}) can be rewritten as
\begin{align*}
\Xi_{n}^{(b)}(Y) & =\sum_{i=1}^{\ell}\frac{d_{\lambda^{i}(n-\f)}d_{\nu^{i}(n-\v)}}{d_{\mu_{a}^{i}(n-\e_{a})}d_{\mu_{b}^{i}(n-\e_{b})}d_{\mu_{c}^{i}(n-\e_{c})}d_{\mu_{d}^{i}(n-\e_{d})}}\cdot\left[\Upsilon\left(\left\{ \sigma_{f}^{\pm},\tau_{f}^{\pm}\right\} ,\nu^{i}(n-\v),\left\{ \mu_{f}^{i}(n-\e_{f})\right\} ,\lambda^{i}(n-\f)\right)\right.\\
 & ~~~~~~~~~~~~~~~~~~~~~~~~~~~~~~~~~~~~~~~~~~~~~~~~~~~~~\left.+\Upsilon\left(\left\{ \sigma_{f}^{\pm},\tau_{f}^{\pm}\right\} ,\check{\nu^{i}(n-\v)},\left\{ \check{\mu_{f}^{i}(n-\e_{f})}\right\} ,\check{\lambda^{i}(n-\f)}\right)\right]\\
 & =2\sum_{i=1}^{\ell}\frac{d_{\lambda^{i}(n-\f)}d_{\nu^{i}(n-\v)}}{d_{\mu_{a}^{i}(n-\e_{a})}d_{\mu_{b}^{i}(n-\e_{b})}d_{\mu_{c}^{i}(n-\e_{c})}d_{\mu_{d}^{i}(n-\e_{d})}}\cdot\Upsilon\left(\left\{ \sigma_{f}^{\pm},\tau_{f}^{\pm}\right\} ,\nu^{i}(n-\v),\left\{ \mu_{f}^{i}(n-\e_{f})\right\} ,\lambda^{i}(n-\f)\right)
\end{align*}
where the last line used Lemma \ref{lem:upsilon-sym}. For each $i$,
the ratio of dimensions above agrees with a rational function $\mathcal{Q}_{i}(n)\in\mathbf{Q}(n)$
of $n$ (at least when $n-\v\geq2b$) by Lemma \ref{lem:dim-polynomial}.
Combining this with Lemma \ref{lem:Upsilon-analytic} gives that 
\[
\Xi_{n}^{(b)}(Y)=2\sum_{i=1}^{\ell}\mathcal{Q}_{i}(n)\Upsilon^{*}\left(\nu^{i},\left\{ \mu_{f}^{i}\right\} ,\lambda^{i},n^{-1}\right)
\]
agrees with $F(n^{-1})$, where $F$ is a function of a complex variable
$z$ that is meromorphic in an open disc with center $0$, and with
coefficients of its Laurent series in $\mathbf{Q}$. Hence $F(n^{-1})$
itself can be approximated to order $O(n^{-M})$ as $n\to\infty$
by a Laurent polynomial in $n$ with coefficients in $\mathbf{Q}$.
\end{proof}

\subsection{Estimating $\Xi_{n}(Y)$\label{subsec:Estimating-Xi}}

In this $\S\S$\ref{subsec:Estimating-Xi} we give estimates for $\Xi_{n}(Y)$
for fixed $Y$ which is boundary reduced or strongly boundary reduced.
\begin{prop}
\label{prop:xi-bound-for-boundary reduced}If $Y$ is a boundary reduced
tiled surface then as $n\to\infty$,
\[
\Xi_{n}(Y)=O_{Y}(1).
\]
\end{prop}

\begin{proof}
By Proposition \ref{lem:truncated Xi} there is some $b=b(Y)$ such
that 
\[
\Xi_{n}(Y)=\Xi_{n}^{(b)}(Y)+O_{Y}\left(n^{-1}\right).
\]
As before, let $b_{\lambda}\eqdf\left|\lambda\right|-\lambda_{1}$.
As in the proof of Proposition \ref{prop:Xi-Laurent}, for $n-\v>2b$
we can write 
\begin{align}
\Xi_{n}^{(b)}(Y) & =2\sum_{\begin{gathered}\nu\vdash n-\v~:~b_{\nu}<b\\
\nu\subset_{\v-\e_{f}}\mu_{f}\subset_{\e_{f}-\f}\lambda
\end{gathered}
}\frac{d_{\lambda}d_{\nu}}{d_{\mu_{a}}d_{\mu_{b}}d_{\mu_{c}}d_{\mu_{d}}}\Upsilon\left(\left\{ \sigma_{f}^{\pm},\tau_{f}^{\pm}\right\} ,\nu,\{\mu_{f}\},\lambda\right).\label{eq:Xib-symmetrized formula}
\end{align}
Note that in (\ref{eq:Xib-symmetrized formula}), each of 
\[
\lambda/\nu,\mu_{a}/\nu,\mu_{b}/\nu,\mu_{c}/\nu,\mu_{d}/\nu
\]
has $\leq\v-\f$ boxes outside their first row, so Lemma \ref{lem:dim-ratio-bound}
implies that 
\begin{equation}
\frac{d_{\lambda}d_{\nu}}{d_{\mu_{a}}d_{\mu_{b}}d_{\mu_{c}}d_{\mu_{d}}}\ll_{Y}\frac{1}{d_{\nu}^{~2}}n^{b_{\lambda}-b_{\mu_{a}}-b_{\mu_{b}}-b_{\mu_{c}}-b_{\mu_{d}}+3b_{\nu}}=\frac{1}{d_{\nu}^{~2}}n^{b_{\lambda/\nu}-b_{\mu_{a}/\nu}-b_{\mu_{b}/\nu}-b_{\mu_{c}/\nu}-b_{\mu_{d}/\nu}},\label{eq:dim-ratio-asymptotic}
\end{equation}
(recall the notation $\ll$ from $\S\S$\ref{subsec:Notation}). Since
all $\lambda,\nu$ in the sum (\ref{eq:Xib-symmetrized formula})
have a bounded number of boxes outside their first row, depending
on $Y$, for sufficiently large $n$, the condition $\lambda_{1}+\nu_{1}>n-\f+(\v-\f)^{2}$
of Proposition \ref{prop:Upsilon-non-trivial-bound} is met for large
$n$. Hence by Proposition \ref{prop:Upsilon-non-trivial-bound} Part
1, 

\begin{align*}
\left|\Xi_{n}^{(b)}\left(Y\right)\right| & \ll_{Y}\sum_{\nu\vdash n-\v~:~b_{\nu}<b}\frac{1}{d_{\nu}^{2}}\left[\sum_{\nu\subset\mu_{f}\subset_{\e_{f}-\f}\lambda}\left(n\cdot\frac{(\v-\f)^{2}}{\lambda_{1}+\nu_{1}-(n-\f)}\right)^{b_{\lambda/\nu}-b_{\mu_{a}/\nu}-b_{\mu_{b}/\nu}-b_{\mu_{c}/\nu}-b_{\mu_{d}/\nu}}\right]\\
 & \ll_{Y}\sum_{\nu\vdash n-\v~:~b_{\nu}<b}\frac{1}{d_{\nu}^{2}}\left[\sum_{\nu\subset\mu_{f}\subset_{\e_{f}-\f}\lambda}1\right]\ll_{Y}\sum_{\nu\vdash n-\v~:~b_{\nu}<b}\frac{1}{d_{\nu}^{2}}\le\zeta^{S_{n-\v}}\left(2\right)=2+O\left(\frac{1}{n^{2}}\right),
\end{align*}
where the second asymptotic inequality follows as $n\cdot\frac{(\v-\f)^{2}}{\lambda_{1}+\nu_{1}-(n-\f)}\ll1$
uniformly over all $\nu,\lambda$ in play, and $b_{\lambda/\nu}-b_{\mu_{a}/\nu}-b_{\mu_{b}/\nu}-b_{\mu_{c}/\nu}-b_{\mu_{d}/\nu}$
is bounded from above by a constant, the third asymptotic inequality
follows as the number of YDs that extend a given $\nu$ by at most
$\v$ boxes is bounded independently of $n$, and the last asymptotic
inequality follows by Proposition \ref{prop:Liebeck-Shalev-non-effective}
with $b=1$.
\end{proof}
If $Y$ is strongly boundary reduced then we get a finer estimate.
\begin{prop}
\label{prop:xi-bound-for-strongly-boundary reduced}If $Y$ is a strongly
boundary reduced tiled surface then as $n\to\infty$,
\begin{equation}
\Xi_{n}(Y)=2+O_{Y}\left(n^{-1}\right).\label{eq:Xi-asymptotic-fine}
\end{equation}
\end{prop}

\begin{proof}
We begin as in the proof of Proposition \ref{prop:xi-bound-for-boundary reduced}
by choosing $b(Y)$ such that $\Xi_{n}(Y)=\Xi_{n}^{(b)}(Y)+O_{Y}(n^{-1})$
and (\ref{eq:Xib-symmetrized formula}) holds. It now suffices to
prove the proposition with $\Xi_{n}(Y)$ replaced with $\Xi_{n}^{(b)}(Y)$.

There are summands of (\ref{eq:Xib-symmetrized formula}) corresponding
to $\lambda/\nu$ having all boxes in the first row. By Proposition
\ref{prop:Upsilon-non-trivial-bound} Part 3, each of these summands
contributes $2\cdot\frac{d_{\lambda}d_{\nu}}{d_{\mu_{a}}d_{\mu_{b}}d_{\mu_{c}}d_{\mu_{d}}}$
to $\Xi_{n}^{(b)}$, but in this case $\nu,\mu_{f},\lambda$ all belong
to the same family of YDs, so this contribution is $\frac{2}{d_{\nu}^{2}}\left(1+O\left(n^{-1}\right)\right)$.
As there is one of these summands for each $\nu\vdash n-\v$ with
$b_{\nu}<b$, together these contribute 
\begin{equation}
2\left(1+O\left(\frac{1}{n}\right)\right)\sum_{\nu\vdash n-\v~:~b_{\nu}<b}\frac{1}{d_{\nu}^{2}}=2+O\left(\frac{1}{n}\right)\label{eq:main-term-appearance}
\end{equation}
by Lemma \ref{lem:dim-polynomial}. The constant term $2$ appearing
in (\ref{eq:main-term-appearance}) is the main term of (\ref{eq:Xi-asymptotic-fine}).

For any other summand of (\ref{eq:Xib-symmetrized formula}) $b_{\lambda/\nu}>0$,
and so by Proposition \ref{prop:Upsilon-non-trivial-bound} Part 2
combined with (\ref{eq:dim-ratio-asymptotic}),
\[
\frac{d_{\lambda}d_{\nu}}{d_{\mu_{a}}d_{\mu_{b}}d_{\mu_{c}}d_{\mu_{d}}}\Upsilon\left(\left\{ \sigma_{f}^{\pm},\tau_{f}^{\pm}\right\} ,\nu,\{\mu_{f}\},\lambda\right)\ll_{Y}\frac{1}{d_{\nu}^{2}}\cdot\frac{1}{n},
\]
so as argued in the proof of Proposition \ref{prop:xi-bound-for-boundary reduced},
the total contribution of these summands is $O_{Y}\left(n^{-1}\right)$.
Hence
\[
\Xi_{n}^{(b)}(Y)=2+O_{Y}\left(\frac{1}{n}\right).
\]
\end{proof}

\section{Proofs of main theorems\label{sec:Proofs-of-main-theorems}}

\subsection{Proof of Theorem \ref{thm:rational-approx} and its extension to
finitely generated subgroups\label{subsec:Proof-of-Theorem-analytic}}

We give the proof when $g=2$; the extension to other $g\geq2$ is
clear. We are given a finitely generated subgroup $J\leq\Gamma=\Gamma_{2}$
and $M\in\mathbf{N}$ and we wish to show that $\E_{n}[\fix_{J}]\eqdf\E_{2,n}[\fix_{J}]$
can be approximated to order $O(n^{-M})$ by a Laurent polynomial
of $n$ with rational  coefficients. Given this, the trivial bound
$\E_{n}[\fix_{J}]\leq n$ implies that the Laurent polynomial takes
the form (\ref{eq:Laurent-polynomial}). The fact that the $a_{i}(J)$
do not depend on $M$ is clear. By setting $J=\langle\gamma\rangle$
we obtain Theorem \ref{thm:rational-approx} from this. 

By Lemma \ref{lem:E(fix) as expected number of lifts}, we have 
\begin{equation}
\E_{n}\left[\fix_{J}\right]=\E_{n}\left(\core(J)\right).\label{eq:epxected-number-of-fix-to-prob-for-core}
\end{equation}
Now let $\mathcal{R}$ be \uline{any} finite resolution of $\core(J)$;
by Theorem \ref{thm:existence-of-resolution} at least one exists.

Each element of this resolution is a morphism $h:\core(J)\to W_{h}$
of tiled surfaces. By Lemma \ref{lem:resolution-sum-of-probabilities}
\[
\E_{n}(\core(J))=\sum_{h\in\mathcal{R}}\mathbb{\E}_{n}^{\emb}\left(W_{h}\right).
\]
Now using Theorem \ref{thm:E_n-emb-exact-expression} for each of
the terms $\E_{n}^{\emb}(W_{h})$ gives 
\[
\E_{n}\left[\fix_{J}\right]=\frac{\left(n!\right)^{3}}{\left|\X_{n}\right|}\sum_{h\in\mathcal{R}}\frac{\left(n\right)_{\v\left(W_{h}\right)}\left(n\right)_{\f\left(W_{h}\right)}}{\prod_{f\in\{a,b,c,d\}}\left(n\right)_{\e_{f}\left(W_{h}\right)}}\Xi_{n}\left(W_{h}\right)
\]
where $\v(W_{h})$, $\e_{f}(W_{h})$, $\f(W_{h})$ are the number
of vertices, $f$-labeled edges ($f\in\{a,b,c,d\}$), and octagons,
respectively, of $W_{h}$. Also recall the definition of $\Xi_{n}$
from Theorem \ref{thm:E_n-emb-exact-expression}. By (\ref{eq:Hurwitz}),
$\left|\X_{n}\right|=\left(n!\right)^{3}\cdot\zeta^{S_{n}}\left(2\right)$,
and so
\begin{equation}
\E_{n}\left[\fix_{J}\right]=\frac{1}{\zeta^{S_{n}}(2)}\sum_{h\in\mathcal{R}}\frac{(n)_{\v(W_{h})}(n)_{\f(W_{h})}}{\prod_{f\in\{a,b,c,d\}}(n)_{\e_{f}(W_{h})}}\Xi_{n}\left(W_{h}\right).\label{eq:expectation-expression}
\end{equation}
Now we note:
\begin{itemize}
\item By Corollary \ref{cor:zeta-inv-poly}, there is a polynomial $Q_{2,M}\in\mathbf{\Z}[t]$
with $\frac{1}{\zeta^{S_{n}}(2)}=\frac{1}{2}Q_{2,M}\left(n^{-1}\right)+O\left(n^{-M}\right)$.
\item For any fixed $\ell\geq0$, both $(n)_{\ell}$ and $(n)_{\ell}^{-1}$
agree with Laurent polynomials of $n$ with $\mathbf{Q}$-coefficients
up to order $O(n^{-M})$ as $n\to\infty$.
\item By Proposition \ref{prop:Xi-Laurent}, there is a Laurent polynomial
$\Xi_{M}^{*}(W_{h})\in\mathbf{Q}[t,t^{-1}]$ such that $\Xi_{n}(W_{h})=\Xi_{M}^{*}(W_{h})[n]+O(n^{-M})$
as $n\to\infty$.
\end{itemize}
Hence all terms in (\ref{eq:expectation-expression}) can be approximated
by Laurent polynomials of $n$ with rational coefficients to order
$O(n^{-M})$ as $n\to\infty$. This proves Theorem \ref{thm:rational-approx}.
$\square$

\subsection{Proof of Theorems \ref{thm:expected-fixed-points-bounded} and \ref{thm:subgroups}}

Again, we give the proofs of Theorems \ref{thm:expected-fixed-points-bounded}
and \ref{thm:subgroups} when $g=2$. Given a finitely generated subgroup
$J\leq\Gamma=\Gamma_{2}$ let $\chi_{\max}(J)$ be as in (\ref{eq:x_max}).
We can assume $J$ is non-trivial since Theorem \ref{thm:subgroups}
is obvious in this case (as is Theorem \ref{thm:expected-fixed-points-bounded}
when $\gamma$ is the identity). 

Let $\mathcal{R=\mathcal{R}}(\core(J),\chi_{\max}(J))$ be the resolution
of $\core(J)$ defined in Definition \ref{def:finite resolution for compact surfaces},
certified to be a resolution by Theorem \ref{thm:existence-of-resolution}.
Let $\mathcal{R}_{\max}(\gamma)$ (resp.~$\mathcal{R}_{<\max}(\gamma)$)
be the morphisms $h:$ $\core(J)\to W_{h}$ of $\mathcal{R}$ with
$\chi(W_{h})=\chi_{\max}(J)$ (resp.~$\chi(W_{h})<\chi_{\max}(J)$),
so 
\[
\mathcal{R}=\mathcal{R}_{\max}\sqcup\mathcal{R}_{<\max}.
\]
By Theorem \ref{thm:existence-of-resolution} all elements of $\mathcal{R}_{\max}$
are strongly boundary reduced and all elements of $\mathcal{R}$ are
boundary reduced. Repeating the argument of $\S\S$\ref{subsec:Proof-of-Theorem-analytic}
we obtain (\ref{eq:expectation-expression}) again.

Now we observe:
\begin{itemize}
\item By Proposition \ref{prop:Liebeck-Shalev-non-effective} with $b=1$,
$\zeta^{S_{n}}(2)=2+O(n^{-2})$ as $n\to\infty$. 
\item For each $h:\core(J)\to W_{h}\in\mathcal{R}$, the ratio of Pochhammer
symbols satisfies as $n\to\infty$
\begin{equation}
\frac{(n)_{\v(W_{h})}(n)_{\f(W_{h})}}{\prod_{f\in\{a,b,c,d\}}(n)_{\e_{f}(W_{h})}}=n^{\chi\left(W_{h}\right)}+O\left(n^{\chi(W_{h})-1}\right).\label{eq:pochhammer-ratio-decay}
\end{equation}
\item By Proposition \ref{prop:xi-bound-for-strongly-boundary reduced},
for each $h:\core(J)\to W_{h}\in\mathcal{R}_{\max}$ we have $\Xi_{n}\left(W_{h}\right)=2+O_{W_{h}}\left(n^{-1}\right)$.
\item By Proposition \ref{prop:xi-bound-for-boundary reduced}, for each
$h:\core(J)\to W_{h}\in\mathcal{R}_{<\max}$ we have $\Xi_{n}\left(W_{h}\right)=O_{W_{h}}(1)$.
\end{itemize}
Hence from (\ref{eq:expectation-expression}), \vspace{-10bp}

\begin{align*}
\E_{n}[\fix_{J}] & =\frac{1}{\zeta^{S_{n}}(2)}\left[\sum_{h\in\mathcal{R}_{\max}}n^{\chi_{\max}(J)}\left(1+O\left(n^{-1}\right)\right)\cdot\left(2+O\left(n^{-1}\right)\right)+\sum_{h\in\mathcal{R}_{<\max}}O\left(n^{\chi_{\max}(J)-1}\right)\cdot O(1)\right]\\
 & =\frac{1}{2+O\left(n^{-2}\right)}\left[2\cdot\left|\mathcal{R}_{\max}\right|\cdot n^{\chi_{\max}(J)}+O\left(n^{\chi_{\max}(J)-1}\right)\right]\\
 & =\left|\mathcal{R}_{\max}\right|\cdot n^{\chi_{\max}(J)}+O\left(n^{\chi_{\max}(J)-1}\right),
\end{align*}
where all implied constants depend on $J$. By Proposition \ref{prop:resolutions of core surfaces},
$|\mathcal{R}_{\max}|=|\mog(J)|$ which proves Theorem \ref{thm:subgroups}.

Finally, if $J=\langle\gamma\rangle$, and $q$ is maximal such that
$\gamma=\gamma_{0}^{~q}$ for some $\gamma_{0}\in\Gamma$, then Corollary
\ref{cor:resolution of a cyclic subgroup} tells us that $|\mathcal{R}_{\max}|=d(q)$.
This proves Theorem \ref{thm:expected-fixed-points-bounded}. $\square$

\bibliographystyle{alpha}
\bibliography{surface}

\newcommand{\etalchar}[1]{$^{#1}$}
\begin{thebibliography}{EWPS21}

\bibitem[AB83]{AB}
M.~F. Atiyah and R.~Bott.
\newblock The {Y}ang-{M}ills equations over {R}iemann surfaces.
\newblock {\em Philos. Trans. Roy. Soc. London Ser. A}, 308(1505):523--615,
  1983.

\bibitem[ABB{\etalchar{+}}11]{ABBGNRS1}
M.~Abert, N.~Bergeron, I.~Biringer, T.~Gelander, N.~Nikolov, J.~Raimbault, and
  I.~Samet.
\newblock On the growth of {B}etti numbers of locally symmetric spaces.
\newblock {\em C. R. Math. Acad. Sci. Paris}, 349(15-16):831--835, 2011.

\bibitem[ABB{\etalchar{+}}17]{ABBGNRS2}
M.~Abert, N.~Bergeron, I.~Biringer, T.~Gelander, N.~Nikolov, J.~Raimbault, and
  I.~Samet.
\newblock On the growth of {$L^2$}-invariants for sequences of lattices in
  {L}ie groups.
\newblock {\em Ann. of Math. (2)}, 185(3):711--790, 2017.

\bibitem[Bau62]{baumslag1962generalised}
G.~Baumslag.
\newblock On generalised free products.
\newblock {\em Math. Zeit.}, 78(1):423--438, 1962.

\bibitem[BP20]{BakerPetri}
E.~Baker and B.~Petri.
\newblock Statistics of finite degree covers of torus knot complements.
\newblock arXiv preprint arXiv:2005.11956, 2020.

\bibitem[BR10]{Bou-Rabee}
K.~Bou-Rabee.
\newblock Quantifying residual finiteness.
\newblock {\em J. Algebra}, 323(3):729--737, 2010.

\bibitem[BS87]{BirmanSeries}
J.~S. Birman and C.~Series.
\newblock Dehn's algorithm revisited, with applications to simple curves on
  surfaces.
\newblock In {\em Combinatorial group theory and topology ({A}lta, {U}tah,
  1984)}, volume 111 of {\em Ann. of Math. Stud.}, pages 451--478. Princeton
  Univ. Press, 1987.

\bibitem[BS01]{BenjaminiSchramm}
I.~Benjamini and O.~Schramm.
\newblock Recurrence of distributional limits of finite planar graphs.
\newblock {\em Electron. J. Probab.}, 6:no. 23, 13, 2001.

\bibitem[Bus10]{Buser}
P.~Buser.
\newblock {\em Geometry and spectra of compact {R}iemann surfaces}.
\newblock Modern Birkh\"{a}user Classics. Birkh\"{a}user Boston, Ltd., Boston,
  MA, 2010.
\newblock Reprint of the 1992 edition.

\bibitem[C{\'S}06]{CS}
B.~Collins and P.~{\'S}niady.
\newblock Integration with respect to the {H}aar measure on unitary, orthogonal
  and symplectic group.
\newblock {\em Comm. Math. Phys.}, 264(3):773--795, 2006.

\bibitem[CSST10]{Ceccherini-SilbersteinScarabottiTolli}
T.~Ceccherini-Silberstein, F.~Scarabotti, and F.~Tolli.
\newblock {\em Representation theory of the symmetric groups: the
  {O}kounkov-{V}ershik approach, character formulas, and partition algebras}.
\newblock Cambridge University Press, 2010.

\bibitem[Deh12]{Dehn}
M.~Dehn.
\newblock Transformation der {K}urven auf zweiseitigen {F}l\"{a}chen.
\newblock {\em Math. Ann.}, 72(3):413--421, 1912.

\bibitem[Dix69]{dixon1969probability}
John~D Dixon.
\newblock The probability of generating the symmetric group.
\newblock {\em Mathematische Zeitschrift}, 110(3):199--205, 1969.

\bibitem[EWPS21]{West}
D.~Ernst-West, D.~Puder, and M.~Seidel.
\newblock Word measures on ${GL}_n(\mathbb{F}_q)$ and free group algebras.
\newblock arXiv preprint arXiv:2110.11099, 2021.

\bibitem[FRT54]{FRT}
J.~S. Frame, G.~de~B. Robinson, and R.~M. Thrall.
\newblock The hook graphs of the symmetric groups.
\newblock {\em Canadian J. Math.}, 6:316--324, 1954.

\bibitem[Gam06]{gamburd2006poisson}
A.~Gamburd.
\newblock Poisson--{D}irichlet distribution for random {B}elyi surfaces.
\newblock {\em The Annals of Probability}, 34(5):1827--1848, 2006.

\bibitem[Gol84]{Goldman}
W.~M. Goldman.
\newblock The symplectic nature of fundamental groups of surfaces.
\newblock {\em Adv. in Math.}, 54(2):200--225, 1984.

\bibitem[Hat05]{hatcher2005algebraic}
A.~Hatcher.
\newblock {\em Algebraic topology}.
\newblock Cambridge University Press, 2005.

\bibitem[Hem72]{Hempel}
J.~Hempel.
\newblock Residual finiteness of surface groups.
\newblock {\em Proc. Amer. Math. Soc.}, 32:323, 1972.

\bibitem[HP20]{hanany2020word}
Liam Hanany and Doron Puder.
\newblock Word measures on symmetric groups.
\newblock arXiv preprint arXiv:2009.00897, 2020.

\bibitem[Hur02]{hurwitz1902ueber}
A.~Hurwitz.
\newblock Ueber die anzahl der {R}iemann'schen fl{\"a}chen mit gegebenen
  verzweigungspunkten.
\newblock {\em Mathematische Annalen}, 55(1):53--66, 1902.

\bibitem[Lab13]{Labourie}
F.~Labourie.
\newblock {\em Lectures on representations of surface groups}.
\newblock Zurich Lectures in Advanced Mathematics. European Mathematical
  Society (EMS), Z\"{u}rich, 2013.

\bibitem[LLM19]{LLM}
N.~Lazarovich, A.~Levit, and Y.~Minsky.
\newblock Surface groups are flexibly stable.
\newblock {\em Preprint, arXiv:1901.07182}, 2019.

\bibitem[LS04]{LiebeckShalev}
M.~W. Liebeck and A.~Shalev.
\newblock Fuchsian groups, coverings of {R}iemann surfaces, subgroup growth,
  random quotients and random walks.
\newblock {\em J. Algebra}, 276(2):552--601, 2004.

\bibitem[Lul96]{Lulov96}
N.~Lulov.
\newblock {\em Random walks on symmetric groups generated by conjugacy
  classes}.
\newblock PhD thesis, Harvard University, 1996.

\bibitem[Mag21a]{magee2021surface-unitary1}
M.~Magee.
\newblock Random unitary representations of surface groups {I}: Asymptotic
  expansions.
\newblock {\em Comm. in Mathematical Physics, arXiv:2101.00252}, 2021.
\newblock to appear, 53 pages.

\bibitem[Mag21b]{magee2021surface-unitary2}
M.~Magee.
\newblock Random unitary representations of surface groups {II}: The large $n$
  limit.
\newblock preprint arXiv:2101.03224, 2021.

\bibitem[Med78]{Mednyhk}
A.~D. Mednyhk.
\newblock Determination of the number of nonequivalent coverings over a compact
  {R}iemann surface.
\newblock {\em Dokl. Akad. Nauk SSSR}, 239(2):269--271, 1978.

\bibitem[Mir07]{Mirzakhani}
M.~Mirzakhani.
\newblock Simple geodesics and {W}eil-{P}etersson volumes of moduli spaces of
  bordered {R}iemann surfaces.
\newblock {\em Invent. Math.}, 167(1):179--222, 2007.

\bibitem[MNP20]{magee2020random}
M.~Magee, F.~Naud, and D.~Puder.
\newblock A random cover of a compact hyperbolic surface has relative spectral
  gap $\frac{3}{16}-\epsilon$.
\newblock arXiv preprint arXiv:2003.10911, 2020.

\bibitem[MP02]{muller2002character}
T.~W. M{\"u}ller and J.C. Puchta.
\newblock Character theory of symmetric groups and subgroup growth of surface
  groups.
\newblock {\em J. of the London Math. Soc.}, 66(3):623--640, 2002.

\bibitem[MP19a]{MPunitary}
M.~Magee and D.~Puder.
\newblock Matrix group integrals, surfaces, and mapping class groups {I}:
  {$\mathcal{U}(n)$}.
\newblock {\em Invent. Math.}, 218(2):341--411, 2019.

\bibitem[MP19b]{MPorthsymp}
M.~Magee and D.~Puder.
\newblock Matrix group integrals, surfaces, and mapping class groups {II}:
  $\mathrm{O}(n)$ and $\mathrm{Sp}(n)$.
\newblock preprint arXiv:1904.13106, 2019.

\bibitem[MP21a]{MPcore}
M.~Magee and D.~Puder.
\newblock Core surfaces.
\newblock preprint arXiv:2108.00717, 2021.

\bibitem[MP21b]{MPsurfacewords}
M.~{Magee} and D.~{Puder}.
\newblock Surface words are determined by word measures on groups.
\newblock {\em Israel Journal of Mathematics}, 241:749--774, 2021.

\bibitem[Nic94]{nica1994number}
A.~Nica.
\newblock On the number of cycles of given length of a free word in several
  random permutations.
\newblock {\em Random Structures \& Algorithms}, 5(5):703--730, 1994.

\bibitem[PP15]{PP15}
D.~Puder and O.~Parzanchevski.
\newblock Measure preserving words are primitive.
\newblock {\em Journal of the American Mathematical Society}, 28(1):63--97,
  2015.

\bibitem[Sco78]{scott1978subgroups}
P.~Scott.
\newblock Subgroups of surface groups are almost geometric.
\newblock {\em Journal of the London Mathematical Society}, 2(3):555--565,
  1978.

\bibitem[Sta83]{stallings1983topology}
J.~R. Stallings.
\newblock Topology of finite graphs.
\newblock {\em Invent. Math.}, 71(3):551--565, 1983.

\bibitem[Tom67]{Tomita}
M.~Tomita.
\newblock On canonical forms of von {N}eumann algebras.
\newblock In {\em Fifth {F}unctional {A}nalysis {S}ympos. ({T}\^{o}hoku
  {U}niv., {S}endai, 1967) ({J}apanese)}, pages 101--102. Math. Inst.,
  T\^{o}hoku Univ., Sendai, 1967.

\bibitem[VDN92]{VDN}
D.~V. Voiculescu, K.~J. Dykema, and A.~Nica.
\newblock {\em Free random variables}, volume~1 of {\em CRM Monograph Series}.
\newblock American Mathematical Society, Providence, RI, 1992.

\bibitem[VO96]{VershikOkounkov}
A.~Vershik and A.~Okounkov.
\newblock A new approach to representation theory of symmetric groups.
\newblock {\em Selecta Mathematica New Series}, 2(4):581--606, 1996.

\bibitem[Wit91]{Witten1991}
E.~Witten.
\newblock On quantum gauge theories in two dimensions.
\newblock {\em Comm. Math. Phys.}, 141(1):153--209, 1991.

\bibitem[Wri20]{wright2020tour}
Alex Wright.
\newblock A tour through mirzakhani’s work on moduli spaces of riemann
  surfaces.
\newblock {\em Bulletin of the American Mathematical Society}, 57(3):359--408,
  2020.

\bibitem[Zag94]{Zagier}
D.~Zagier.
\newblock Values of zeta functions and their applications.
\newblock In {\em First {E}uropean {C}ongress of {M}athematics, {V}ol. {II}
  ({P}aris, 1992)}, volume 120 of {\em Progr. Math.}, pages 497--512.
  Birkh\"{a}user, Basel, 1994.

\end{thebibliography}
Michael Magee, Department of Mathematical Sciences, Durham University,
Lower Mountjoy, DH1 3LE Durham, United Kingdom

\noindent \texttt{michael.r.magee@durham.ac.uk}\\

\noindent Doron Puder, School of Mathematical Sciences, Tel Aviv University,
Tel Aviv, 6997801, Israel\\
\texttt{doronpuder@gmail.com}
\end{document}